\documentclass[11pt]{article}
\usepackage{amsmath,amsxtra,amssymb,color,latexsym,epsfig,amscd,amsthm,subfigure,fancybox,epsfig}
\usepackage[mathscr]{eucal}
\usepackage{bbm}
\usepackage{float}
\usepackage{graphicx}
\usepackage{enumerate}
\usepackage{enumitem}
\usepackage{caption}
\usepackage{epsfig}
\usepackage{epstopdf}
\usepackage{cases}
\usepackage{hyperref}
\usepackage{cleveref}

\def\para#1{\vskip .4\baselineskip\noindent{\bf #1}}

\newtheorem{thm}{Theorem}[section]

\newtheorem{deff}{Definition}[section]

\newtheorem{lem}{Lemma}[section]
\newtheorem{prop}{Proposition}[section]
\theoremstyle{definition}

\theoremstyle{remark}
\newtheorem{rem}{Remark}
\newtheorem{example}{Example}[section]
\numberwithin{equation}{section}

\newcommand{\eps}{\varepsilon}

\newcommand{\h}{\mathcal{H}}

\newcommand{\A}{\mathcal{A}}

\newcommand{\M}{\mathcal{M}}
\newcommand{\F}{\mathcal{F}}
\newcommand{\X}{\mathcal{X}}
\newcommand{\E}{\mathbb{E}}

\newcommand{\N}{\mathbb{N}}
\newcommand{\PP}{\mathbb{P}}

\newcommand{\R}{\mathbb{R}}

\newcommand{\co}{\bar{\mathrm{co}}}

\numberwithin{equation}{section}

\newcommand{\vo}{\vec 0}

\newcommand{\vx}{\vec x}
\newcommand{\vy}{\vec y}

\newcommand{\vX}{\vec X}
\newcommand{\vU}{\vec U}
\newcommand{\vW}{\vec W}
\newcommand{\vb}{\vec b}
\newcommand{\vh}{\vec h}

\newcommand{\vQ}{\vec Q}
\newcommand{\vM}{\vec M}

\newcommand{\vga}{\boldsymbol{\Gamma}}
\newcommand{\vps}{\boldsymbol{\Psi}}
\newcommand{\vpp}{\boldsymbol{\psi}}
\newcommand{\vta}{\boldsymbol{\tau}}
\newcommand{\vbe}{\boldsymbol{\beta}}

\newcommand{\wdt}{\widetilde}

\renewcommand{\vec}[1]{\mathbf{#1}}
\newcommand{\w}{{\vec w}}

\newcommand{\bed}{\begin{displaymath}}
\newcommand{\eed}{\end{displaymath}}
\newcommand{\bea}{\bed\begin{array}{rl}}
\newcommand{\eea}{\end{array}\eed}
\newcommand{\ad}{&\!\!\!\disp}

\newcommand{\barray}{\begin{array}{ll}}
\newcommand{\earray}{\end{array}}

\def\disp{\displaystyle}
\newcommand{\rr}{{\Bbb R}}
\newcommand{\1}{\boldsymbol{1}}

\def\bar{\overline}
\def\hat{\widehat}
\def\a.s{\text{\;a.s.\;}}
\def\a.e{\text{\;a.e.\;}}

\begin{document}
\title{Stochastic Approximation with Discontinuous Dynamics, Differential Inclusions, and Applications\thanks{This research
was supported in part by the Air Force Office of Scientific Research.}}
\author{Nhu N. Nguyen\thanks{Department of Mathematics, University of Connecticut, Storrs, CT
06269,
nguyen.nhu@uconn.edu} \and George Yin\thanks{Department of Mathematics, University of Connecticut, Storrs, CT
06269,
gyin@uconn.edu.}}
\maketitle

\begin{abstract}
This work develops new results for stochastic approximation algorithms. The emphases are on treating algorithms and limits with discontinuities.  The main ingredients include the use of differential inclusions, set-valued analysis, and non-smooth analysis, and stochastic differential inclusions. Under broad conditions, it is shown that a suitably scaled sequence of the iterates has a differential inclusion limit. In addition,  it is shown for the first time that a centered and scaled sequence of the iterates converges weakly to a stochastic differential inclusion limit.  The results are then used to treat several application examples including Markov decision process, Lasso algorithms, Pegasos algorithms, support vector machine classification, and learning. Some numerical demonstrations are also provided.

\medskip



\medskip
\noindent{\bf Running Title.} Stochastic Approximation with Discontinuity
\end{abstract}
\newpage

\addtocontents{toc}{\setcounter{tocdepth}{2}}
\tableofcontents

\section{Introduction}
This paper examines  stochastic approximation
from new angles. One of
the main motivations stems from the
minimization of a  non-differentiable function
or finding the zeros of a set-valued mapping
corrupted with random disturbances.
In contrast to the existing literature,
this paper focuses on stochastic approximation with discontinuous dynamics and set-valued mappings and
 develops new techniques for analyzing algorithms involving set-valued analysis and stochastic differential inclusions.

Let us begin with a
stochastic approximation algorithm of the form
\begin{equation}\label{eq-int-1}
\vX_{n+1}=\vX_n+a_n\vb_n(\vX_n,\xi_n),
\end{equation}
and the corresponding projection algorithm
\begin{equation}\label{eq-int-2}
\vX_{n+1} = \Pi_H(\vX_n+a_n\vb_n(\vX_n,\xi_n)),
\end{equation}
where $H$ is a
constrain set
and $\Pi_H$ is a projection operator.
  Introduced by Robbins-Monro in \cite{RM51} in 1951,  stochastic approximation algorithms have been  studied extensively with a wide range of applications \cite{B96,BMP90,KC78,KY03,L77,MP87}. In addition to the traditional areas,
  recent applications also include  cooperative dynamics and games \cite{BF12}
   and  multilevel Monte Carlo methods \cite{Fri16}.
When the sequence $\vb_n(\cdot,\xi_n)$ and the associated ``average''
satisfy some smoothness conditions,
 the asymptotic properties of the algorithms are
 relatively well-understood \cite{KC78,KY03,L77}.
We refer to such cases as ``stochastic approximation with continuous dynamics and continuous limits".
When $\vb_n(\cdot,\xi_n)$ is not necessarily continuous but its average
is continuous, the analysis can be found in \cite{K83,KY03, MP87}.
We refer to such cases as ``stochastic approximation with discontinuous dynamics and continuous limits".
In
 \cite{KY03},
$\vb_n(\cdot,\xi_n)$ is continuous while the limits
belong to a set-valued mapping is considered and is referred to as ``stochastic approximation with continuous dynamics but discontinuous limits".
 In applications, sometimes, we need to handle the cases both $\vb_n(\cdot,\xi_n)$ and its average
 are discontinuous functions and/or set-valued mappings. We refer to such a case as ``stochastic approximation with discontinuous dynamics and discontinuous limits".

Many systems and practical problems
 require  optimizing non-smooth functions such as dimension reduction problem in high-dimensional statistics (the $L^1$-norm regularized term) \cite{G15}, support vector machines classification (the hinge loss function), neural networks (the rectified linear unit), collaborative filtering and recommender systems (various types of matrix regularizers) \cite{KW19}, complementarity problems \cite{Che12}, compressed modes in physics, the partial consensus problem \cite{HH19}, among others.
Because the objective functions are not continuously differentiable, the gradient-based methods are often replaced by subgradient-based counterparts.
As a result, discontinuous dynamics and set-valued mappings are ubiquitous in the optimization problems.
There are also numerous problems and algorithms in control engineering, economics, and operations research that require the treatment of
 discontinuous dynamics and/or set-valued mappings; see
learning algorithms in Markov decision processes \cite{PL12},
 algorithms in  approachability theory and the study of fictitious play in game theory \cite{BHS05,BHS12}, etc.

To proceed, let us consider an important class of algorithms, namely, adaptive filtering arising from signal processing and control engineering among others.
 Adaptive filtering problems
can be described as follows. Let $\varphi_n\in \rr^d$ and $y_n\in \rr$ be
measured output and reference signals, respectively.
Assuming
that the sequence $\{(y_n, \varphi_n)\}$ is stationary, we adjust a system
parameter $\theta\in \rr^d$ adaptively so that the weighted output $\theta^\top \varphi_n$ best
matches the reference signal $y_n$ in the sense that a cost function
is minimized. This class of algorithms is important; it has drawn a considerable attention in the literature of both probability and engineering; see \cite{BMP90,EM,KY03,WS85} and numerous references therein.
If the cost function $L(\theta)$ is ``mean square deviation", i.e.,
$L(\theta)= \E| y_n - \theta^\top \varphi_n|^2$, then the algorithm is given by
\begin{equation*}
	\theta_{n+1}= \theta_n + a_n \varphi_n (y_n - \varphi_n^\top \theta_n).\end{equation*}
If the cost function is  $L(\theta)= \E| y_n - \theta^\top \varphi_n|$, then the algorithm is given as
\begin{equation}\label{eq-adf-1}
\theta_{n+1}=\theta_n+a_n\varphi_n\text{sign}(y_n-\varphi^\top_n \theta_n),
\end{equation}
where $\text{sign}(y)=\1_{\{y>0\}}-\1_{\{y<0\}}$ is the sign function (see \cite{YKC03}).
The objective function in \eqref{eq-adf-1} is  non-smooth and the dynamic system in  the algorithm is  discontinuous. Although the asymptotic behavior of the algorithm can be studied as in \cite{YKC03}, using the results and the techniques of this paper, we can relax the conditions used,  characterize the limit dynamical system as a differential inclusion of the form
$$\dot \theta (t)\in  \mathcal K[\mathrm{sign}]\Big(\int \big(y - \varphi^\top\theta(t) \big) \nu(dy\times d \varphi)\Big),$$
 where $\mathcal K[\mathrm{sign}](\cdot)$ is the Krasovskii operator of the $\mathrm{sign}$ function (see Appendix \ref{sec:ode}) and
 $\nu(\cdot)$ is the
 distribution of the sequence $\{(y_n,\varphi_n)\}$.
 Moreover, the rate of convergence of the algorithms can be also obtained using stochastic differential inclusions.
In addition to the above sign-error algorithms, one can also study sign-regressor and sign-sign algorithms, all of which contain discontinuity.
%

From another view point, randomness can affect
samplings, mini-batching computations, partial observations, noisy measurements, and many other sources. As was mentioned, various functions involved in applications could possibly be non-smooth or even not continuous.
Thus, it is  necessary to study stochastic approximation algorithms \eqref{eq-int-1} and \eqref{eq-int-2} with both  $\vb_n(\cdot,\cdot)$ and its averages being
discontinuous functions and/or
set-valued mappings.

With the motivations coming from  applications,
this paper formulates the problem by using
a general and unified setting, introduces new techniques, proves
convergence under
mild conditions, and establishes rates of convergence of stochastic approximations with possibly discontinuous dynamics and discontinuous limits.
Both constrained, unconstrained, and biased algorithms are considered.
To be more specific,
using appropriate piecewise linear and piecewise constant interpolations,
 we prove the boundedness and equicontinuity
  of the sequences in a functional space.
The compactness  enables us
to extract a convergent subsequence.
Most existing works in the literature
 use continuity for either the dynamics or the limit systems.
 If the dynamics are not continuous but the limit systems have enough regularity, Kushner in \cite{K83} used an ``averaging method" to handle this problem under some conditions on the existence of certain Lyapunov functions. In contrast, M\'etivier and Priouret in \cite{MP87} used probabilistic approach by averaging out the noise
 with respect to the invariant measure.
 To analyze algorithms with both the dynamics and limits being discontinuous, we need a new approach.
 In this paper,
we use ordinary differential equations (ODEs) with discontinuous right-hand sides, differential inclusions,  set-valued dynamical systems, and convex analysis to
characterize the
asymptotic behavior of the algorithms.
To obtain the stability,  we use results of stability for differential inclusions
together with novel concepts and techniques from non-smooth analysis.
We also examine
biased stochastic approximation using continuation of chain recurrent sets in set-valued dynamic systems.
In addition, the rates of convergence  are  obtained by using the theory of stochastic differential inclusions and the newly developed theory of variational analysis.

\begin{rem}\label{ref-conv}{\rm Reference Label Convention. Throughout the paper, we  use several sets of assumptions. To facilitate the reading, we shall use the following conventions. Conditions headed by {\bf (A)} corresponds to standard assumptions; conditions headed by {\bf (K)}, {\bf (G)}, and {\bf (P)} are assumptions  involving {\bf Krasovskii} operator, {\bf general} set-valued mapping, and {\bf projection}, respectively;  conditions headed by {\bf (KS)}, {\bf (GS)}, and {\bf (PS)} are stability assumptions corresponding to that of (K), (G), and (P), respectively;  conditions headed by {\bf (R)} are  for  the rates of convergence study.}\end{rem}

To proceed, we  summarize our results as follows.
The first
 convergence results are obtained  in Theorem \ref{mth-1} and Theorem \ref{mth-3}.
The boundedness and equicontinuity of appropriate interpolated sequences
enable
us to extract a convergent subsequence.
By examining the closure of the solutions of differential inclusions, we are able to  characterize the limit systems by differential inclusions.
The asymptotic behavior is then examined by the set of chain recurrent points of the limit differential inclusions; the stability is studied under assumptions on the stability of differential inclusions in the sense of Lyapunov.
Our first convergence theorem
establishes that
the discontinuous components can be averaged out
with the use of
the Krasovskii operator of some vector-valued function.
Next, the Krasovskii operator is replaced by general set-valued mappings.
One of the main difficulties in this case is that we have to obtain some ``nice" properties (the same as that of Krasovskii operator) for set-valued mappings having closed graph, for which we need to use set-valued analysis and convex analysis;
see Proposition \ref{prop63}.

 To continue, we investigate the global stability of the limit differential inclusions, and then establish the
convergence of stochastic approximations to the desired points  by using Assumption \ref{S2} or Assumption \ref{G2} in Theorem \ref{mth-2} and Theorem \ref{mth-4}.
These conditions are similar to that
of the existence of  Lyapunov functions in the stability of ODEs.
However, because of the absence of smoothness conditions, some quantities need to be redefined, for example, $\mathcal U$-generalized Lyapunov function is used instead of Lyapunov function.
In contrast to the ODEs, studying the asymptotic stability of differential inclusions appears to be more difficult.
 With the help from non-smooth analysis and novel results of stability of differential inclusions,
our approach is shown to be
more effective than existing results in the literature; see Section \ref{sec:app4}.

Next, projection algorithms  are examined in Theorem \ref{mth-5}, in which, the projection space $H$ is assumed to be compact and convex (Assumption \ref{H1}). Assumption \ref{H2} provides sufficient conditions for globally asymptotic stability for algorithms with projections.
The
results in this case have similarity to that of the unconstrained algorithms.

To continue, we study biased stochastic approximation, and demonstrate that
the convergence (to $0$) in Assumption \ref{A}(v), and/or Assumption \ref{A}(iii) and Assumption \ref{A}(iv), can be replaced by a neighborhood (of $0$) with radius $\eta$.
Such an idea also stems from the so-called worst case analysis or robustness in handling systems arising in control theory.
We prove that the distance of the sequence of iterates and the set of chain recurrent points of the limit differential inclusions is bounded above by a function $\phi(\eta)$ of $\eta$ satisfying that $\phi(\eta)\to 0$ as $\eta\to0$.
The main idea  is to modify, combine, and extend
our methods in characterizing limit for unbiased case and the continuation of chain recurrent set of differential inclusions developed
in \cite{BHS12}.

Under assumptions on regularities of set-valued mappings, the rate of convergence for stochastic approximations is obtained in Theorem \ref{mth-7}.
Since this case is relatively complex, we consider a simple version of the algorithm so as to get the main ideas across without undue notational complexity;
more complex algorithms can be handled.
Again, the main difficulties lies in characterizing the limit. In lieu of stochastic differential equations,
stochastic differential inclusions
and variational analysis [continuity and $T$-differentiability (see Definition \ref{54-def-14})] are used to derive the desired result.

We demonstrate the utility of our results
by examining several
problems
including the multistage decision making models with partial observations in Markov decision process; and stochastic sub-gradient descent algorithms in minimizing non-differentiable loss functions, $L^1$-norm (Lasso) regularized (or penalized) loss minimization in reducing high-dimensional statistics, robust regression, and  Pegosos algorithm in support vector machine (SVM) classification in machine learning.
We also demonstrate that
 certain convergence results can be obtained by using our results while that cannot be done (or more difficult to obtain) by using the existing results in the literature.
 While the study of Lasso and SVM algorithms may have been around for quite some time,
 the treatment of the nonsmooth and non-continuous cases
and
 the characterization of the limit of the un-scaled and scaled dynamics using differential inclusions and stochastic differential inclusions have not been considered in the past.

\para{Related works and our
contributions.}
To proceed,
we
 highlight our contributions and novelties of the paper in contrast to the existing literature.
\begin{itemize}
	\item
	Although the algorithms involving discontinuous dynamics and set-valued mappings were considered in \cite{KC78}, continuity in an appropriate sense of set-valued  mappings was needed. The continuity, however, may  fail in applications.
	Except \cite{BHS05},
there has been
no
general approach in the literature
for studying convergence of stochastic approximation schemes involving set-valued mappings
without continuity.
Although both papers dealing with differential inclusions,
the setup and results of
 the current paper
 are different
  than  that of \cite{BHS05}.
		Using our approach,
it is possible to recover the setting in \cite{BHS05}; see
	 Remark \ref{rem-compare}.
	  Moreover, the limit processes in \cite{BHS05} were shown to be perturbed solutions of
differential inclusions, whereas in the current paper,
we  characterize the limit processes by the solutions (rather than perturbed solutions) of the limit differential inclusions.
Our convergence analysis is done partially by examining the closure of the set of solutions of a family of differential inclusions for general set-valued mappings,  which is a crucial point in the development.
	Both constrained and unconstrained algorithms are also considered in this paper.
	
    \item In addition, to prove the convergence to the equilibrium point, the stability of differential inclusions corresponding to stochastic approximation schemes is carefully investigated
    using a Lyapunov functional method that is novel and not considered in the existing literature of stochastic approximation.
   To be more specific, we use a ${\mathcal U}$-generalized Lyapunov functional.
    Our approaches and results
    appear to be more effective
    and easily applicable (see examples in Section \ref{sec:app}). The idea behind this approach is that one can ignore some ``less important" points, which
    do not affect the stability of the dynamics.

   \item We consider biased stochastic approximation with discontinuous dynamics and set-valued mappings.
       Although biased stochastic approximation counterpart with smooth dynamics was dealt with in \cite{TD17},
       to the best of our knowledge, this is the first time
         biased stochastic approximation in conjunction with set-valued mappings without continuity is treated.

    \item
    In addition, this work provides
     a rate of convergence study with discontinuities and set-valued mappings.
      Stochastic differential inclusions are used
      for the first time to ascertain the rate of convergence of stochastic approximation.

    \item For applications,
    we provide a unified
    framework and new approaches to
   analyze convergence, rates of convergence, robustness for stochastic and non-smooth optimization problems and/or algorithms involving discontinuous dynamics and set-valued mappings.
   The
   applications considered in the paper include
   algorithms in machine learning and  Markov decision process.
  For applications to machine learning algorithms,
  we provide new insights in analyzing these algorithms by characterizing the limit behavior and rates of convergence using the dynamic systems generated from differential inclusions and stochastic differential inclusions.
   In the machine learning literature to date, almost all
   existing analysis is based purely on constructing some kind of ``contraction estimates" in expectation;  it seems that there has been no  unified framework for
	analyzing stochastic approximation algorithms with discontinuous dynamics and set-valued mappings.
   We also
   fill in the gap for studying convergence of algorithms with non-smooth loss functions.
    Treating Markov decision process, we demonstrate how to apply our results for multistage decision making  with partial observations.
\end{itemize}

\para{Outline of the paper.} The rest of paper is
arranged as follows. Section \ref{sec:conv} obtains convergence of stochastic approximation algorithms with the emphasis on  discontinuity and set-valued mappings.
Section \ref{sec:rate} ascertains rates of convergence with the use of stochastic differential inclusion limits.
Section \ref{sec:app} examines a number of
applications together with numerical results.
Section \ref{sec:con} summarizes our findings and provides further remarks.
Finally,
mathematical background in ODEs with discontinuous right-hand sides, differential inclusions,  non-smooth analysis,  set-valued dynamic systems and analysis, and stochastic differential inclusions are summarized in Appendix \ref{sec:pre}.

\section{Convergence}\label{sec:conv}
Denote by  $\R^d$ the  $d$-dimensional Euclidean space with
the usual Euclidean norm $|\cdot|$, and let
 $(\Omega, \F,\{\F_t\}_{t\geq 0},\PP)$ be a complete
 filtered probability space satisfying the usual conditions.
Consider the following general stochastic approximation algorithm
\begin{equation}\label{eq-algo}
\vX_{n+1}=\vX_n+a_n\vb_n(\vX_n,\xi_n)+a_n\vh(\vX_n,\zeta_n)
+a_n\vh_0(\wdt \zeta_n)+a_n\vbe_n,
\end{equation}
and the associated projection
algorithm
\begin{equation}\label{eq-proj}
\begin{cases}
\tilde \vX_{n+1}=\vX_n+a_n\vb_n(\vX_n,\xi_n)
+a_n\vh(\vX_n,\zeta_n)+a_n\vh_0(\wdt \zeta_n)+a_n\vbe_n,\\
\vX_{n+1}=\Pi_{H}(\tilde \vX_{n+1}),
\end{cases}
\end{equation}
where $\Pi_H$ is the projection operator (orthogonal projection into the set $H$),
and $H$ is a subset of $\R^d$.
 The $\{a_n\}$ is a sequence of step sizes
 (a sequence of positive real numbers) satisfying
 $a_n\to 0$ and $\sum_{n=1}^\infty a_n=\infty$.
The sequences
$\{\xi_n\}$, $\{\zeta_n\}$, and $\{\wdt\zeta_n\}$  noise processes that are correlated in time but independent of each other,
and $\{\vbe_n\}$
represents the bias;
see \cite{KC78,KY03}.
In the literature,
 $\vbe_n$ is often formulated as a diminishing bias so that it
tends to 0 w.p.1.
However, there are cases that one has to face
asymptotically non-zero bias in the sense
$\lim_{n\to\infty}\|\vbe_n\|>0$.

Motivated by many applications, the  functions
$\vb_n(\cdot,\cdot)$ are allowed to be discontinuous and belong to
a set-valued mapping, which can be used to represent sub-gradient
of  non-differentiable components in the loss function, whereas $\vh(\cdot,\cdot)$ is a continuous function (in $\vx$) representing the gradient of the smooth parts in the loss function.
The discontinuity of $\vb_n(\cdot,\cdot)$ and/or set-valued
mappings
appear frequently in applications.
Dealing with such functions and mappings is one of the main objectives of this paper.

\para{Notation.}
 Similar to \cite{KC78,KY03}, define $t_0=0$ and for $n\geq 1$, $t_n:=\sum_{i=0}^{n-1}a_i$, $m(t):=\max\{n:t_n\leq t\}$ if $t\geq 0$ and $m(t)=0$ if $t<0$; and  define the piecewise constant interpolation $\bar{\vX}^0(t)$
 and the piecewise linear interpolation $\vX^0(t)$ of $\vX_n$ with interpolation intervals $\{a_n\}$ as
$$\bar{\vX}^0(t):=\vX_n\text{ in }[t_n,t_{n+1}),$$
$$
\vX^0(t_n):=\vX_n\text{ and }
\vX^0(t):=\frac{t_{n+1}-t}{a_n}\vX_n+\frac{t-t_n}{a_n}\vX_{n+1}\text{ in }(t_n,t_{n+1}),
$$
respectively, and
 define the shift sequence $\vX^n(\cdot)$ on $(-\infty,\infty)$ as
$$
\vX^n(t):=
\begin{cases}
\vX^0(t+t_n),\text{ if }t\geq -t_n,\\
\vX_0\text{ if }t\leq -t_n.
\end{cases}
$$
For two sets $S$, $S_1$, and either a set-valued or a vector-valued mapping $F$, and a real number $k$, we define
$
S+S_1:=\{\vec x+\vec y: \vec x\in S,\vec y\in S_1\},
$
and
$
F(S):=\cup_{\vec x\in S}F(\vx),
$
and
$
k S:=\{k \vx:\vx\in S\}.
$
Throughout the paper, $B$ denotes the
unit open ball $B=\{\vx\in\R^d:|\vx|<1\}$  and $\bar {B}$ is its closure; ``\text{co}" is the convex hull and $``\bar{\text{co}}"$ is the
convex closure; $2^{\R^d}$ is the collection of all subsets of $\R^d$.
To analyze the convergence, we present the following standard assumptions first.
[Recall the reference label conventions in Remark \ref{ref-conv}.]

\begin{enumerate}[label=\textbf{(A)}]
\item\label{A}
\begin{itemize}[leftmargin=*,align=left]
\item[]
{\rm (i)}
$\vh(\cdot,\zeta)$ is continuous in $\vx$, uniformly in $\zeta$ on bounded sets of $\vx$.

\item[]{\rm(ii)}
Either
$\vh(\cdot,\cdot)$ is a bounded measurable function
or
there are non-negative measurable functions $ g_1(\cdot)$ of $\vx$,
and $g_2(\cdot)$ and $g_3(\cdot)$ of $\zeta$ such that $g_1(\cdot)$ is bounded on bounded sets (of $\vx$) and
\begin{equation}\label{eq-a21}
|\vh(\vx,\zeta)|\leq g_1(\vx)g_2(\zeta)+g_3(\zeta),
\end{equation}
and for each $\eps>0$,
\begin{equation}\label{eq-a22}
\lim_{\Delta\to 0}\lim_{n\to\infty}\PP\left\{\sup_{j\geq n}\max_{t\leq \Delta}\sum_{i=m(j\Delta)}^{m(j\Delta+t)-1}a_i[g_2(\zeta_i)+g_3(\zeta_i)]\geq \eps\right\}=0.
\end{equation}

\item[]{\rm (iii)}
There exists a continuous function $\bar \vh(\cdot)$ such that for some $T>0$, each $\eps>0$, and each $\vx$,
\begin{equation}\label{a3eq}
\lim_{n\to\infty}\PP\left\{\sup_{j\geq n}\max_{t\leq T}\Big|\sum_{i=m(jT)}^{m(jT+t)-1}a_i\left(\vh(\vx,\zeta_i)-\bar \vh(\vx)\right)\Big|\geq \eps\right\}=0.
\end{equation}

\item[]{\rm(iv)}
The
$\{\xi_n\}$, $\{\zeta_n\}$, $\{\wdt\zeta_n\}$ are
sequences of independent and exogenous
noises,
and
the function $\vh_0(\cdot)$ is measurable such that for some $T>0$ and each $\eps>0$,
\begin{equation}\label{a4eq}
\lim_{n\to\infty}\PP\left\{\sup_{j\geq n}\max_{t\leq T}\Big|\sum_{i=m(jT)}^{m(jT+t)-1}a_i\vh_0(\wdt \zeta_i)\Big|
\geq \eps\right\}=0.
\end{equation}
By exogenous noises, we mean that  the distribution of $\{\xi_i, i>n\}$ conditioned on $\{\xi_i, \vX_n: i\leq n\}$ is
the same as that of $\{\xi_i, i>n\}$
conditioned on $\{\xi_i: i\leq n\}$ and similar assumptions for $\zeta_n$ and $\wdt\zeta_n$.

\item[]{\rm(v)}
The $\{\vbe_n\}$ is a  sequence of
bounded random variables satisfying $|\vbe_n|\to 0$ w.p.1.
\end{itemize}
\end{enumerate}

\begin{rem}\label{rem-00}
Assumption \ref{A} together with
 the boundedness of the iterates $\{\vX_n\}$
 or a projection algorithm
 (e.g., Assumption \ref{H1} given later)
presents
broad conditions,
which guarantee the boundedness and equicontinuity of $\{\vX^n(\cdot)\}$.
Sufficient conditions
 guaranteeing the boundedness can be provided;
 see \cite[Section 4.7 and Theorem 4.7.4]{KC78} (see also \cite{L77}) or using a projection algorithm \cite{KC78,KY03}.
To handle non-exogenous noise, the reader can consult
\cite[Section 6.6]{KY03} for the treatment of state-dependent noise.
In this paper, for simplicity, we will not deal with such cases.
The noise processes $\{\xi_n\}$, $\{\zeta_n\}$, $\{\wdt\zeta_n\}$ take values in some measurable spaces. However, due to we do not assume any regularity of functions $\vb$, $\vh$, $\vh_0$ on these variables, we often do not specify these spaces.
Moreover, one can combine $\vec h_0(\wdt\zeta_n)$ and $\vbe_n$ in mathematically treating, however, due to their motivations in application (one presents the noise and the other presents the bias), we still keep these two different terms in the setting.
Assumption \ref{A}(v) (as well as \eqref{eq-a22}, \eqref{a3eq}, \eqref{a4eq}) can be relaxed, which will be considered later.
\end{rem}

%
%

\para{Convergence.}
Now, we state our main convergence results; some preliminary results and
 concepts are relegated to Appendix \ref{sec:pre}.
  We use $\mathcal C^d(-\infty,\infty)$ to denote the space of $\R^d$-valued continuous functions defined on $(-\infty,\infty)$, and $D(-\infty,\infty)$  and $D[0,\infty)$ to denote the spaces
 of
 real-valued
 functions defined on $(-\infty,\infty)$ and $[0,\infty)$, respectively, which are
 right continuous and have left limits,  endowed with the Skorohod topology. We use $D^d(-\infty,\infty)$ (resp., $D^d[0,\infty)$) to denote the corresponding $D$ spaces taking values in $\R^d$.
 The convergence of sequence of functions in $\mathcal C^d(-\infty,\infty)$ or $D^d(-\infty,\infty)$ (resp., $D^d[0,\infty)$) is in the sense of weak topology (uniform convergence on bounded intervals).

	As was mentioned, the  functions
	$\vb_n(\cdot,\cdot)$ are possibly discontinuous and belong to
	some set-valued mapping so that they can be used to represent sub-gradients
	of non-differentiable components in the loss function.
	To illustrate,
we first consider the case that this set-valued mapping can be expressed as the Krasovskii operator of some vector-valued function. [For example, sub-gradient of $|\cdot|$ can be expressed as the Krasovskii operator of the $\text{sign}(\cdot)$ function.]
In fact, we allow perturbations of this set-valued mapping, which
is presented in Assumption \ref{S1}.

\begin{enumerate}[label=\textbf{(K)}]
	\item\label{S1} There are a locally bounded function $\bar {\vb}(\cdot)$
	and a sequence of (positive real-valued) continuous (in $\vx$, uniformly in $\xi$)
	functions $\{m_n(\vx,\xi)\}$
	such that
	$\forall n$, $\vx$, $\xi$,
	$$
	\vb_n(\vx,\xi)\in \mathcal K[\bar\vb](\vx)+m_n(\vx,\xi)\bar B,
	$$
	and for some $T>0$, each $\eps>0$, and each $\vx$,
	\begin{equation*}
	\lim_{n\to\infty}\PP\left\{\sup_{j\geq n}\max_{t\leq T}\Big|\sum_{i=m(jT)}^{m(jT+t)-1}a_im_i(\vx,\xi_i)\Big|\geq \eps\right\}=0.
	\end{equation*}
	
	In the above, $\mathcal K[\bar \vb]$ is the Krasovskii operator of $\bar \vb$, i.e., $\mathcal K[\bar\vb]:\R^d\to 2^{\R^d}$ is defined by
	$$
	\mathcal K[\bar {\vb}](\vy):=\cap_{\delta>0}\co\;\bar {\vb}(B(\vy,\delta)),
	$$
	where
	$B(\vy,\delta)$ is the open ball in $\R^d$ with center $\vy$ and radius $\delta$.
	More details on the
	Krasovskii operator and related results are provided in Section \ref{sec:ode}.
\end{enumerate}

\begin{thm}\label{mth-1}
Consider
algorithm \eqref{eq-algo}.  Assume that {\rm\ref{A}} 
and {\rm\ref{S1}} hold and that $\{\vX_n\}$ is bounded w.p.1.
\begin{itemize}
\item
Then there is a null set $\Omega_0$ such that $\forall \omega\notin\Omega_0$, $\{\vX^n(\cdot)\}$ is bounded and equicontinuous on bounded intervals.

\item
Let $\vX(\cdot)$ be the limit of a convergent subsequence of $\{\vX^n(\cdot)\}$.
 Then
$\vX(t)$ is a Krasovskii solution of
\begin{equation}\label{eq-bh}
\dot \vX(t)=\bar {\vb}(\vX(t))+\bar \vh(\vX(t)),
\end{equation}
that is, $\vX(\cdot)$ is a solution of the differential inclusion $($see Section {\rm\ref{sec:ode}} for detailed definitions$)$
\begin{equation}\label{eq-kbh}
\dot \vX(t)\in \mathcal K[\bar {\vb}+\bar \vh](\vX(t)).
\end{equation}
\item The limit set of $\vX(\cdot)$ is internally chain transitive $($with respect to \eqref{eq-kbh}$)$
and
the limit points of $\{\vX_n\}$ are contained in $\mathcal R$, the set of chain-recurrent points
of \eqref{eq-kbh} $($see Appendix \ref{sec:dyn} for the definitions$)$.
\item Moreover, let $\Lambda$ be a locally asymptotically stable set $($in the sense of Lyapunov$)$ of all Krasovskii solutions of \eqref{eq-bh} and  $DA(\Lambda)$ be its domain of attraction.
If
$\{\vX_n\}$ visits the compact subset of $DA(\Lambda)$ infinitely often  with probability 1 $($resp., with probability $\geq \rho)$,
then $\vX_n\to \Lambda$ when $n\to\infty$ with probability 1 $($resp., with probability $\geq \rho)$.
\end{itemize}
\end{thm}

\begin{proof}[Proof of Theorem \ref{mth-1}]
To help the reading, we divide the proof into four parts.

{\bf Part 1: Boundedness and Equi-continuity.}
The argument in proving  $\{\vX^n(\cdot)\}$ being bounded and equicontinuous is similar to \cite[Proof of Theorem 2.4.1 and Theorem 2.4.2]{KC78} or \cite[Proof of Theorem 6.1.1]{KY03}.
Hence, we only outline the main points and highlight the differences.
Let $\mathcal H_0$ be a countable dense subset of $\R^d$
 and $\Omega_0$ be the null set that contains all paths, in which $\{\vX_n\}$ is unbounded and the exceptional sets in Assumption {\rm \ref{A}(ii)-(v)},
 {\rm\ref{S1}},
 union over $\mathcal H_0$. In the assumptions, the null or exceptional sets
  are the sets,
   in which the boundedness or convergence does not hold.
   For example, the exceptional set (at $\vx$) in {\rm \ref{A}(iii)} is the set of all $\omega$, in which
$$
\limsup_{n\to\infty}\max_{t\leq T}\Big|\sum_{i=m(nT)}^{m(nT+t)-1}a_i\left(\vh(\vx,\zeta_i)-\bar \vh(\vx)\right)\Big|\neq 0.
$$
We refer to \cite[Proof of Lemma 2.2.1]{KC78} for the proof of the above exceptional sets being null sets.
Since $\mathcal H_0$ is countable, $\Omega_0$ is still a null set.
Now, we  work with a fixed $\omega\notin\Omega_0$.
We write $\vX^n(\cdot)$ as
\begin{equation}\label{eq-4.00}
\vX^n(t)
=\vX_n+\!\int_0^t \vb_n(\bar{\vX}^0(t_n+s),\bar \xi^0(t_n+s))ds+\!\int_0^t \vh(\bar{\vX}^0(t_n+s),\bar \xi^0(t_n+s))ds
+\vga^n(t)+\vps^n(t),
\end{equation}
if $t\geq -t_n$,  otherwise $\vX^n(t)=\vX_0$,
where
$\vga^n(t)$ and $\vps^n(t)$ are the piecewise linear interpolations of $\sum_{i=0}^{n-1}a_i\vbe_i$ and $\sum_{i=0}^{n-1}a_i\vh_0(\xi_i)$, respectively. That is,
$$
\begin{aligned}
&\vga^0(t_n)=\sum_{i=0}^{n-1} a_i\vbe_i\;;\;\vga^0(t)=\frac{t_{n+1}-t}{a_n}\vga^0(t_n)+
\frac{t-t_n}{a_n}\vga^0(t_{n+1})\text{ for } \ t\in (t_n,t_{n+1}),\\
&\vga^n(t)=\begin{cases}\vga^0(t+t_n)-\vga^0(t_n) \ \text{ if }t\geq -t_n,\\
-\vga^0(t_n)\ \text{ if }t\leq -t_n,
\end{cases}\\
&\vps^0(t_n)=\sum_{i=0}^{n-1} a_i\vh_0(\xi_i)\;;\;\vps^0(t)=\frac{t_{n+1}-t}{a_n}\vps^0(t_n)
+\frac{t-t_n}{a_n}\vps^0(t_{n+1})\text{ for } \ t\in(t_n,t_{n+1}),\\
&\vps^n(t)=\begin{cases}\vps^0(t+t_n)-\vps^0(t_n) \ \text{ if }t\geq -t_n,\\
-\vps^0(t_n)\ \text{ if }t\leq -t_n;
\end{cases}
\end{aligned}
$$
and $\bar\vbe^0(\cdot)$ and $\bar \xi^0(\cdot)$ are the piecewise constant interpolations of $\{\vbe_n\}$  and $\{\xi_n\}$, i.e.,
$
\bar\vbe^0(t)=\vbe_n$ and $\bar\xi^0(t)=\xi_n$  for $ t\in [t_n,t_{n+1})
$.
 Note that we have three different sequences of noise processes, $\{\xi_n\}$, $\{\zeta_n\}$, and $\{\wdt\zeta_n\}$, but we write them as $\{\xi_n\}$ (and $\bar\xi^0(t)$ for the interpolations) to simplify the notation.
For simplicity again, we will always write the algorithm as \eqref{eq-4.00}, whether $t\geq -t_n$ or $t\leq -t_n$ with the understanding that $\vX^n(t)=\vX_0$ if $t\leq -t_n$.

First, by {\rm \ref{A}(iv)} and {\rm \ref {A}(v)},
$\{\vga^n(\cdot)$ and $\vps^n(\cdot)\}$ are equicontinuous and bounded,
and any convergent subsequence converges
uniformly to a zero process on bounded intervals (see e.g., \cite[Lemma 2.2.1]{KC78}).
Second,
note that $\{\vb_n(\cdot,\cdot)\}$ is (uniformly in
 the variable $\xi$) bounded due to Assumption {\rm\ref{S1}} and the boundedness of $\{\vX_n\}$.
By {\rm \ref{A}(ii)}, if $\vh(\cdot,\cdot)$ is uniformly bounded, combining with boundedness of $\{\vb_n(\cdot,\cdot)\}$, $\{\vX^n(\cdot)\}$ is equicontinuous.
Otherwise, by \eqref{eq-a21} and \eqref{eq-a22},
 we obtain
$$
\int_t^{t+s}\big|\vh(\bar{\vX}^0(r),\bar\xi^0(r))\big|dr\leq K\int_t^{t+s}g_2(\bar \xi^0(r))dr+\int_t^{t+s}g_3(\bar\xi^0(r))dr,
$$
where $K$ is some finite constant; such a $K$ always exists due to the boundedness of $\{\vX_n\}$ and local boundedness of $g_1(\cdot)$ in {\rm\ref{A}(ii)}.
Thus, by \eqref{eq-a22}, we get that
 $\int_t^{t+s} |\vh(\bar{\vX}^0(r),\bar\xi^0(r))|dr$ is uniformly continuous in $t,s$ in $[0,\infty)$.
Therefore, it is easy to show that $\vX^0(\cdot)$ is uniformly continuous, so $\{\vX^n(\cdot)\}$ is equicontinuous.
As a consequence, we obtain boundedness and equicontinuity
 of $\{\vX^n(\cdot)\}.$

{\bf Part 2: Characterize the limit.}
Take a convergent subsequence of $\{\vX^n(\cdot)\}$ and still denote it by $\{\vX^n(\cdot)\}$ for simplicity of notation and denote its limit by $\vX(\cdot)$.
From the integral form \eqref{eq-4.00}, we have that
\begin{equation}
\begin{aligned}
\vX^n(t)
=\vX_n&+\int_0^t \vb_n(\bar{\vX}^0(t_n+s),\bar \xi^0(t_n+s))ds+\int_0^t \bar \vh(\vX(s))ds\\
&+
\int_0^t \left[\vh(\bar{\vX}^0(t_n+s),\bar \xi^0(t_n+s))-\bar \vh(\vX(s))\right]ds
+\vga^n(t)+\vps^n(t).
\end{aligned}
\end{equation}
Hence, we obtain that
\begin{equation}\label{eq-9.41}
\vQ^n(t)=\vQ^n(0)+\int_0^t \vb_n(\bar{\vX}^0(t_n+s),\bar\xi^0(t_n+s))ds+\int_0^t\bar \vh(\vX(s))ds,
\end{equation}
where
$$
\vQ^n(t):=\vX^n(t)-\vga^n(t)-\vps^n(t)-
\int_0^t\left[\vh(\bar{\vX}^0(t_n+s),\bar \xi^0(t_n+s))-\bar \vh(\vX(s))\right]ds.
$$
Because of Assumption {\rm\ref{S1}}, we get
\begin{equation}\label{eq-9.42}
\begin{aligned}
\vb_n(\bar{\vX}^0(t_n+t),\bar \xi^0(t_n+t))&\in \mathcal K[\bar {\vb}]\big(\bar{\vX}^0(t_n+t)\big)+m_n(\bar{\vX}^0(t_n+t),\bar\xi^0(t_n+t)\bar B\\
&=\mathcal K[\bar {\vb}]\big(\vX(t)+\vec p_n(t)\big)+m_n(\bar{\vX}^0(t_n+t),\bar\xi^0(t_n+t)\bar B,
\end{aligned}
\end{equation}
where
$m_n(\vx,\xi)$ is as in Assumption {\rm\ref{S1}} and
$
\vec p_n(t):=\bar{\vX}^0(t_n+t)-\vX(t).
$

Next, we  prove $\vec p_n(t)$ converges to $\vec 0$ and $\vQ^n(t)$ converges to $\vX(t)$ uniformly on bounded $t$-intervals.
First, 
it is easy to see that
$\bar{\vX}^0(t_n+\cdot)-\vX(\cdot)$ converges to $\vo$ uniformly on bounded intervals,
which leads to that $\{\vec p_n(\cdot)\}$ converges to $\vo$ uniformly on bounded intervals.
Second, by the
continuity
of $\vh(\cdot,\xi)$ in Assumptions \ref{A}(i), and the fact that $\bar{\vX}^0(t_n+\cdot)-\vX(\cdot)$ converges to $\vo$ uniformly on bounded intervals, we obtain
that (see e.g., \cite[Proof of Theorem 2.4.1]{KC78})
\begin{equation}\label{eq-4.0}
\begin{aligned}
\int_0^t\left(\vh(\bar{\vX}^0(t_n+s),\bar \xi^0(t_n+s))-\vh(\vX(s),\bar\xi^0(t_n+s))\right)ds\to \vo\text{ uniformly on bounded intervals}.
\end{aligned}
\end{equation}
On the other hand, we also have that
\begin{equation}\label{eq-4.3}
\lim_{n\to\infty}\int_0^t\left( \vh(\vx,\bar \xi^0(t_n+s))-\bar \vh(\vx)\right)ds=\vo,
\end{equation}
uniformly in $(t,\vx)$ on bounded sets.
In fact, by {\rm\ref{A}(iii)} we first only get  the convergence \eqref{eq-4.3} being uniform on bounded $t$-intervals for $\vx$ being in countable dense set $\mathcal H_0$. However, because of the assumptions on regularity of $\vh(\cdot,\cdot)$ and $\bar \vh(\cdot)$, we obtain the uniform convergence on bounded sets.
Combining \eqref{eq-4.0} and \eqref{eq-4.3} implies that
$$
\int_0^t\left[\vh(\bar{\vX}^0(t_n+s),\bar \xi^0(t_n+s))-\bar \vh(\vX(s))\right]ds\to\vo\text{ uniformly on bounded intervals}.
$$
The uniform convergence to $\vo$ of $\vga^n(\cdot)$ and $\vps^n(\cdot)$ follow from Assumptions {\rm \ref{A}(iv)} and {\rm \ref{A}(v)}.
Hence, $\{\vQ^n(\cdot)\}$ converges to $\vX(\cdot)$ uniformly on bounded intervals.
To proceed, we have the following proposition, whose proof
can be found in \cite[Lemma 4.1 and Lemma 4.2]{H79}.

\begin{prop}\label{lem-41}
We have the following results.
\begin{itemize}
\item[]{\rm (a)} Let $C(t):\R\to 2^{\R^d}$ be a set-valued mapping, whose values are compact, convex, and all contained in a common ball, i.e., there is a finite ball $B_C\subset\R^d$ such that $C(t)\subset B_C$ for all $t$. Then $\int_0^1C(t)dt$ is compact and convex.
\item[]{\rm (b)} Let $S(t):\R\to 2^{\R^d}$ be a set-valued mapping, whose values are all
   contained in a common ball.
If $\vX(\cdot):[0,1]\to\R^d$ satisfies that
$$
\vX(t)-\vX(s)\in \int_s^t S(r)dr,\;\text{for all }s<t\in [0,1],
$$
then $\vX(\cdot)$ is absolutely continuous and satisfies that
$
\dot \vX(t)\in \co \;S(t)$  almost everywhere in $[0,1]$.

\end{itemize}
\end{prop}

Now, let $\eps,\delta>0$ be arbitrary. On bounded intervals, for $n$ large enough,
$
|\vec p_n(\cdot)|<\eps/2$.
Moreover, because of Assumption {\rm \ref{S1}}, the average of the ``radius of neighbor" $m_n(\vx,\xi_n)$ tends to $0$, thus on bounded intervals, for $n$ large enough, we have from \eqref{eq-9.42} that
$$
\int_s^t \vb_n(\bar{\vX}^0(t_n+r),\bar\xi^0(t_n+r))dr\in \int_s^t \mathcal K[\bar {\vb}]\big(\vX(r)+\vec p_n(r)\big)dr+\delta\bar B.
$$
Hence, for all
$t,s$ in bounded intervals, for $n$ large enough, one obtains from \eqref{eq-9.41} and \eqref{eq-9.42} that
$$
\vQ^n(t)-\vQ^n(s)\in \int_s^t\bar \vh(\vX(r))dr+\int_s^t\co\;\big(\bar {\vb}(\vX(r)+\eps B)\big)dr+\delta\bar B.
$$
By part (a) of Proposition \ref{lem-41}, letting $n\to\infty$, we obtain that
$$
\vX(t)-\vX(s)\in\int_s^t \bar \vh(\vX(r))dr+\int_s^t \co\;\big(\bar {\vb}(\vX(r)+\eps B)\big)dr+\delta\bar B.
$$
Letting $\delta\to 0$ combined with part (b) implies that $\vX(t)$ is absolutely continuous and
for almost $t$ in bounded intervals,
$$
\dot \vX(t)\in \co\;\big(\bar {\vb}(\vX(t)+\eps B)\big)+\bar \vh(\vX(t)),\;\forall \eps>0.
$$
Taking $\eps \to 0$, we obtain that for almost $t$ in bounded intervals
$$
\dot \vX(t)\in \cap_{\eps>0}\co\;\big(\bar {\vb}(\vX(t)+\eps B)\big)+\bar \vh(\vX(t))=\mathcal K[\bar {\vb}](\vX(t))+\bar \vh(\vX(t)).
$$
Hence, combined with Lemma \ref{lem-K},
we obtain that $\vX(t)$ satisfies the differential inclusion
$$
\dot \vX(t)\in \mathcal K[\bar {\vb}+\bar \vh](\vX(t)).
$$

{\bf Part 3: Stability.}
The proof of the limit set of $\vX(\cdot)$ being internally-chain transitive can be found in \cite[Theorem 3.6]{BHS05}.
Hence, the limit points of $\{\vX_n\}$ are contained in $\mathcal R$, the set of chain-recurrent points.
Since we still use the definition of
stability in the sense of Lyapunov, the argument for obtaining stability  is the same as that of \cite[Proof of Theorem 2.3.1]{KC78} or \cite[Proof of Theorem 5.2.1]{KY03}.
We will study
the stability (in the sense of Lyapunov) for differential inclusions later.
\end{proof}

Theorem \ref{mth-1} can be generalized
when we replace the Krasovskii operator by arbitrary set-valued mappings. We proceed with the conditions needed and the assertions.

\begin{enumerate}[label=\textbf{(G)}]
	\item\label{G1} There is a set-valued mapping $G:\R^d\to 2^{\R^d}$
	satisfying:
	\begin{itemize}
		\item[]{\rm (i)} $G(\cdot)$ has non-empty, compact, convex values, and all values are contained in a finite common ball, i.e., there
		is a finite ball $B_G\subset\R^d$ such that $G(\vx)\subset B_G$ for all $\vx$;
		\item[]{\rm (ii)}
		$G$ has a closed graph, i.e.,
		$\text{Graph}(G):=\{(\vx,\vy): \vy\in G(\vx)\},$
		is a closed subset of $\R^d\times\R^d$;
		\item[]{\rm (iii)}
		there is a sequence of (positive real-valued) continuous (in $\vx$, uniformly in $\xi$) functions $\{m_n(\vx,\xi)\}$ such that
		for all $n$, $\vx$, $\xi$,
		$$
		\vb_n(\vx,\xi)\in G(\vx)+m_n(\vx,\xi)\bar B,
		$$
		and that for some $T>0$, each $\eps>0$, and each $\vx$,
		\begin{equation*}
		\lim_{n\to\infty}\PP\left\{\sup_{j\geq n}\max_{t\leq T}\Big|\sum_{i=m(jT)}^{m(jT+t)-1}a_im_i(\vx,\xi_i)\Big|\geq \eps\right\}=0.
		\end{equation*}
		
	\end{itemize}
\end{enumerate}

\begin{thm}\label{mth-3}
If we replace Assumption {\rm\ref{S1}} by {\rm\ref{G1}} in Theorem \ref{mth-1}, then the conclusions in Theorem \ref{mth-1} continue to hold with
the limit differential inclusion \eqref{eq-kbh}  replaced by
\begin{equation}\label{eq-hg}
\dot \vX(t) \in \bar \vh(\vX(t))+G(\vX(t)).
\end{equation}
\end{thm}

\begin{proof}[Proof of Theorem \ref{mth-3}.]
To prove Theorem \ref{mth-3}, we need to generalize the results on closure of the set of Krasovskii solutions
for the set of solutions of the classes of differential inclusions that satisfy a ``nice" property (property \ref{cdF}) like the Krasovskii operator. Then we will prove this property holds for the set-valued mappings in our setting (having compact, convex values, contained in a finite common ball and having close graph).
The results are shown in the following two propositions.

\begin{prop}\label{prop-3.2-F}
Let $\vX_k(\cdot)$ be satisfied the following for all $t,s$ in $[0,1]$
$$
\vX_k(t)-\vX_k(s)\in \int_s^t \Big(F(\vX_k(r)+\vec p_k(r,\vX_k(r)))+\vec q_k(r,\vX_k(r))\Big)dr,
$$
for some sequences of functions $\{\vec p_k(\cdot)\}$ and $\{\vec q_k(\cdot)\}$ satisfying
$
\vec p_k\to\vec 0$ $\vec q_k\to\vec 0$ uniformly $($in $[0,1]$$)$.
Assume that $F:\R^d\to2^{\R^d}$ is a set-valued mapping, whose values are non-empty, compact, convex, and  in a common ball, and that
\begin{equation}\label{cdF}
\cap_{\eps>0}\co\;F(\vx+\eps B)= F(\vx),\;\forall \vx.
\end{equation}
If $\vX_k(\cdot)$ converges $($uniformly$)$ to $\vX(\cdot)$, then the limit $\vX(\cdot)$ is a solution of the following differential inclusion
$$
\dot \vX(t)\in F(\vX(t)).
$$
\end{prop}

\begin{proof}
For arbitrary $\eps, \delta>0$, there is a large number $N$ such that $\forall n\geq N$,
$$
|\vec p_n(\cdot)|<\eps,\quad|\vec q_n(\cdot)|<\delta,\quad |\vX_n(\cdot)-\vX(\cdot)|<\eps.
$$
Hence, we have
$$
\vX_k(t)-\vX_k(s)\in \int_s^t\co\;\big(F(\vX_k(r)+\eps B)+\delta B\big)dr\in \int_s^t\co\;\big(F(\vX(r)+2\eps B)+\delta B\big)dr.
$$
Letting $k\to\infty$, it follows from
Proposition \ref{lem-41} that $\vX(t)$ is absolutely continuous and for almost all $t$
$$
\dot \vX(t)\in \co\;\big(F(\vX(t)+2\eps B)+\delta B\big).
$$
Taking $\delta\to0$, we obtain
$
\dot \vX(t)\in \co\;F(\vX(t)+2\eps B)\;\forall \eps>0.
$
As a consequence,
$
\dot \vX(t)\in \cap_{\eps>0}\co\;F(\vX(t)+2\eps B).
$
Using \eqref{cdF}, we complete the proof.
\end{proof}

\begin{prop}\label{prop63}
Assume $F:\R^d\to 2^{\R^d}$ is a set-valued mapping, whose values are non-empty, convex, compact subsets, and contained in a finite common ball, and whose graph is closed.
Then, one has
\begin{equation*}
\cap_{\eps>0}\co\;F(\vx+\eps B)= F(\vx),\;\forall \vx.
\end{equation*}
\end{prop}

\begin{proof}
Let $\vx$ be fixed but otherwise arbitrary.
By Lemma \ref{lem5.2}, $F$ is upper semicontinuous. Hence, by \cite[Proposition 3, Chapter 1]{AC84}, we have
$
\co\;F(\vx+\eps B)\subset \co\;F(\vx+\eps \bar {B})=\text{co}\;F(\vx+\eps \bar {B}).
$
Therefore,
$
\cap_{\eps>0}\co\;F(\vx+\eps B)\subset \cap_{\eps>0}\text{co}\;F(\vx+\eps \bar {B}).
$
On the other hand, by \cite[Theorem 5.7]{RW97},  from closed graph property of $F$, we obtain that
\begin{equation}\label{10-9-1}
\{\vec u: \text{there exist }\vx_n\to \vx\text{ and } \vy_n\in F(\vx_n)\text{ such that }\vy_n\to \vec u\}\subset F(\vx).
\end{equation}

Now, let $\vec u\in \cap_{\eps>0}\co\;F(\vx+\eps B)$. Then $\vec u\in \cap_{\eps>0}\text{co}\;F(\vx+\eps \bar {B})$. As a consequence,
$\vec u\in\text{co}\;F(\vx+\frac 1n\bar {B})$  for all $n\in \N.$
By Carath\'eodory's theorem for convex hulls of sets in a Euclidean space \cite[Theorem 2.29]{RW97}, for each $n$, there are $d+1$ points $\vy_n^0,\dots,\vy_n^d$ and $d+1$ points $\vx_n^0,\dots,\vx_n^d$, $|\vx-\vx_n^i|\leq \frac 1n,\;\forall i=0,\dots,d$ and real numbers $a_n^0,\dots,a_n^d\in [0,1]$, $\sum_i a_n^i=1$ such that
$
\vec u=\sum_{i=0}^d a_n^i\vy_n^i$, $\vy_n^i\in F(\vx_n^i).
$
Since $2d+2$ sequences $\{a_n^i\}_{n=0}^\infty$, $\{\vy_n^i\}_{n=0}^\infty$ for $i=0,\dots,d$ are bounded, we can extract subsequences (still index the sequences by $n$ for simplicity)
such that all of them are convergent.
As a result,
$
\vec u=\sum_{i=0}^d a^i\lim_n\vy_n^i,\; \vy_n^i\in F(\vx_n^i),
$
where
$
a^i=\lim_n a_n^i$, and  $a^i\in[0,1],\sum_{i=0}^d a^i=1.
$
Since $\vx_n^i\to \vx$, by \eqref{10-9-1}, $\lim_n\vy_n^i\in F(\vx)$.
Combined with the convexity of $F(\vx)$, we obtain $\vec u\in F(\vx)$.
So,
$\cap_{\eps>0}\co\;F(\vx+\eps B)\subset F(\vx).$
The proof is complete.
\end{proof}

It is noted that we need only  take care of the ``characterization of the limit" part since the other parts are the same as that of Theorem \ref{mth-1}.
With the helps of Propositions \ref{prop-3.2-F} and \ref{prop63}, the arguments of ``characterization of the limit" part are similar to that of Theorem \ref{mth-1}; the details are thus omitted.
\end{proof}

\begin{rem}\label{rem-compare}
	The difficulty in our setting is that we impose neither
	continuity to the dynamics of the discrete iterations nor the limit systems.
	As a result, although we obtain the boundedness and equicontinuity of $\{\vX^n(\cdot)\}$ and can extract a convergent subsequence with the limit $\vX(\cdot)$, it is impossible to characterize the limit
	using continuity and compactness.
	To illustrate,
	we mention some related works and methods in the literature.
	In \cite{KC78,KY03}, the continuous dynamics
	with the limits being
	a set-valued mapping were treated.
	In this case, it is still possible to pass to the limit after extracting  convergent subsequence
	to characterize the limit.
	In \cite{K83,KY03,MP87}, possibly discontinuous $\vb_n(\cdot,\cdot)$ were considered, but the limits
	have some regularities.
	Under the regularities of the limits and some assumptions on existence of a Lyapunov function, certain  average takes place; see \cite{K83} for more details.
	Along another line, M\'etivier and Priouret in \cite{MP87} express the limit function
	in term of integration of $\vb_n(\cdot,\xi)$ over invariant measure of the noise process
	and use a Poisson equation approach; and thus,
	under some suitable conditions,
	the corresponding limit (continuous) differential equation may be obtained.
	In \cite{KC78}, the  case of that $\vb_n(\cdot,\cdot)$ allowing to be discontinuous and the limit being a set-valued mapping $G(\cdot)$ is considered.
	But the continuity of $G(\cdot)$ in the Hausdorff metric defined as
	$$
	d(S_1,S_2):=\sup_{\vy\in S_2}\inf_{\vx\in S_1} |\vy-\vx|+\sup_{\vx\in S_1}\inf_{\vy\in S_2}|\vy-\vx|,\;\forall S_i\subset \R^d, i=1,2,
	$$ is needed.
	However, these assumptions may not be satisfied when we do not have the desired continuity in applications. For instant, in the example of Lasso algorithm, which will be illustrated later, we need to consider set-valued mapping representing the sub-gradient of the function $|\vx|$. For example, in one-dimensional example, one may need to consider
	$
	G(x)=\begin{cases}
	\{-1\}\text{ if }x>0,\\
	[-1,1]\text{ if }x=0,\\
	\{1\}\text{ if }x<0.
	\end{cases}
	$
	This set-valued mapping is not continuous at $0$ in the Hausdorff metric.
	Except \cite{BHS05}, there has been
	no
	general approach in the literature
	for studying convergence of stochastic approximation schemes involving set-valued mappings
	without continuity.
		However,
	the setup and results of
	the current paper
	are different
than  that of \cite{BHS05}.
	If we let
	 $\vec b_n(\vx,\xi)$ be independent of $\xi$, $\vec h(\vx,\xi)=0$, $\vbe_n=0$, and $m_n(\vx,\xi)=0,\forall n,\vx,\xi$, where $m_n(\vx,\xi)$ is as in Assumption \ref{S1} or \ref{G1}, we recover the setting and results in \cite{BHS05}.
	 In this paper, $m_n(\vx,\xi)$ is not required tending to $0$, which
	 makes the setting
	 more general and
	 applicable in real applications.
In addition, in
	\cite{BHS05},
	 the limit processes are perturbed solutions of the corresponding differential inclusions, whereas we characterize the limit processes by differential inclusions rather than perturbed differential inclusions.
That is done by examining the closure of the set of solutions of a family of differential inclusions for
general set-valued mappings.
\end{rem}

\para{Convergence to equilibrium point.}
The following results
are concerned with globally asymptotic stability of the limit
differential inclusions.
 It also establishes the
 convergence to the equilibria
 of stochastic approximation algorithm \eqref{eq-algo}.
 We introduce the following stability condition for Krasovskii solutions of ODEs with discontinuous right-hand sides, which is similar to Lyapunov condition in classical stability theory.

\begin{enumerate}[label=\textbf{(KS)}]
%
	
	\item\label{S2} There is a unique equilibrium $\vx^*$ of $\bar {\vb}(\cdot)+\bar \vh(\cdot)$, i.e., $\bar {\vb}(\vx^*)+\bar \vh(\vx^*)=\vec 0$ (where, $\bar \vh(\cdot)$ is as in Assumption {\rm\ref{A}}(iii) and $\bar {\vb}(\cdot)$ is as in Assumption {\rm\ref{S1}});
	and there exists a $C^\infty$-smooth pair of functions $(V,W)$ satisfying that $V(\vx)>0$ and $W(\vx)>0,\forall \vx\neq\vec 0$, $V(\vec 0)=0$, and the sublevel sets $\{\vx\in\R^d: V(\vx)\leq l\}$ are bounded for every $l\geq 0$, and
	$$\limsup_{\vy\to \vx}\langle \nabla V(\vx), \bar {\vb}(\vy+\vx^*)+\bar \vh(\vy+\vx^*)\rangle\leq -W(\vx),\;\forall \vx\neq\vec 0.$$

\end{enumerate}

\begin{thm}\label{mth-2}
	Consider
	algorithm \eqref{eq-algo}. Under Assumptions {\rm \ref{A}},
	{\rm\ref{S1}}, {\rm\ref{S2}}, and  boundedness of $\{\vX_n\}$,
	there exists a null set $\Omega_0$ such that if $\omega\notin\Omega_0$, then
	$\vX_n$ converges to the unique equilibrium $\vx^*$.
	
\end{thm}

For the general case, where the Krasovskii operator is replaced by  set-valued mappings, we introduce a stability condition {\rm \ref{G2}} as follows.
 Our approach is based on a novel method, namely, $\mathcal U$-generalized Lyapunov functional method for differential inclusions.

\begin{enumerate}[label=\textbf{(GS)}]
%
%

	\item\label{G2}
	There is a unique $\vx^*$ such that $\vec 0\in \bar \vh(\vx^*)+G(\vx^*)$ (where, $\bar \vh(\cdot)$ is as in Assumption {\rm\ref{A}(iii)} and $G(\cdot)$ is as in Assumption {\rm\ref{G1}}); and there exists a $\mathcal U$-generalized Lyapunov function $V:\R^d\to\R_+$ such that the sublevel sets $\{\vx\in\R^n:V(\vx)\leq l\}$ are compact for every $l>0$ and
	the $\mathcal U$-generalized derivative $\dot{\bar V}^{G^*}_{\mathcal U}(\vx)$ satisfies
	$\dot{\bar V}^{G^*}_{\mathcal U}(\vx)\leq -\hat V_0(\vx),\;\forall \vx\neq\vec 0,$
	for some positive
	definite
	function
	$\hat V_0$,
	where
	$G^*(\vx):= \bar \vh(\vx+\vx^*)+G(\vx+\vx^*);$
	see Section \ref{sec:non} (Definition \ref{def-56-11}, Definition \ref{non-def-1}(iv)) for  these concepts.
	
\end{enumerate}

\begin{thm}\label{mth-4}
If we replace Assumptions {\rm\ref{S1}} and {\rm\ref{S2}} by {\rm\ref{G1}} and {\rm \ref{G2}}, then the conclusion of Theorem \ref{mth-2} continue to hold.
\end{thm}

\begin{proof}[Proof of Theorems \ref{mth-2} and \ref{mth-4}.]
The stability of differential inclusions is carefully studied in Section \ref{sec:63-asy}.
The proof of Theorem \ref{mth-2} follows from Theorem \ref{mth-1} and Theorem \ref{thm-sta} in Section \ref{sec:63-asy}.
First, under Assumption {\rm \ref{S2}}, the Krasovskii solutions of \eqref{eq-bh} are strongly asymptotically stable (in Clarke's sense) at $\vx=\vx^*$. Therefore, every Krasovskii solutions of \eqref{eq-bh} is globally asymptotically stable at $\vx=\vx^*$ in the Lyapunov sense.
As the last part of Theorem \ref{mth-1}, $\{\vX_n\}$ must converge to the equilibrium point $\vx^*$ w.p.1.
Similarly, Theorem \ref{mth-4} is obtained by combining Theorem \ref{mth-3} and Theorem \ref{thm-sta2}.
\end{proof}

\para{Projection algorithms.}
As was mentioned before, the assumption on boundedness of $\{\vX_n\}$ is not restrictive.
Since the boundedness is not our main focus, we  often assume it in our main results so as to make the argument simpler.
Further conditions and/or various projection algorithms may be used; see Remark \ref{rem-00}.
We proceed to state the results for constrained algorithm \eqref{eq-proj}.

\begin{enumerate}[label=\textbf{(P)}]
	\item\label{H1}
	The projection space $H$ is a hyper-rectangle, i.e., $H=\{\vx\in\R^d: b_i\leq x_i\leq c_i\}$ for simplifying arguments. In general, $H$ can be compact and convex; and $H=\{\vx\in\R^d: q_i(\vx)\leq 0, i=1,\dots,N\}$, the constrained functions $q_i(\cdot)$, $i=1,\dots,N$ are continuously differentiable and at $\vx\in\partial H$, the gradients $q_{i,\vx}(\cdot)$ are linearly independent.
\end{enumerate}	

\begin{enumerate}[label=\textbf{(PS)}]
	\item\label{H2}
	There is a unique $\vx^*\in H$ such that $\vec 0\in \co\;\Pi_H\big[\bar \vh(\vx^*)+G(\vx^*)\big]$; and there exists a $\mathcal U$-generalized Lyapunov function $V:\R^d\to\R_+$ such that the sublevel sets $\{\vx\in\R^d:V(\vx)\leq l\}$ are compact for every $l>0$ and
	$\dot{\bar V}^{G_H^*}_{\mathcal U}(\vx)\leq -\hat V_0(\vx),\;\forall\; \vec 0\neq \vx\in H,$
	for some positive definite function $\hat V_0$, where
	$G_H^*(\vx):= \co\;\Pi_H\big[\bar \vh(\vx+\vx^*)+G(\vx+\vx^*)\big].$

\end{enumerate}

\begin{thm}\label{mth-5}
Consider
algorithm \eqref{eq-proj}. Assume {\rm\ref{G1}}, {\rm \ref {H1}}, and {\rm \ref{A}}
with {\rm\ref{A}(ii)} replaced by $\vh(\vx,\xi)$ being $($uniformly in $\xi)$ locally bounded in $\vx$  {\rm (}i.e., $|\vh(\vx,\xi)|\leq K(\vx)$ for some locally bounded function $K)$.
Then, there is a null set $\Omega_0$ such that $\forall \omega\notin\Omega_0$, $\{\vX^n(\cdot)\}$ is bounded and equicontinuous.
Let $\vX(\cdot)$ be the limit of a convergent subsequence of $\{\vX^n(\cdot)\}$. Then $\vX(t)$ is a solution of the differential inclusion
\begin{equation}\label{eq-kbh-proj}
\dot \vX(t)\in \co\;\Pi_H\big(\bar\vh(\vX(t))+G(\vX(t))\big).
\end{equation}
The limit set of $\{\vX(\cdot)\}$ is internally chain transitive
and as a consequence,
the limit points of $\{\vX_n\}$ are contained in $\mathcal R$, the set of chain recurrent points
 of \eqref{eq-kbh-proj} $($see Section \ref{sec:dyn} for the definitions$)$.
In addition, if we assume further {\rm\ref{H2}}, then $\{\vX_n\}$ converges to $\vx^*$ w.p.1.
\end{thm}

\begin{proof}[Proof of Theorem \ref{mth-5}.]
First, to use Assumption
{\rm \ref{A}(iv)}
in the projection algorithm, let $Y_n$ be a sequence of positive real numbers such that $Y_n\to 0$ and $|a_n\vh_0(\wdt \zeta_n)|\leq Y_n/2$ excepting a finite number of $n$ w.p.1 (such a sequence $Y_n$ exists owing to Assumption
{\rm\ref{A}(iv)},
Borel-Cantelli lemma \cite[Section 5]{KC78}),
and let
$I_n$ be
the indicator of the set where $|a_n\vh_0(\wdt \zeta_n)|\leq Y_n/2$.
To proceed, we write
algorithm \eqref{eq-proj} as
\begin{equation}\label{eq-t5-1}
\vX_{n+1}=\vX_n+a_n\big[\vb_n(\vX_n,\xi_n)+\vh(\vX_n,\zeta_n)+\vh_0(\wdt \zeta_n)+\vbe_n\big]+\vta_n+\vpp_n,
\end{equation}
where
$$
\vta_n=\Pi_H(\vM_n^Y)-\vM_n^Y,\quad \vpp_n=(\vM_n^Y-\vM_n)+\big[\Pi_H(\vM_n)-\vX_n\big](1-I_n),
$$
$$
\vM_n=\vX_n+a_n\big[\vb_n(\vX_n,\xi_n)+\vh(\vX_n,\zeta_n)+\vh_0(\wdt \zeta_n)+\vbe_n\big],
$$
and
$$
\vM_n^Y=\vX_n+a_n\big[\vb_n(\vX_n,\xi_n)+\vh(\vX_n,\zeta_n)+\vh_0(\wdt \zeta_n)+\vbe_n\big]I_n.
$$
The purpose of partitioning \eqref{eq-t5-1} enables us to apply directly our assumptions (which is assumed without any constrains).

{\bf Part 1: Boundedness and Equicontinuity.}
Similar to \eqref{eq-4.00}, we have that
\begin{equation}\label{eq-4.00-11}
\begin{aligned}
\vX^n(t)
=&\vX_n+\int_0^t \vb_n(\bar{\vX}^0(t_n+s),\bar \xi^0(t_n+s))ds+\int_0^t \vh(\bar{\vX}^0(t_n+s),\bar \xi^0(t_n+s))ds\\
&+\vga^n(t)+\vps^n(t)+\vta^n(t)+\vpp^n(t).
\end{aligned}
\end{equation}
In the above,
$\vta^n(t):=\vta^0(t_n+t)-\vta^0(t_n)$, $\vpp^n(t):=\vpp^0(t_n+t)-\vpp^0(t_n),$
where, $\vta^0(\cdot)$ and $\vpp^0(\cdot)$ are the piecewise linear interpolations of $\{\sum_{i=0}^{n-1}a_i\vta_i\}$ and $\{\sum_{i=0}^{n-1}a_i\vpp_i\}$, respectively; and the $\vga^n(\cdot)$, $\vps^n(\cdot)$ are as in the proof of Theorem \ref{mth-1}.

Let $\Omega_0$ be the union of sets in which $|a_n\vh_0(\wdt \zeta_n)|\geq Y_n/2$ infinitely often and the exceptional sets in
\ref{A}(iii)-(v),
\ref{G1} (the union being taken over countable dense set $\mathcal H_0$).
[As we mentioned before, for each $\vx\in \mathcal H_0$, there are (null) exceptional sets (in which, the convergence assumptions do not hold) corresponding to
\ref{A}(iii)-(v),
and \ref{G1};  $\Omega_0$ is taken to contain all these sets.
Since $\mathcal H_0$ is countable, $\Omega_0$  still has measure zero.]
Therefore, we work with a fixed $\omega\notin\Omega_0$.

	As in the proof of Theorem \ref{mth-1}, we proved that $\vga^n(\cdot)$ and $\vps^n(\cdot)$ converge uniformly to $\vo$ on finite $t$-intervals.
	Moreover, because $I_n$ is 0 only for a finite number of $n$,  only a finite number of the terms of $\{(1-I_n)\}$ are nonzero.  Since
	$$
	\vpp_n\leq a_n |h(\vX_n,\xi_n)+b_n(\vX_n,\zeta_n)+\vec h_0(\wdt\zeta_n)+\vbe_n|(1-I_n)+|\Pi_H(\vM_n)-\vX_n|(1-I_n),
	$$
	it is readily seen that $\vpp^n(\cdot)$ converges to $\vo$ uniformly on finite intervals as $n\to\infty$.
	The boundedness of $\{\vX^n(\cdot)\}$ is
clear because of the use of the
projection algorithm.
	
	Next, we  prove the equicontinuity of $\{\vX^n(\cdot),\vta^n(\cdot)\}$.
	It suffices to prove the equicontinuity for $\vta^n(\cdot)$;
see the proof of Theorem \ref{mth-1}.
	By the definition of $\vta_n$, we have following observations (see e.g., \cite[Proof of Theorem 5.3.1]{KC78}):
		$\vta_n$ is orthogonal to $H$ at the point $\Pi_H(\vM_n^Y)$;
		and $|\vta_n|\leq a_n(K_1+Y_n)$ for some constant $K_1$;
		 and there is a constant $K_2$ such that $\vta_n=\vec 0$ if $\text{distance}(\partial H,\vX_n)\geq K_2(Y_n+a_n)$.
	Because of these observations and the fact that $\vX^0(\cdot)-\vta^0(\cdot)$ is uniformly continuous (due to this difference is in fact the process in non-projected case and is proved before),
	$\vta^0(\cdot)$ must be uniformly continuous on $[0,\infty)$.
	Otherwise, there would be $s_k\to\infty$, $\delta_k\to 0$ and $\eps>0$ such that
	$$
	|\vX^0(s_k+\delta_k)-\vX^0(s_k)| \geq \eps,\text{ for all }k,
	$$
	with $\text{distance} (\vX^0(s_k),\partial H)\to0$ as $k\to\infty$ and $\text{distance}(\vX^0(s_k+\delta_k),\partial H)\geq \eps/2$.
 However, this contradicts the observations of $\vta_n$ and the uniform continuity of $\vX^0(\cdot)-\vta^0(\cdot)$. The uniform continuity of $\vta^0(\cdot)$ implies the equicontinuity of $\{\vta^n(\cdot)\}$.

{\bf Part 2: Characterization of the limit.}
Extract a convergent subsequence of $\{(\vX^n(\cdot),\vta^n(\cdot))\}$, and index it again by $n$ with the limit
$(\vX(\cdot),\vta(\cdot))$.
Using the fact that $\vga^n(\cdot)$, $\vps^n(\cdot)$, $\vpp^n(\cdot)$ converge to $\vo$ uniformly and letting $n\to\infty$ in \eqref{eq-4.00-11}, by a similar argument as in the unconstrained case, one has that on bounded intervals
\begin{equation}\label{eq-t5-2}
\vX(t+s)-\vX(t)\in \int_t^{t+s} G(\vX(r))dr +\int_t^{t+s} \bar \vh(\vX(r))dr +\vta(t+s)-\vta(s).
\end{equation}
As in \cite[Proof of Theorem 5.3.1]{KC78} or \cite[Proof of Theorem 6.8.1]{KY03}, we have
\begin{equation}\label{eq-t5-3}
\vta(t+s)-\vta(s)=\int_t^{t+s}\vec z(\vX(r))dr,
\end{equation}
where $\vec z(\vX(t))$ is the minimal force needed to keep $\vX(t)$ in $H$.
A consequence of \eqref{eq-t5-2} and \eqref{eq-t5-3} is that
\begin{equation}\label{eq-t5-4}
\vX(t+s)-\vX(s)\in \int_t^{t+s} \co\;\Pi_H\big[\bar\vh(\vX(r))+G(\vX(r))\big]dr.
\end{equation}
Combining \eqref{eq-t5-4} and Proposition \ref{lem-41}, one has that $\vX(t)$ is absolutely continuous and for almost all $t$,
$$
\dot{\vX}(t)\in \co\;\Pi_H\big[\bar\vh(\vX(t))+G(\vX(t))\big].
$$

{\bf Part 3: Asymptotic stability.}
This part is the same as that of the unconstrained case and is thus omitted.
\end{proof}

\begin{rem}\label{rem1}
Recall
that we often wish to find
roots of some functions and/or set-valued mappings. [For the roots of set-valued mappings, we mean that at these points (roots), the value of these mappings (being a set) contains $\vec 0$].
These points are often called  ``stationary points" of the corresponding differential equations or inclusions.
In the  set-valued and differential inclusion cases, the roots
may not be (strongly) stationary, where ``strongly" means the statement is true for all solutions.
If the function is vector-valued and is sufficiently smooth (namely, $d$-time continuously differentiable, where $d$ is the dimension),
 then the set of stationary points is
equal to $\mathcal R$, the set of chain recurrent points (of the corresponding differential equation).
Otherwise, by a sophisticated process, Hurley in \cite{H95} shows that if the function is smooth but not smooth enough, the set of chain-recurrent points maybe strictly larger than the set of stationary points. Hence, $\{\vX_n\}$ may not converge to the desired
stationary points.
In our setting, if there is no condition to guarantee the stability of stationary
points (termed roots for simplicity), the algorithm may
not converge to a set of roots, even if the algorithm starts at one of the roots.
It is  easy to give an example;
see Example \ref{ex-num-4} in Section \ref{sec:num}.
\end{rem}

\begin{rem}
The stability of the
systems of interest can be characterized by means of the stability in set-valued dynamical systems \cite[Section 3 and 4]{BHS05} and references therein.
However, the conditions in the aforementioned reference is relatively abstract and difficult
to verify in applications.
We use criteria on ${\mathcal U}$-generalized Lyapunov functions instead.
Moreover,
we give an example in Section \ref{sec:app4} to show that convergence
can be proved by applying our results, but cannot be done otherwise.
\end{rem}

\begin{rem}\label{rem-5}
It is worth noting that  the Clarke sub-differential of Lipschitz continuous function has the important property \eqref{cdF}.
In addition, the stability assumptions \ref{S2}, \ref{G2}, and \ref{H2} are not restrictive.
  They are similar to Lyapunov conditions in classical stability analysis excepting that we need to compute a new type of derivative (namely, $\mathcal U$-generalized derivative)
 for $\mathcal U$-generalized Lyapunov functional. The examples of computing these new functions are given in Section \ref{sec:app} and Appendix \ref{sec:non}.
Moreover, Theorems \ref{mth-1} and \ref{mth-2} are less general than Theorems \ref{mth-3} and \ref{mth-4}. However, if we can
express a set-valued mapping as the Karasovskii operator of some vector-valued function,  condition \ref{S2} is more convenient
 to verify
 than that of \ref{G2}.
In the projection algorithm,
since
$H$ is convex, $\Pi_H(\vx)$ is uniquely defined and
$\Pi_H(\cdot)$ is continuous.
However, the convex closure in \eqref{eq-kbh-proj} cannot be relaxed
since a continuous projection operator may not preserve the convexity.
\end{rem}

		
		

		
		
		

\para{Biased Stochastic Approximation.}
Next, we study biased stochastic approximation.
With the term $\vbe_n$ representing a bias,
by ``biased stochastic approximation", we mean the bias is not ``asymptotically negligible".
To proceed,
let
$
\eta=\limsup_{n\to\infty}\|\vbe_n\|
$
be a random variable that is bounded w.p.1.
We study
stochastic approximation schemes \eqref{eq-algo} and \eqref{eq-proj}
with the dependence on $\eta$.
For a set $S\subset\R^d$, an $\eps$-neighborhood of $S$ denoted by $N_{\eps}(S)$ is defined as
$$
N_{\eps}(S)=\{\vx\in\R^d: \text{distance}(\vx,S)\leq \eps\},\;\text{distance}(\vx,S):=\inf_{\vy\in S}|\vx-\vy|.
$$

\begin{thm}\label{mth-6}
Consider
algorithm \eqref{eq-algo}, assume that
{\rm \ref{A}(i)-(iv)}
and {\rm \ref{G1}} hold, and that $\{\vX_n\}$ is bounded w.p.1 $($resp., consider
algorithm \eqref{eq-proj} and assume that
{\rm \ref{A}(i)-(iv)},
{\rm \ref{G1}}, and {\rm\ref{H1}} hold$)$.
\begin{itemize}
\item
 Then, there is a null set $\Omega_0$ such that $\forall \omega\notin\Omega_0$, $\{\vX^n(\cdot)\}$ is bounded and equicontinuous.
\item
Let $\vX(\cdot)$ be the limit of a convergent subsequence of $\{\vX^n(\cdot)\}$. Then $\vX(\cdot)$ is a solution of the differential inclusion
\begin{equation}\label{ubi-1}
\dot \vX(t) \in N_{2\eta}\left(\bar \vh(\vX(t))+G(\vX(t))\right)\ \Big(resp.,\;\dot \vX(t)\in N_{2\eta}\big(\co\;\Pi_H\big(\bar\vh(\vX(t))+G(\vX(t))\big)\big)\Big).
\end{equation}


\item
There exists a $($deterministic$)$ positive function $\phi(\cdot):[0,\infty)\to[0,\infty)$ depending on $\limsup_{n}|\vX_n|$ $($resp., the projection space $H)$ such that
$\lim_{t\to 0}\phi(t)=\phi(0)=0$ and
\begin{equation}\label{ubi-2}
\limsup_{n\to\infty}\mathrm{distance}(\vX_n,\mathcal {R})\leq\phi(\eta),
\end{equation}
where $\mathcal R$ is the set of chain recurrent points of differential inclusion 
$$
\dot{\vX}(t)\in\bar \vh (\vX(t))+G(\vX(t))\ \left(resp.,\;\dot \vX(t)\in\co\;\Pi_H\big(\bar\vh(\vX(t))+G(\vX(t))\big)\right).
$$

\item
Assume further that there is a unique $\vx^*$ such that $\vec 0\in \bar \vh(\vx^*)+G(\vx^*)$; and that there exists a $\mathcal U$-generalized Lyapunov function $V:\R^d\to\R_+$ such that the sublevel sets  $\{x\in\R^n:V(\vx)\leq l\}$ are compact for every $l>0$ and
the $\mathcal U$-generalized derivative $\dot{\bar V}^{G_\eta^*}_{\mathcal U}(\vx)$ satisfies the ``decay condition" in the sense of Assumption {\rm\ref{G2}} $($resp., {\rm\ref{H2}}$)$ with $G^*(\vx)$ being replaced by
$
G^*_\eta(\vx):=N_{2\eta}\left(\bar \vh(\vx+\vx^*)+G(\vx+\vx^*)\right).
$
Then,  $\{\vX_n\}$ converges to $\vx^*$ w.p.1.
\end{itemize}
\end{thm}

\begin{rem}{\rm
If the dynamics and the limits of stochastic approximations are smooth enough (namely, real analytic or $k$-times continuously differentiable with $k> d$ and $d$ being the dimension of the space), a more precise characterization of the asymptotic bias ($\phi(\eta)$ in \eqref{ubi-2}) is obtained by Tadi\'c and Doucet in \cite{TD17}  using Yomdin theorem $($a qualitative version of the Morse-Sard theorem$)$ and the
Lojasiewicz inequality.
This paper deals with systems with discontinuity.}
\end{rem}

\begin{proof}[Proof of Theorem \ref{mth-6}.] We prove the assertion for unconstrained case only; the constrained case can be handled similarly.
	
{\bf Part 1: Boundedness and equicontinuity.} This part is the same as that of unbiased case.
In fact, we can treat $\vbe_n$ as a (uniformly) bounded term  and hence, boundedness and equicontinuity of sequence of the piecewise linear interpolated processes still hold.

{\bf Part 2: Characterization of the limit.}
The process of obtaining the limit system is almost the same as that of unbiased case, excepting for that the limit differential inclusion should be relaxed.
To be more specific, in Proposition \ref{prop-3.2-F}, if we relax the condition that $\vec q_k(\cdot)\to \vec 0$ uniformly to be that $|\vec q_k(\cdot)|<\eta$ uniformly for $k$ large enough, we will obtain similar results as in unbiased case with $G(\cdot)$ being replaced by its $2\eta$-neighborhood due to the bias term $\vbe_n$.

{\bf Part 3:
Proof of \eqref{ubi-2} and stability assertion.}
Let $Q$ be a compact set such that $\{\vX_n\}_{n=1}^{\infty}\subset Q$ and
$\M_Q$ be the largest invariant set contained in $Q$; and let
$\mathcal R_{2\eta}(Q)$ be the set of chain recurrence points of following differential inclusion restricted in $Q$
\begin{equation}\label{proof-6-eq-1}
\dot\vX(t) \in N_{2\eta}\left(\bar \vh(\vX(t))+G(\vX(t))\right),
\end{equation}
i.e., $\mathcal R_{2\eta}(Q)$ contains all $\boldsymbol{\theta}$ satisfying that for any $\eps >0$, $T>0$, there are an integer $n$, real numbers $t_1,\dots,t_n>T$,
and solutions $\vx_1(\cdot),\dots,\vx_n(\cdot)$ of \eqref{proof-6-eq-1} such that $\forall k=1,\dots,n$,
$$
\vx_k(0)\in \M_Q,\quad
|\vx_1(0)-\boldsymbol{\theta}|<\eps,\quad |\vx_k(t_k)-\vx_{k+1}(0)|\leq \eps,\quad |\vx_n(t_n)-\boldsymbol{\theta}|\leq \eps.
$$
To proceed, we need the following lemma, whose proof can be found in \cite[Lemma 5.1]{TD17}, which used the continuation of chain-recurrent set developed in \cite[Theorem 3.1]{BHS12}.

\begin{lem}\label{proo-6-lem-1} $($Continuation of chain recurrent set$)$
	There exists a function $\phi(\cdot):[0,\infty)\to[0,\infty)$ $($depending on $Q$ and $\mathcal R)$ such that $\phi(\cdot)$ is non-decreasing with $\lim_{t\to0}\phi(t)=\phi(0)=0$ and $\mathcal R_{2\eta}(Q)\subset N_{\phi(\eta)}(\mathcal R)$, where $N_{\phi(\eta)}(\mathcal R)$ is a $\phi(\eta)$-neighborhood of $\mathcal R$.
\end{lem}

It is similar to the arguments in the unbiased case, we obtain that
the limit points of $\{\vX_n\}$ are contained in $\mathcal R_{2\eta}(Q)$.
On the other hand,
applying Lemma \ref{proo-6-lem-1}, we obtain that $\mathcal R_{2\eta}(Q)\subset  N_{\phi(\eta)}(\mathcal R)$.
Therefore, we conclude our results.
Finally, the stability assertion is similar to that of unbiased case,
which turns out to be the study of the stability of the limit differential inclusion.

\end{proof}

\section{Rates of Convergence}\label{sec:rate}
This section is concerned with the rates of convergence of the stochastic approximation algorithms.
 One of the new features of our work is that
stochastic differential inclusions are used in the rate of convergence study for the first time.

For simplicity, we consider the following algorithm
\begin{equation}\label{eq-algo-r}
\vX_{n+1}=\vX_n+a_n\vh(\vX_n,\xi_n)+a_n\vb_n(\vX_n), \  \ \vb_n (\vX_n ) \in G (\vX_n)
.\end{equation}
We assume that the limit dynamical system has
a global stable limit point $\vx^*$. The rate of convergence is focused on the asymptotic behavior of
$
\vU_n:=\frac{\vX_n-\vx^*}{\sqrt {a_n}}.
$
Let $\vU^0(\cdot)$ be the piecewise constant interpolation of $\{\vU_n\}$,
 and $\vU^n(\cdot)$ be its shifted process, i.e.,
$$
\vU^0(t):=\vU_n\text{ if }t\in [t_n,t_{n+1});\text{ and }\vU^n(t):=\vU^0(t_n+t),\;t\geq 0.
$$
We
state only the results for unconstrained case with assumption on boundedness of $\{\vX_n\}$. The projection case is similar with a slight modification.
We assume following assumption.

\begin{enumerate}[label=\textbf{(R)}]
	\item\label{R}
	\begin{itemize}[leftmargin=*,align=left]
	\item[]{\rm (i)}
	The sequence of step sizes  $\{a_n\}_{n\geq 0}$ satisfies $0< a_n \to 0$ as $n\to \infty$ and
	$(a_n/a_{n+1})^{1/2}=1+\eps_n$ where
	{\rm(a)} $\eps_n=\frac 1{2n}+o(\eps_n)$ if $a_n= 1/n$,
	or
	{\rm(b)} $\eps_n=o(a_n)$.
	\item[]{\rm (ii)}
	There is a limit point $\vx^*$ satisfying the following conditions:
	{(a)}
	$\vX_n\to \vx^*$ w.p.1
	and $\bar \vh(\vx^*)+G(\vx^*)=\{\vec 0\}$;
	{(b)}
	$\{(\vX_n-\vx^*)/\sqrt{a_n}\}$ is tight.
	\newcommand{\cggg}{{\mathcal G}}
	\newcommand{\chhh}{{\mathcal H}}
	\item[]{\rm (iii)}
	The functions $\vh(\cdot,\cdot)$ and  $\vh_\vx(\cdot,\cdot)$ (gradient with respect to $\vx$) are continuous in $(\vx,\xi)$ and bounded on bounded $\vx$-sets. The second partial derivative (with respect to $\vx$) $\vh_{\vx \vx}(\cdot,\xi)$ exists and is bounded uniformly in $\xi$, and
	$\vh_{\vx\vx}(\cdot, \xi)$ is continuous in a neighborhood of $\vx^*$. The $\{\xi_n\}$  is a sequence of uniformly bounded and stationary uniform mixing process satisfying that:
	$\E \vh(\vx,\xi_n)= \bar \vh(\vx)$ and $\E \vh_\vx(\vx, \xi_n)=
	\bar \vh_\vx(\vx)$.
	Let \bea \ad \psi_n=\vh(\vx^*,\xi_n)- \bar \vh(\vx^*),\ \wdt \psi_{n}=\vh_\vx(\vx^*,\xi_n)-\bar \vh_\vx(\vx^*), \\
	\ad \sf \cggg_n=\sigma \{\psi_j; j\le n\},\  \cggg^n=\sigma
	\{\psi_j; j\ge n\},\
	\chhh_n=\sigma \{\psi_{j}; j\le n\},\  \chhh^n=\sigma
	\{\psi_{j}; j\ge n\},\\
	\ad \phi(m)=\sup_{\A\in
		\cggg^{n+m}}|\PP(\A|\cggg_n)-\PP(\A)|_{\infty},\
	\tilde \phi(m)=\sup_{\A\in
		\chhh^{n+m}}|\PP(\A|\h_n)-\PP(\A)|_{\infty},\\
	\eea
	For some $\Delta>0$,
	$\sum_j \phi^{\Delta\over {1+\Delta}}(j) <\infty, \ \sum_j \wdt\phi^{\Delta\over {1+\Delta}}(j) < \infty.
	$

	\item[]{\rm (iv)}
	The set-valued mapping $G(\cdot)$ has non-empty, convex, and compact values, which are contained in a finite common ball such that
	$
	\vb_n(\vx)\in G(\vx)\;\forall n.
	$
	Moreover,
	there is a continuous and positively homogeneous set-valued mapping $T$, whose values are non-empty, convex, compact, and contained in a finite common ball such that $G$ is outer $T$-differentiable at $\vx^*$ (see Section \ref{sec:set} for these concepts).
	\end{itemize}
\end{enumerate}

\begin{rem}\label{rem:weak-conv}
	Condition \ref{R}(i) covers commonly used step sizes $\{a_n\}$.
	Because our main interest here is on the rate of convergence, we simply assume the convergence of $\vX_n$ to $\vx^*$.
	For simplicity of presentation and as a division of labor, we assume
	the tightness of $\{\frac{\vX_n-\vx^*}{\sqrt{a_n}}\}$
	in \ref{R}(ii).  Sufficient conditions ensuring the tightness are given at the end of this section and presented as Proposition \ref{prop:tight}.
Regarding {\rm \ref{R}(iii)}, we use the notation as in \cite[Chapter 7, pp. 345-346]{EK86}.
That is, $|\cdot|_p$ denotes the $p$-norm for $L^p(\Omega, {\cal F}, {\mathbb P})$ with $1\le p\le \infty$.
	It can be shown (see \cite{Yin91}) that
	\begin{itemize}
		\item[]{(a)}
		$\sum_{i=n}^{m(t_n+ \cdot)-1}\sqrt {a_i} [\vh(\vx^*, \xi_i)-\bar\vh(\vx^*)]$ converges weakly to a Brownian motion $\vW (\cdot)$ with covariance $\Sigma_1 t$ as $n\to \infty$,  and
		\item[]{(b)} $\sum_{i=n}^{m(t_n+ t)-1}a_i \vh_\vx (\vx^*, \xi_i)$ converges in probability to  $\bar\vh_\vx(\vx^*):=A $  as  $n\to \infty.$
	\end{itemize}
\end{rem}

\begin{thm}\label{mth-7}
	Consider
	algorithm \eqref{eq-algo-r} and assume Assumption
	{\rm\ref{R}} holds.
	Then $\{\vU^n(\cdot)\}$ converges weakly to the
		solutions of the following stochastic differential inclusion $($see Section {\rm\ref{sec:sdi}} for the definitions$)$
	\begin{equation}
	d\vU(t)\in\big[A\vU(t)+ T(\vU(t))\big]dt+\Sigma_1^{1/2}d\bar \vW(t),
	\end{equation}
	if {\rm \ref{R}(i)(a)} holds, and
	\begin{equation}
	d\vU(t)\in\big[(A+I/2)\vU(t)+ T(\vU(t))\big]dt+\Sigma_1^{1/2}d\bar \vW(t),
	\end{equation}
	if {\rm \ref{R}(i)(b)} holds, where
	$\bar \vW(t)$ is a $d$-dimensional standard Brownian motion.
\end{thm}

\begin{rem}{\rm
The main difficulties in deriving the result come from the lack of continuity of $\vb_n(\cdot)$ and the handling of the set-valued mappings, provided
the normalized noise terms converge (in distribution) to a Wiener process.
	Although we only state and prove the rate of convergence results for a simple algorithm,
	similar results for general algorithms  can
	also be obtained with
	modifications.
}\end{rem}

\begin{proof}[Proof of Theorem \ref{mth-7}.]
Define $\vW^n(\cdot)$ on $(-\infty,\infty)$ by
$$
 \vW^n(t)=\begin{cases}\disp
\sum_{i=n}^{m(t_n+t)-1}\sqrt {a_i}[\vh(\vx^*,\xi_i)-\bar\vh(\vx^*)]\text{ if }t\geq 0,\\
\disp -\sum_{i=m(t_n+t)}^{n-1}\sqrt{a_i}[\vh(\vx^*,\xi_i)-\bar\vh(\vx^*)]\text{ if }t\leq 0.
\end{cases}
$$
It is similar to \cite[Theorem 10.2.1]{KY03} (see also \cite{Yin91}) that $\{(\vU^n(\cdot),\vW^n(\cdot))\}$ is tight in $D^d[0,\infty)\times D^d(-\infty,\infty)$.
Hence, we can extract a convergent subsequence (still denoted by $\{(\vU^n(\cdot),\vW^n(\cdot))\}$) such that $\{(\vU^n(\cdot),\vW^n(\cdot))\}$ converges weakly to a limit, denoted by $(\vU(\cdot),\vW(\cdot))$.
First, Remark \ref{rem:weak-conv}
 yields that $\vW(t)$ is a Wiener process with covariance matrix $\Sigma_1 t$.
For simplicity of notation, we  also assume that the sequence $\{\vU_n\}$ is bounded
and suppress the truncation step
 (see e.g., \cite[Theorem 10.2.1]{KY03}.
  Note that the difficulty in proving Theorem \ref{mth-7} comes from the discontinuity of $\vb_n(\cdot)$ and the appearance of set-valued mapping $G(\cdot)$. However, this term is assumed to be bounded. Thus, we only need to use the truncated process to handle the smooth term $\vh(\cdot)$, which is the reason that the similar truncation step in \cite[Theorem 10.2.1]{KY03} is valid here.
To proceed, we work with
the case
\ref{R}(i)(b); the case \ref{R}(i)(a) can be handled similarly and is thus omitted.

To proceed, we have that
\begin{equation}\label{eq-un}
\begin{aligned}
\vU_{n+1}&=\Big(\frac {a_n}{a_{n+1}}\Big)^{1/2}\Big\{\vU_n+\sqrt{a_n}\big(\bar \vh(\vx^*)+\vb_n(\vX_n)+\vh_\vx (\vx^*,\xi_n)(\vX_n-\vx^*)+a_nO(|\vU_n|^2)\big)\\
&\hspace{2.5cm}  +\sqrt{a_n}[\vh(\vx^*, \xi_n)- \bar \vh(\vx^*)]\Big\}\\
&=\vU_n+\Big(\Big(\frac {a_n}{a_{n+1}}\Big)^{1/2}-1\Big)\vU_n
+\Big(\frac {a_n}{a_{n+1}}\Big)^{1/2}\Big\{a_n\vh_\vx(\vx^*,\xi_n)\vU_n
+a_n\vec v_n(\vU_n)\\
&\hspace{2.5cm} +a_n^{3/2}O(|\vU_n|^2)  +\sqrt{a_n}[\vh(\vx^*, \xi_n)- \bar \vh(\vx^*)]\Big\},
\end{aligned}
\end{equation}
where
$$\vec v_n(\vU_n):=\frac{\bar \vh(\vx^*)+\vb_n(\vX_n)}{\sqrt{a_n}}.$$


Let $\delta\in (0,1)$ be fixed and otherwise arbitrary.
Since $G(\cdot)$ is outer $T$-differentiable at $\vx^*$, there is a neighborhood $V$ of $\vx^*$ such that \eqref{eq-otd} holds, i.e.,
$$
G(\vx)\subset G(\vx^*)+T(\vx-\vx^*)+\delta |\vx-\vx^*|B\text{ for all }\vx\in V.
$$
Since $\vX_n$ tends to $\vx^*$ w.p.1 (Assumption \ref{R}(ii)(a)),
for $n$ large enough, we have that
\begin{equation}
\begin{aligned}
\vh(\vx^*)+\vb_n(\vX_n)&\in \vh(\vx^*)+G(\vX_n)\\
&\subset \vh(\vx^*)+ G(\vx^*)+T(\vX_n-\vx^*)+\delta |\vX_n-\vx^*|B\\
&\subset T(\vX_n-\vx^*)+\delta |\vX_n-\vx^*|B.
\end{aligned}
\end{equation}
As a consequence,
\begin{equation*}
\vec v_n(\vU_n)\in \frac 1{\sqrt{a_n}}T(\vX_n-\vx^*)+\delta \frac{|\vX_n-\vx^*|}{\sqrt{a_n}} B=T(\vU_n)+\delta |\vU_n|B,
\end{equation*}
where we have used the fact that $T$ is positively homogeneous (in Assumption \ref{R}(iv)).
Let
\begin{equation*}
M_\delta(\vx):=T(\vx)+\delta |\vx|\bar {B}.
\end{equation*}
Then one has
\begin{equation}\label{27-eq0}
\vec v_n(\vU_n)\in M_\delta(\vU_n).
\end{equation}
Hence, from \eqref{eq-un}, \eqref{27-eq0}, $\left(a_n/a_{n+1}\right)^{1/2}=1+o(a_n)$, and $\sum_{i=n}^{m(t_n+ t)-1}a_i \vh_\vx (\vx^*, \xi_i)$ converges in probability to  $\bar\vh_\vx(\vx^*):=A $ in Remark \ref{rem:weak-conv}, by the same argument as in the proofs of previous theorems,
we  obtain that for $n$ large enough
\begin{equation}\label{27-eq10}
 \vU^n(t)- \vU^n(s)\in\int_s^t\left( A\vU^n(r)+M_{\delta}(\vU^n(r))\right)dr+\vy_n(t)-\vy_n(s)+\vW^n(t)-\vW^n(s),
\end{equation}
where $\vy_n(\cdot)$ is some process converging to  zero  and $\vW^n(\cdot)$ converges to $\vW(\cdot)$ weakly.
Using the Skorohod representation theorem \cite[Chapter 3, Theorem 1.8]{EK86} but without changing notation, we can assume $\vy_n(\cdot)+\vW^n(\cdot)$ converges to $\vW(\cdot)$ w.p.1.
Let $\delta_1\in(0,\delta)$ (depending on $\delta$) be such that
\begin{equation}\label{27-16-10-1}
T(\vx+\delta_1B)\subset T(\vx)+\delta B.
\end{equation}
Such $\delta_1$
always exists since $T$ is continuous.
Because of the convergence of $\vy_n(\cdot)+\vW^n(\cdot)$ to $\vW(\cdot)$, we have that on bounded intervals, for $n$ large, $|\vy_n(\cdot)+\vW^n(\cdot)-\vW(\cdot)|\leq \delta_1/2$.
As a consequence, if we let
$$\bar \vU^n(\cdot):=\vU^n(\cdot)-\vy_n(\cdot)-\vW^n(\cdot)+\vW(\cdot),$$
then on bounded intervals, for $n$ large, $|\bar \vU^n(\cdot)-\vU^n(\cdot)|\leq \delta_1/2$,
which together with \eqref{27-eq10} implies that for $s,t$ in bounded intervals, for $n$ large,
\begin{equation}\label{27-16-10-2}
\begin{aligned}
\bar \vU^n(t)-\bar \vU^n(s)&\in\int_s^t\left( A\vU^n(r)+M_\delta(\vU^n(r))\right)dr+\vW(t)-\vW(s)\\
&\subset \int_s^t\left( A\bar \vU^n(r)+\bar M_\delta(\bar \vU^n(r))\right)dr+\vW(t)-\vW(s),
\end{aligned}
\end{equation}
where
\begin{equation}\label{27-eqm}
\bar M_\delta(\vx):=T(\vx)+\delta \left(2+\|A\|+|\vx|\right)\bar {B},
\end{equation}
and
in \eqref{27-16-10-2},
 we have used the following facts:
$$
|A\vx-A\vy|\leq \|A\||\vx-\vy|, \|A\|\text{ is the sup-norm of } A,
$$
and if $|\vx-\vy|< \delta_1$ then
$
T(\vy)\subset T(\vx)+\delta B,
$
due to \eqref{27-16-10-1}.
It indicates that
on bounded intervals,
$
\bar \vU^n(\cdot)
$ (for $n$ large)
is a solution of
\begin{equation}\label{1111-barU}
d\bar \vU^n(t)\in \big[A\bar \vU^n(t)+\bar M_\delta(\bar \vU^n(t))\big]dt+\Sigma_1^{1/2}d\bar \vW(t),
\end{equation}
where $\bar \vW(t)$ is a $d$-dimensional standard Wiener process.

To proceed, we  state in Lemma \ref{prof-27-lem1} sufficient conditions for the weak compactness of the set of solutions of stochastic differential inclusions by Kisielewicz in \cite{Kis01}(see also \cite{Kis05,Kis13}). Then in Lemma \ref{prof-27-lem2}, we verify these conditions.

\begin{lem}\label{prof-27-lem1} $($see \cite[Theorem 12]{Kis01}$)$.
Consider the stochastic differential inclusion
\begin{equation}\label{1111-FG}
d\vX(t)\in F_1(\vX(t))dt+F_2(\vX(t))d\bar \vW(t).
\end{equation}
Assume that set-valued mappings $F_1:\R^d\to 2^{\R^d}$, $F_2:\R^d\to 2^{\R^{d\times d}}$ are measurable and bounded, and have convex values, where $F_2$ has convex values
 in the sense of that
 $\{gg^\top: g\in F_2(\vx)\}$ is convex for each $\vx\in \R^d$; and that $F_1,F_2$ are continuous $($see Section \ref{sec:set} for the definition$)$.
Then, for any initial distribution, the set of solutions to \eqref{1111-FG}
 is sequentially $($weakly$)$ closed with respect to the convergence in distribution.
\end{lem}

\begin{lem}\label{prof-27-lem2}
For each $\delta>0$, $\bar M_\delta(\cdot)$ is continuous, where $\bar M_\delta(\cdot)$ was defined in \eqref{27-eqm}.
\end{lem}

\begin{proof}
We  prove this lemma by using Lemma \ref{5.4-lem54}. Let $\vec p\in \R^d$ be arbitrary and consider the map $\sigma(\vec p,\bar M_\delta(\cdot))$, defined by $\sigma(\vec p,\bar M_\delta(\vx)):=\sup_{\vec a\in \bar M_\delta(\vx)} \vec p^\top \vec a$. We have
$$
\begin{aligned}
\sigma(\vec p,\bar M_\delta(\vx))
&=\sup_{\vec a_1\in T(\vx),\; \delta_2\in [0,\delta],\;\vec e \text { is the unit vector in }\R^d} \vec p^\top (\vec a_1+\delta_2(2+\|A\|+|\vx|)\vec e)\\
&=\sup_{\vec a_1\in T(\vx)}\vec p^\top \vec a_1+\delta(2+\|A\|+|\vec x|)|\vec p|\\
&=\sigma(\vec p,T(\vx))+\delta(2+\|A\|+|\vx|)|\vec p|.
\end{aligned}
$$
Since $T(\cdot)$ is continuous, $\sigma(\vec p,T(\cdot))$ is continuous. As a result, $\sigma(\vec p,\bar M_\delta(\cdot))$ is continuous and then, $\bar M_\delta(\cdot)$ is continuous.
\end{proof}

Since $\vU(\cdot)$ is the limit of $\vU^n(\cdot)$, it is also the limit of $\bar \vU^n(\cdot)$.
Hence, by Lemmas \ref{prof-27-lem1} and \ref{prof-27-lem2}, on bounded intervals, $\vU(\cdot)$ is such that
\begin{equation}\label{27-eq1}
d \vU(t)\in \big[A \vU(t)+\bar M_\delta(\vU(t))\big]dt+\Sigma_1^{1/2}d\bar \vW(t),\text{ for all }\delta>0.
\end{equation}
Because $\vU(\cdot)$ is a solution to \eqref{27-eq1}, by \cite[Lemma 1]{Kis01}, we deduce from \eqref{27-eq1} that for any bounded interval $[0,T_0]$ and for all $k\in\N$, there exists $f^k(\cdot)$ such that $f^k(\vx)\in \bar M_{1/k}(\vx)\;\forall \vx$ and for all $s<t\in[0,T_0]$,
$\vU(t)-\vU(s)+\vW(t)-\vW(s)=\int_s^t A\vU(r)dr+ \int_s^t f^k(\vU(r))dr,\text{ w.p.1}.
$
This yields that
\begin{equation}\label{27-eq2}
\vU(t)-\vU(s)+\vW(t)-\vW(s)=\int_s^t A\vU(r)dr+ \int_s^t f^k(\vU(r))dr,\;\forall k\in\N,\text{ w.p.1}.
\end{equation}
A consequence of \eqref{27-eq2} is that
\begin{equation}\label{27-eq2-11}
\vU(t)-\vU(s)+\vW(t)-\vW(s)-\int_s^t A\vU(r)dr\in \int_s^t
\bar M_{1/k}(\vU(r))dr,\;\forall k\in\N,\text{ w.p.1}.
\end{equation}
The $T(\vx)$ is non-empty, compact, and convex,
so is $\bar M_{1/k}(\vx)$.  It is readily seen that $\cap_{k\in\N}\bar M_{1/k}(\vx)=T(\vx), \forall \vx$.
Combining this fact together with \eqref{27-eq2-11} and Proposition \ref{lem-41}, we have that for all $s,t\in[0,T_0]$
$$
\vU(t)-\vU(s)+\vW(t)-\vW(s)-\int_s^t A\vU(r)dr\in \int_s^t
T(\vU(r))dr,\text{ w.p.1}.
$$
Therefore,
we have
$$
\vU(t)-\vU(s)\in \int_s^t \big[A\vU(r)+T(\vU(r))\big]dr+\int_s^t \Sigma_1^{1/2}d\bar \vW(r),\text{ w.p.1}.
$$
Equivalently, $\vU(\cdot)$ is a solution to
$$
d \vU(t)\in \big[A \vU(t)+T(\vU(t))\big]dt+\Sigma_1^{\frac 12}d\bar \vW(t).
$$
The proof is complete.
\end{proof}

\para{Tightness criteria of normalized sequence.}
In Assumption \ref{R}(ii), we assumed the tightness of the normalized sequence as a division of labor.
To end this section, we provide sufficient conditions for the tightness of sequence $\{\frac{\vX_n-\vx^*}{\sqrt{a_n}}\}$ for large $n$.
These conditions are essentially concerned with the stability of the limit point $\vx^*$.
We will obtain the tightness by adapting and modifying the perturbed Lyapunov functional method for differential inclusions.
Such a
method was first used in the treatment of partial differential
equations and stochastic analysis, and later on used for many different stochastic systems in \cite{KY03}.
Here, we modify this idea to treat our cases.
[The assumptions given below are not restrictive.
 In fact, in  many applications, $V(\vx)=|\vx|^2$ can be used as a simple but promising candidate, which is shown in Section \ref{sec:app}. In addition, locally quadratic Lyapunov functions (see \cite{KY03}) can also be considered.]
We state a proposition below. A sketch of the proof is relegated to the Appendix \ref{sec:provetight}.

\begin{prop}\label{prop:tight}
	Consider algorithm \eqref{eq-algo-r} with $\vec b_n(\vx)\in G(\vx),\forall n$, $G(\vx)$ is a set-valued mapping; and suppose $\vX_n$ is bounded and converges to $\vx^*$ w.p.1 and Assumption {\rm\ref{R}(iii)} holds. Assume that there is a function $V:\R^d\to \R$ such that
	\begin{itemize}
		\item  $V (\vx^*) = 0$, $V (\vx) > 0$ for each $\vx \in\R^d$, $\vx\neq \vx^*$,
		$V (\cdot)$ together with its partial derivatives up to the second order in $\vx$ is continuous;
		$|V_{\vx}(\vx)|^2 \leq K(1 + V (\vx))$, $V_{\vx\vx}(\cdot)$ is uniformly bounded;
		and $V(\vx)\geq c_0|\vx-\vx^*|^2+o(|\vx-\vx^*|^2)$ as $\vx\to\vx^*$ for some positive constant $c_0$; and
		\item there is a $\lambda > 0$ such that $\max\dot{\bar V}^{G+\bar {\bold h}}(\vx) \leq -\lambda V (\vx) \text { for }\vx\neq \vx^*$
	$($where $\dot{\bar V}^{G+\bar {\bold h}}$ is the set-valued derivative of $V$ with respect to the set $G+\bar {\bold h}$, see definition {\rm\ref{non-def-1}(iii)}$)$; and
		\item  for each $n$, each $\vx \in\R^d$ and each $\xi$,
		$|\vec b_n(\vx)|^2+|\bold h(\vx,\xi)|^2 \leq K(1 + V (\vx))$.
	\end{itemize}
Moreover, assume the sequence of step sizes $\{a_n\}$ satisfies either $a_n = \frac 1n$ and $\lambda> 1$,
or $a_n\to 0$, and for each $T > 0$
$$
\liminf_n\min_{n\geq i\geq m(t_n-T)}\frac{a_n}{a_i}=1,
$$
where $t_n,m(t)$ are defined at the beginning.
Then, there is an $N$ such that $\{\frac{\vX_n -
	\vx^*}{\sqrt{a_n}}; n\geq N\}$ is tight in $\R^d$.
\end{prop}

\section{Applications}\label{sec:app}
In this section, we  apply our  results developed in previous sections to a number of application examples.

\subsection{Stochastic Sub-gradient Descent}\label{ssec:sub}
We begin with the description under a deterministic setup.
Suppose that we aim to find the minimizers of a loss function $L(\vec w)$, i.e.,
$\text{argmin}_{\vec w\in\R^d} L(\vec w).
$
If $L(\vec w)$ is continuously differentiable with respect to $\vec w$, the minimizer $\vec w^*$ is the solution of the equation
$
\nabla_{\vec w} L(\vec w)=0.
$
In this case, we can find the optimum
by gradient descent algorithms
as usual.
However, if $L(\vec w)$ is only strictly convex and not differentiable, we cannot define the gradient $\nabla_{\vec w}L(\vec w)$.
Rather, we define its sub-gradient $\partial{L(\vec w)}$  as
$
\partial{L(\vec w)}:=\left\{\vec m\in\R^d: L(\vec y)\geq L(\w)+\vec m^\top (\vec y-\vec w),\;\forall \vec y\in\R^d\right\}.
$
Hence, the minimizer $\w^*$  satisfies
$
\vec 0\in \partial L(\w^*).
$
The algorithm for the minimization is of the form
$
\vec w_{n+1}=\w_n-a_n \vec g_n,
$
for some $\vec g_n\in \partial L(\w_n)$.
Assume that $L(\cdot)$ can be decomposed into two
components,
one
satisfies certain smooth conditions
and the other verifies
convexity.
Then we can assume
$
\partial L(\w)=\bar \vh(\w)+G(\w),
$
where $\bar \vh$ is a continuous function and $G(\w)$ is a set-valued mapping. Our objective is
to find the minimizer $\w^*$
satisfying
$\vec 0\in \partial L(\w^*).
$

When noisy observations or measurements are involved,
$\partial L(\w_n)$ is often not available. As a result,
we use $\tilde{\vec g}_n$, which is an unbiased or biased estimator of $\partial L(\w_n)$.
With noisy observations or measurements,
 we can write the estimator of $\tilde{\vec g}_n$ as
\begin{equation}\label{tilde-g}
\tilde{\vec g}_n=\vb_n(\w_n,\xi_n)+\vh(\w_n,\zeta_n)+\vh_0(\wdt \zeta_n)+\vbe_n,
\end{equation}
where
$\vh(\cdot,\cdot)$ is
 a smooth (w.r.t. $\w$) function that will be averaged out
to $\bar \vh$ (or a neighborhood of $\bar \vh$  if it involves some bias term that is not asymptotically negligible),
$\vb_n(\cdot,\cdot)$ is bounded
 with values belonging to a set-valued function $G(\cdot)$,
$\xi_n$, $\zeta_n$, $\wdt\zeta_n$ are the
noises,
 and $\vh_0(\wdt \zeta_n)$ and $\vbe_n$
can be either averaged out
  or asymptotically bounded
by $\eta$ when the bias cannot be ignored.

Using \eqref{tilde-g}, we construct the
algorithm
\begin{equation} \label{eq-app-algo}
\w_{n+1}=\w_n-a_n\big[\vb_n(\w_n,\xi_n)+\vh(\w_n,\zeta_n)+\vh_0(\wdt \zeta_n)+\vbe_n\big],
\end{equation}
or its projected version
\begin{equation} \label{eq-app-proj}
\begin{cases}
\tilde {\w}_{n+1}=\w_n-a_n\big[\vb_n(\w_n,\xi_n)+\vh(\w_n,\zeta_n)+\vh_0(\wdt \zeta_n)+\vbe_n\big],\\
\w_{n+1}=\Pi_H(\tilde\w_{n+1}).
\end{cases}
\end{equation}
Then, under our conditions,  in algorithms \eqref{eq-app-algo} or \eqref{eq-app-proj},
 $\w_n$ converges w.p.1
to the minimizer $\w^*$.
We can also obtain robustness and rates of convergence of these algorithms by applying Theorems \ref{mth-6} and \ref{mth-7}.
Our proposed conditions are mild and can be verified.
The assumptions in the noises are mild and can be verified by many common noise sequences such as
i.i.d. sequences, martingale difference sequences, mixing noise, etc.
Note also that
the boundedness of non-smooth term $\vb$ and local boundedness of smooth term $\vh$ are often clear if we use projection algorithms and/or the noise does not make the iterates blow-up.
Only conditions for stability (such as {\rm \ref{S2}, \ref{G2}, \ref{H2}}) need to be verified carefully. However, it is shown later that many algorithms in the literature satisfy these conditions.
Some specific examples (e.g., Lasso algorithm for high-dimensional statistics, and Pegasos algorithm in support vector machine (SVM) classification) will be studied next  and some numerical results will be given in Section \ref{sec:num}.

\begin{rem}\label{rem:sub}
Note that stochastic sub-gradient descent algorithms
 are used often in machine learning community to minimize  a loss function in online learning in which the loss function can often be non-smooth.
When the number of
data in training set
 is large, because computational cost using exact sub-gradient is expensive, sampling or mini-batching computations are needed.
However, there was  no unified
approach
to analyze the convergence of stochastic sub-gradient descent algorithms, neither was there effort for handling algorithms with non-smooth loss functions.
 Most existing studies are based purely on establishing a kind of ``contraction estimate" (in the sense of in expectation); see e.g., \cite{Grim19,ST09,SZ13,RNV09} and references therein.
For example, convergence in expectation was proved in
\cite{Grim19,RNV09} and references therein; or the convergence in probability and almost surely of the sequence
$\{\min_{1\leq k\leq n}\|\w_k\|\}_{n=1}^\infty$
 were obtained in
\cite{NL14}.
Our effort here is to provide a new approach in analyzing the convergence, rates of convergence, robustness of stochastic sub-gradient algorithms, and other algorithms in non-smooth optimization by characterizing their behaviors using dynamical systems generated from differential inclusions and stochastic differential inclusions.
As a direct application of our results, if the corresponding differential inclusion has the minimizer as a globally asymptotically stable point, then we can obtain the almost surely convergence of the algorithm to the minimizer, which recover and/or improve the convergence results in \cite{Grim19,NL14,RNV09} and references therein.
The globally asymptotic stability can be verified by the use of a novel (and effective) Lyapunov functional method, which was presented
in Section \ref{sec:conv}.
The rates of convergence, robustness can also be deduced from our results.
\end{rem}

\begin{rem}
	Some other  variants of stochastic subgradient or gradient descent algorithms for non-smooth and/or non-convex optimization are studied widely recently including incremental sub-gradient descent  \cite{HP09,Kiw04}, proximal algorithms and stochastic proximal algorithms \cite{Nit14}, perturbed proximal primal dual algorithm \cite{HH19}, smoothing methods \cite{Che12}, gradient sampling methods \cite{BCL20}, among others.
Nevertheless, the central issue is the handling of set-valued mappings and nonsmooth loss functions.
Although we will not dwell on each of such algorithms, using our results we can treat such algorithms and obtain respective convergence results.
\end{rem}

\begin{rem}
In the next two sections, we present how our results can be applied to study algorithms in $L^1$-norm penalized (regularized) minimization and support vector machine (SVM) classification.
We will only focus on verifying the stability conditions since assumptions in the noises are mild and can be verified by many common noise sequences such as
i.i.d. sequences, martingale difference sequences, mixing noise, etc.
It is worth noting
that although we will not state explicitly the results for algorithms in $L^1$-norm regularized minimization in Section \ref{sec:app2} and SVM classification problem in Section \ref{sec:app3}, our results on convergence  (Theorems \ref{mth-2}, \ref{mth-4}, and \ref{mth-5}), robustness (Theorem \ref{mth-6}), and rates of convergence (Theorem \ref{mth-7}) hold for these algorithms.
These results recover, improve, and further the state-of-art development for 
Lasso and SVM algorithms.
\end{rem}

\subsection{$L^1$-norm Penalized (Regularized) Minimization: Lasso Algorithms, Least Absolute Deviation (LDA) Estimators}
\label{sec:app2}

We consider stochastic algorithms for minimizing loss functions containing $L^1$-norm, by
providing explicit computations for Lasso algorithm since other cases are similar.
Let us start with the following optimization problem.
Given a sequence of i.i.d. random variables $\{\vx_n,y_n\}$,
with $\vx_n\in\R^d,y\in\R$, we wish to find the weight vector $\w$ so that $\vx_n^\top \w$ best matches $y_n$ in the sense
$\E\|\vx^\top_n\w-y_n\|^2$ is minimized with the constraint $\sum_{i=1}^d|\text{w}_i|=0$, which can be recast into the following problem:
\begin{equation}\label{eq-Lw2}
\text{argmin}_{\w\in\R^d} L(\w),\quad
L(\w)=
\frac 12\E\|\vx^\top_n\w-y_n\|^2+\lambda \sum_{i=1}^d|\text{w}_i|
.
\end{equation}
Alternatively, we are trying
to find $\w^*$ such that
$
\vec 0\in \bar \vh(\w^*)+G(\w^*),
$
where
\begin{equation}\label{eq-HHH}
\barray\ad
\bar \vh(\w)=-\frac 12\nabla_{\w}\E\|\vx^\top_n\w-y_n\|^2,\\
\ad G(\w)=\mathcal K(w_1)\times\dots\times \mathcal K(w_d),\text{ with }\w=(w_1,\dots,w_d)^\top,
\earray
\end{equation}
and
$\mathcal K(w_i)=\begin{cases}
\{-\lambda\}\text{ if }w_i>0,\\
[-\lambda,\lambda]\text{ if }w_i=0,\\
\{\lambda\}\text{ if }w_i<0.
\end{cases}$
A stochastic algorithm  can be constructed as
\begin{equation}\label{ex-4.2-1}
\w_{n+1}=\w_n + a_n(y_n-\w^\top \vec \vx_n)\vec \vx_n+a_ng_n(\w_n),
\end{equation}
where
$g_n(\w_n)\in G(\w_n)$ with $G(\cdot)$ defined in \eqref{eq-HHH};
and the projection algorithm can be written as
\begin{equation}\label{ex-4.2-2}
\begin{cases}
\tilde\w_{n+1}=\w_n + a_n(y_n-\w^\top \vec \vx_n)\vec \vx_n+a_ng_n(\w_n),\\
\w_{n+1}=\Pi_H(\tilde \w_{n+1}),
\end{cases}
\end{equation}
with $H$ being a compact and convex set.
While the other assumptions are easily verified, the stability assumption needs to be checked carefully.
We  verify condition \ref{G2} for algorithm \eqref{ex-4.2-1} later in Proposition \ref{ex-prop-1}, whose proof is postponed to Section \ref{sec:prof3}.

Note that loss functions defined as the sum of the errors of prediction and the $L^1$-norm regularization are often used in dimension reduction problem in high dimensional statistics \cite{G15},
in which, the $L^1$-norm is used to penalize the dimension of subspace that we are trying to project onto.
Roughly,
$\sum_{i=1}^d|\text{w}_i|$ cannot be large causing $w_i$
to be small for all $i\in\{1,\dots,d\}$.
 If we use the squared norm,
all $w_i$ would bare
the same weight.
 If we use the absolute deviation,
 some ``less-informative" coordinates will be highlighted and leads to $w_i=0$ for such coordinates.
More intuitively,
in a two-dimensional case,
 from a geometric point of view, the unit ball in $L^1$-norm is of diamond shape with four vertices
  instead of a circle in $L^2$-norm
   so that the optimal value will  often be obtained on some axis.  For more intuition on Lasso algorithm as well as $L^1$-norm penalization, we refer to the work by  Tibshirani in \cite{Tib96}.

\begin{rem}
In practice, the above algorithms  may need to be modified
such as stochastic coordinate descent (SCD),
truncated gradient (TruncGrad), etc.,
to be more effective in real data and/or in the problem of inducing sparsity
 \cite{LLZ09}. The convergence of these modified algorithms can be obtained by applying our results with
modifications.
Here,
we only discuss a
 simple version of the algorithm.
  There are other algorithms, which minimize loss functions containing absolute norm such
as robust regression and least absolute deviation (LAD) with/without Lasso \cite{HH77,WLJ07}.	Algorithms \eqref{ex-4.2-1} and \eqref{ex-4.2-2} and their variants are widely applied by the machine learning community in  applications with a large-scale data set \cite{HH77,LLZ09,WLJ07}.
\end{rem}

\begin{prop}\label{ex-prop-1}
	Assume that $\E[\vx_n\vx^\top_n]$ is a positive definite matrix.
Let $G^*(\w)=\bar \vh(\w+\w^*)+G(\w+\w^*)$, where $\bar \vh(\cdot)$, $G(\cdot)$ are as in \eqref{eq-HHH}
and $V(\w)=|\w|^2$, $U(\w)=\sum_{i=1}^d w_i$. Then
$
\dot{\bar V}_{\{U\}}^{G^*}(\w)\leq -c_1|\w|^2,
$
where $c_1>0$ is the smallest eigenvalue of $\E[\vx_n\vx^\top_n]$.
\end{prop}

\subsection{Support Vector Machine (SVM) Classification}\label{sec:app3}
We first consider a stochastic optimization problem and then treat the problem of support vector machine (SVM) classification problem. The Pegasos algorithm will be introduced next.
Consider the following problem:
$
\text{minimize }L(\w):=\frac {\lambda}2 \|\w\|^2+\max\{0, 1-\E[y_n \w^\top\vec x_n]\},
$
where $(\vec \vx_n,y_n)$ is a sequence of i.i.d. random variables.
The stochastic version of sub-gradient descent algorithm for this problem is as follows
\begin{equation}\label{ex-4.3-1}
\w_{n+1}=\w_n - a_n\lambda\w_n+a_ng_n(\w_n,\vec \vx_n,y_n),
\end{equation}
where
$g_n (\w_n,\vec \vx_n,y_n)\in\partial \big(\text{-}\max\{0,1-y_n\w_n^\top \vec \vx_n\}\big),$
i.e.,
$
g_n (\w_n,\vec \vx_n,y_n)\in\begin{cases}
\{\vec 0\}\text{ if } y_n\w_n^\top \vec \vx_n>1,\\
\co\;\{\vec 0,y_n\vec \vx_n\}\text{ if }y_n\w_n^\top \vec \vx_n=1,\\
\{y_n\vec \vx_n\}\text{ if } y_n\w_n^\top \vec \vx_n<1;
\end{cases}
$
or as the following projection algorithm with the set $H$ being a compact and convex set,
\begin{equation}\label{ex-4.3-2}
\begin{cases}
\tilde\w_{n+1}=\w_n - a_n\lambda\w_n+a_ng_n(\w_n,\vec \vx_n,y_n),\\
\w_{n+1}=\Pi_H(\tilde\w_{n+1}).
\end{cases}
\end{equation}
Applying our results, the convergence to the optimal point, robustness, rates of convergence of algorithm \eqref{ex-4.3-1} (as well as algorithm \eqref{ex-4.3-2}) can be obtained under conditions in our setup.
We will verify the stability condition \ref{G2} for algorithm \eqref{ex-4.3-1} later in Proposition \ref{ex-prop-2}, whose proof is postponed to Section \ref{sec:prof3}.
The corresponding numerical example
is given in Example \ref{ex-num-2} in Section \ref{sec:num}.

Algorithms \eqref{ex-4.3-1} and \eqref{ex-4.3-2} can be recast into a form known as Pegasos algorithms and widely applied to
support vector machine (SVM) classification problem;
SVM is an effective and a popular classification learning tool  \cite{CS00}.
More intuition, motivation, and details of the hinge loss function as well as the above loss function in SVM classification can be found in \cite{CS00,SS08,SSC11} and references therein.
Algorithms \eqref{ex-4.3-1} and \eqref{ex-4.3-2} as well as their modified versions
were studied in \cite{SS08,SSC11} and references therein.
However,
the convergence was only given in high probability,  not w.p.1.
By applying
our
results,
the convergence w.p.1 is obtained.
The applications of algorithms \eqref{ex-4.3-1} and \eqref{ex-4.3-2} to classification problem in large-scale data can be found in \cite{SS08,SSC11} and references therein.

\begin{prop}\label{ex-prop-2}
Let
$\bar\vh(\w)=-\lambda\w$, and
$$
G_1(\w)=\begin{cases}
\{\vec 0\}\text{ if } \E[y_n\w^\top \vec \vx_n]>1,\\
\co\;\{\vec 0,\E[y_n\vec \vx_n]\}\text{ if }\E[y_n\w^\top \vec \vx_n]=1,\\
\{\E[y_n\vec \vx_n]\}\text{ if } \E[y_n\w^\top \vec \vx_n]<1,
\end{cases}
$$
and $G^*(\w)=\bar \vh(\w+\w^*)+G_1(\w+\w^*)$,
and $V(\w)=|\w|^2$, $U(\w)=\sum_{i=1}^d w_i$. Then
$
\dot{\bar V}_{\{U\}}^{G^*}(\w)\leq -\lambda\;|\w|^2.
$
\end{prop}

\subsection{Root Finding for Set-Valued Mappings}
\label{sec:app4}
In this section, we
 demonstrate
the effectiveness of our results in proving convergence of a stochastic approximation algorithm for set-valued mappings.
Assume that we need to find zero points of a set-valued mapping $G(\cdot)$,
i.e., find $\w^*$ such that
$
\vec 0\in G(\w^*),
$
where $G:\R^2\to\R^2$ is as follow
$
G(\vec w)=(-w_1+w_2+h(w_2),-w_1-w_2+h(w_1))
$
with $\w=(w_1,w_2)^\top$,
and $h(\cdot):\R\to 2^{\R}$ is defined as
$
h(w)=\begin{cases}
0\text{ if }w\neq 1,\\
[-1,1]\text{ if }w=1.
\end{cases}
$
 When only noisy observations or measurements are available, a
 stochastic approximation algorithm for root finding takes the form
\begin{equation}
\w_{n+1}=\w_n+a_n\left(f(\w_n,\xi_n)+\boldsymbol{\beta}_n\right),\;f(\w_n,\xi_n)\in G(\w_n),
\end{equation}
where $\{a_n\}$ is a sequence of step sizes and  $\{\boldsymbol{\beta}_n\}$ is a sequence of  $2$-dimensional i.i.d. random variables that are normally distributed with mean $\vec 0$ and identity covariance matrix.

We  compare our results with the results in \cite{BHS05} as well as other approaches in studying stability of differential inclusions for applications to stochastic approximations.
Under boundedness assumption of $\{\w_n\}$ (or using projection algorithm to convex and compact set), we obtain the limit points of $\{\w_n\}$ are contained in the set of chain-recurrent points of the limit system
$
\dot {\vec w}(t)\in G(\vec \w(t)).
$
Using results in \cite[Section 3 and 4]{BHS05}, to prove $\{\w_n\}$ converges to $\vec 0$, we need to construct a Lyapunov function $V$ such that
$
\nabla V(\w) g(\w)< 0$, $\forall g(\w)\in G(\w)$ for all $\w\neq \vec 0.$
Consider a candidate Lyapunov function $V(\w)=|\w|^2$.
Then
$$
\nabla V(\w)g(\w)=-|\w|^2+w_1g_1(w_2)+w_2g_2(w_1),\;\text{where }g_1(w_2)\in h(w_2), g_2(w_1)\in h(w_1).
$$
At $\w=(1,1)^\top$, one possibility is that
$
\nabla V((1,1))g((1,1))=(-\|\w\|^2+w_1+w_2)\big|_{\w=(1,1)^\top}=0.
$
So, we cannot guarantee the set $\{\vec 0\}$
to be  a globally stable and attracting set. That is,
we cannot prove that
$\{\w_n\}$
converges to $\vec 0$ using this Lyapunov function.
However, using our results, we can prove that $\{\w_n\}$ tends to $\vec 0$ w.p.1 by using the $\mathcal U$-generalized Lyapunov function corresponding to such a candidate function.
Condition {\rm \ref{G2}} needs to be verified, and it is stated
in the following proposition, whose proof is in Section \ref{sec:prof3}.
Roughly speaking, compared with the existing results in the literature, our approach allows one to ignore some ``less important" points (for example, the point $(1,1)$ above), that may make a (promising) candidate Lyapunov function not satisfy the conditions for the stability in the literature though they generally do not affect the stability of the systems.
Moreover,
in fact, our setting even allows $f(\w_n,\xi_n)$ to be in a neighbor of $G(\w_n)$ with (random) radius averaged out to $0$.
A numerical example
is given in Example \ref{ex-num-3} in Section \ref{sec:num}.

\begin{prop}\label{ex-prop-3}
Let $V(\w)=\|\w\|^2$ and
$$U(\w)=\max\{w_1-1,0\}-\min\{w_1+1,0\}+\max\{w_2-1,0\}-\min\{w_2+1,0\}.$$
 Then, one has
$
\dot{\bar V}^{G}_{\{U\}}(\w)\leq -\|\w\|^2,\;\forall\;\w.
$
\end{prop}

\subsection{Multistage Decision Making with Partial Observations}
Let $\mathcal E$ and $\mathcal B$ be measurable
spaces denoting the action space and the state space, respectively.
Suppose that  $\mathcal O \subset\R^d$ is a convex and
compact set denoting the outcome space.
At discrete times $n = 1, 2, \dots$, a decision maker chooses an action $e_n$ from $\mathcal E$ and observes an outcome $M(e_n, b_n)$, where $M : \mathcal E \times\mathcal B \to \mathcal O$ is a (measurable) function.
However, it is worth noting that the outcome is not always  observable in application but is only partially observed with noise.
So, the exact outcome $M(e_n,b_n)$
is not available for the decision maker,
but only noise corrupted outcome
$\wdt M(e_n,b_n,\xi_n)$ is available, where $\xi_n$ represents the noise.

Thus, we consider the following
multistage decision making model with partial observations: (1) the sequence $\{(e_n, b_n)\}_{n\geq 0}$ and $\{\xi_n\}_{n\geq 0}$ are random processes defined on some probability
space $(\Omega, \F, \PP)$ and adapted to the filtration $\{\F_n\}$ and the noise sequence $\{\xi_n\}$ satisfies
that for some $T>0$, each $\eps>0$, and each $(e,b)\in\mathcal E\times\mathcal B$,
\begin{equation}\label{eq-exp5-0}
	\lim_{n\to\infty}\PP\bigg\{\sup_{j\geq n}\max_{t\leq T}\Big|\sum_{i=m(t)}^{m(jT+t)-1}\frac 1{i+1}\big(M(e,b)-\wdt M(e,b;\xi_i)\big)\Big|\geq \eps\bigg\}=0,
\end{equation}
where $m(t):=\max\{n\in\N:\sum_{i=1}^n\frac 1i\leq t\}$;
(2) the action of the decision maker is independent of the environment if provided the past information $\{(e_1,b_1),\dots,(e_n,b_n)\}$, i.e.,
$\PP((e_{n+1}, b_{n+1})\in de\times db | \F_n) = \PP(e_{n+1} \in de| \F_n)\PP(b_{n+1} \in db | \F_n)$;
(3) the decision maker
records only the cumulative average of the past (partially observed) outcomes,
\begin{equation}\label{eq-exp5-1}
\vX_n =\frac 1n\sum_{i=1}^n\wdt M(e_i, b_i;\xi_i);
\end{equation}
 (4) her/his decisions are based on this average, i.e.,
$
\PP(e_{n+1} \in de | \F_n) = Q_{\vX_n} (de),
$
where for each $\vx \in\mathcal O$, $Q_{\vx}(\cdot)$ is a probability measure (in $\mathcal E$), and for each measurable set $A\subset\mathcal E$, the map: $\vx \in \mathcal O\to Q_{\vx}(A) \in [0, 1]$ is measurable. The family $Q = \{Q_{\vx}:\vx\in\mathcal O\}$ is
termed a strategy for the decision maker.

\begin{deff}
{\bf (Blackwell's approachability)} A set $E\subset \mathcal O$ is said to be
approachable if there exists a strategy $Q$ such that $\vX_n\to E$ w.p.1.
\end{deff}

Directed calculations show that
\begin{equation}
\vX_{n+1}=\vX_n +\frac 1{n+1}\Big(-\vX_n+\wdt M(e_{n+1},b_{n+1};\xi_{n+1})\Big).
\end{equation}
For each $\vx \in\mathcal O$, let
$
G_1(\vx) = \big\{\int_{\mathcal E\times\mathcal B}
M(e,b)Q_{\vx}(de)\nu(db) : \nu\in\mathcal P(\mathcal B)\big\},
$
where $\mathcal P(\mathcal B)$ is the set of probability measures over $\mathcal B$.
Define
$
G(\vx)=-\vx+\co\;G_1(\Pi_{\mathcal O}(\vx)),
$
where
$\Pi_{\mathcal O}(\cdot)$ is the (orthogonal) projection (onto $\mathcal O$) operator.
Applying our results (Theorems \ref{mth-3}, \ref{mth-6} and \ref{mth-7}), we obtain following results.

\begin{thm}
	Under the above settings, the following claims hold.
	\begin{itemize}
		\item[]{\rm(1)} The limit of any convergent subsequence of the shifted sequence of linear continuous time interpolated processes of
\eqref{eq-exp5-1} is a solution of the following differential inclusion w.p.1
		\begin{equation}\label{eq-exp5-2}
		\dot{\vX}(t)\in G(\vX(t)).
		\end{equation}
		\item[]{\rm (2)} If there is a strategy $Q$ such that $E$ is a globally asymptotically stable set of differential inclusion \eqref{eq-exp5-2}, then $E$ is approachable.
		\item[]{\rm(3)} If there exists a strategy $Q$ such that $E=\{\vx^*\}$ is a unique approachable set, then under further technical conditions $($as in Theorem {\rm\ref{mth-7}}$)$, the limit processes of convergent subsequences of shifted interpolated processes generated by normalized sequence $\frac{\vX_n-\vx^*}{\sqrt{n}}$ converges weakly to solutions of a stochastic differential inclusion.
		\item[]{\rm (4)} If the ``convergence to $0$" condition \eqref{eq-exp5-0} is relaxed as $|M(e,b)-\wdt M(e,b,\xi)|<\eta, \forall e,b,\xi,$ w.p.1, then the conclusions {\rm(1)} and {\rm(2)} still hold with $G$ in \eqref{eq-exp5-2} being replaced by its neighbor with radius $\eta$. Moreover, if $E_\eta$ is a globally asymptotically stable set of the corresponding $($limit$)$ differential inclusions $($and thus, is a approachable set$)$, then there is a $($deterministic$)$ non-decreasing function $\phi(\cdot)$ satisfying $\lim_{t\to 0}\phi(t)=0$ such that $\mathrm{distance}(E_\eta,E)\leq \phi(\eta)$.
	\end{itemize}
\end{thm}

\begin{rem}
{\rm	The studies on the stability of differential inclusions can be found in Appendix \ref{sec:63-asy}
	(see also \cite{BHS05} and references therein).
	[As was noted, in some cases (for example, as in Section \ref{sec:app4}), the $\mathcal U$-generalized Lyapunov condition presented in this work is more effective than the stability conditions counterpart in existing results.]
	Multistage decision making models (without partial observations) was considered in \cite{BHS05}.
	In this application, we allow the outcome to be partially observed under noise by the decision maker.
	In addition, we characterize the limit processes as solutions rather than perturbed solutions of the limit differential inclusion, and we also obtain results of rates of convergence and robustness.
	This example can be further generalized to treat
 other criteria such as overtaking, bias,
and other so-called advanced criteria of optimality, as well as other systems such as switching dynamical systems. We refer the reader to \cite{JY13} and references therein.
Some other applications to Markov decision process using stochastic approximation can be found in \cite{PL12} and references therein.}
\end{rem}

\subsection{Proof of Theorems in Section \ref{sec:app}}\label{sec:prof3}
\begin{proof}[Proof of Proposition \ref{ex-prop-1}]
The Clarke gradient of $U(\w)$ is given by $\partial U(\w)=(1,\dots,1)^\top$.
As a consequence,
$$
\tilde {G}^*_{\{U\}}(\w)=G^*(\w).
$$
Moreover, $V$ is continuously differentiable, $\partial V(\w)=(2w_1,\dots,2w_2)^\top$.
Therefore, the $\{U\}$-generalized derivative of $V$ in the direction $G$ is given by
\begin{equation}\label{eq-ex-0}
\dot{\bar V}_{\{U\}}^{G^*}(\w)=\max_{\vec q\in \tilde G^*_{\{U\}}(\w)} (\partial V(\w))^\top\vec q=\max_{\vec q\in \tilde G^*_{\{U\}}(\w)} 2\w^\top\vec q.
\end{equation}
Noting that for any $\vec q\in \tilde G^*_{\{U\}}(\w)$,
$$
\vec q=\bar \vh(\w+\w^*)+\bar {\vec q},\text{ for }\bar {\vec q}\in G(\w+\w^*),
$$
and hence,
\begin{equation}\label{eq-ex-3}
\w^\top \vec q=-\w^\top\E[\vx_n\vx_n^\top]\w+\w^\top\big[\E\vx_n(\vx_n^\top\w^*-y)+\bar {\vec q} \big].
\end{equation}
Since $\w^*$ is the minimizer, one has
$$
\vec 0\in \bar \vh(\w^*)+G(\w^*).
$$
In particular,
$$
\vec 0\in -\E \vx_n(\vx_n^\top\w^*-y)+ G(\w^*).
$$
Hence, it is equivalent to
\begin{equation}\label{eq-ex-4}
\E \vx_n(\vx_n^\top\w^*-y)\in -G(\w^*).
\end{equation}

\begin{lem}\label{lem-ex-1}
For any $w_i$, $k\in -\mathcal K(w_i^*)$, $\bar q_i\in \mathcal K(w_i+w_i^*)$,
\begin{equation}\label{eq-ex-1}
w_i(k+\bar q_i)\leq 0.
\end{equation}
As a consequence, for all $\w\in\R^d$, one has
\begin{equation}\label{ine-w}
\w^\top \big(-G(\w^*)+G(\w+\w^*)\big)\leq 0,
\end{equation}
where, \eqref{ine-w}
 is understood as
$$
\w^\top \big(\vec k+\bar {\vec q}\big)\leq 0\text{ for all }
\vec k\in -G(\w^*), \;\bar{\vec q}\in G(\w+\w^*).
$$
\end{lem}

\begin{proof} Three cases are considered.

Case 1: $w_i^*=0$. \eqref{eq-ex-1} is equivalent to
$$
w_i (k+\bar q_i)\leq 0,\text{ for all }k\in [-\lambda,\lambda]\;,\;\bar q_i\in \mathcal K(w_i).
$$
If $w_i=0$, it is obvious. If $w_i>0$ then $\bar q_i=-\lambda$ and $k+\bar q_i\leq 0$ for all $k\in[-\lambda,\lambda]$. If $w_i<0$ then $\bar q_i=\lambda$ and $k+\bar q_i\geq 0$ for all $k\in[-\lambda,\lambda]$.

Case 2: $w_i^*>0$. \eqref{eq-ex-1} is equivalent to
\begin{equation}\label{eq-ex-2}
w_i (\lambda+\bar q_i)\leq 0,\text{ for all }\bar q_i\in \mathcal K(w_i+w_i^*).
\end{equation}
Since $\lambda+\bar q_i\geq 0$ for all $\bar q_i\in \mathcal K(w_i+w_i^*)$, if $w_i\leq 0$, \eqref{eq-ex-2} is clear. If $w_i>0$ then $\lambda+\bar q_i=0$ (due to $\mathcal K(w_i+w_i^*)=\{-\lambda\}$) and \eqref{eq-ex-2} holds.

Case 3: $w_i^*<0$. This case is similar to case 2. The proof of the lemma is complete.
\end{proof}

Combining \eqref{eq-ex-0}, \eqref{eq-ex-3}, \eqref{eq-ex-4}, and Lemma \ref{lem-ex-1},
 we obtain that
$$
\dot{\bar V}_{\{U\}}^{G^*}(\w)\leq -\w^\top\E[\vx_n\vx_n^\top]\w.
$$
Since $\E[\vx_n\vx_n^\top]$ is positive definite and by Rayleigh's inequality (see e.g., \cite[Chapter 3]{CZ14}), one has
$$
\w^\top\E[\vx_n\vx_n^\top]\w\geq c_1 \|\w\|^2,
$$
where $c_1>0$ is the smallest eigenvalue of $\E[\vx_n\vx_n^\top]$.
Therefore, the proposition is proved.
\end{proof}

\begin{rem}
In practice,
to guarantee the boundedness of $\w_n$,
we can use a projection algorithm
 with  a hyper-rectangle $H:=\{\w\in\R^d:-h\leq w_i\leq h,\;\forall i\}$ with $h>\lambda$ being sufficiently large.
In this case, the proof of Proposition \ref{ex-prop-1} for the projection case is similar.
Moreover, the above proof can be simplified
by applying Theorem \ref{mth-2} (in which, we only need to verify condition \ref{S2} instead of \ref{G2}) and using
$
G(\w)=\mathcal K[-\lambda\;\text{sign}](\w),
$
where $\mathcal K$ is the Krasovskii operator and $-\lambda\;\text{sign}(\w)=(-\lambda\;\text{sign}(w_1),\dots,-\lambda\;\text{sign}(w_d))$.
However, in general, given a set-valued mapping $G(\cdot)$, we may not know explicitly $f(\cdot)$ (if it exists) satisfying $G(\cdot)=\mathcal K[f](\cdot)$.
That is the reason in the proof, we only treat
$G(\cdot)$ as a general set-valued mapping, not the Krasovskii operator of some vector-valued function.
\end{rem}

\begin{proof}[Proof of Proposition \ref{ex-prop-2}]
Similar to the proof of Proposition \ref{ex-prop-1},
the Clarke gradient of $U(\w)$ is given by $\partial U(\w)=(1,\dots,1)^\top$
and then
$
\tilde {G}^*_{\{U\}}(\w)=G^*(\w).
$
Moreover, $V$ is continuously differentiable, $\partial V(\w)=(2w_1,\dots,2w_2)^\top=2\w$. Hence, the $\{U\}$-generalized derivative of $V$ in direction $F$ is given by
\begin{equation}\label{eq-ex3-1}
\dot{\bar V}_{\{U\}}^{G^*}(\w)=\max_{\vec q\in \tilde G^*_{\{U\}}(\w)} (\partial V(\w))^\top\vec q.
\end{equation}
Let $\vec q\in G^*(\w)$ be arbitrary, then
\begin{equation}\label{eq-ex3-2}
\vec q=-\lambda \w-\lambda \w^*+\bar{\vec q},\;\bar{\vec q}\in G_1(\w+\w^*).
\end{equation}
Since $\vec 0\in -\lambda \w^*+g_1(\w^*)$,
$-\lambda \w^*\in -g_1(\w^*)$. Therefore, we obtain from \eqref{eq-ex3-2} that
\begin{equation}\label{eq-ex3-3}
\vec q=-\lambda \w+\vec m+\bar {\vec q},\;\bar {\vec q}\in G_1(\w+\w^*),\;\text{(for some) }\vec m\in -G_1(\w^*).
\end{equation}
If we can prove
\begin{equation}\label{eq-ex3-4}
\w^\top\left[\vec m+\bar{\vec q}\right]\leq 0\text{ for all }\vec m\in -G_1(\w^*),\;\bar{\vec q}\in G_1(\w+\w^*),
\end{equation}
then combing \eqref{eq-ex3-1}, \eqref{eq-ex3-3}, and \eqref{eq-ex3-4}, one has
$
\dot{\bar V}_{\{U\}}^{G^*}(\w)<-\lambda\;|\w|^2.
$
Now we prove \eqref{eq-ex3-4}.
Three cases are considered next.

Case 1: $\E[y_n(\w^*)^\top \vec x_n]=1$. So, $\vec m\in -G_1(\w^*)=-\co\;\{\E[y_n\vec x_n],\vec 0\}$ and then $\vec m=-m\E[y_n\vec x_n]$ for some $m\in[0,1]$.
If $\E[y_n\w^\top\vec x_n]= 0$ then
$ G_1(\w+\w^*)=\co\;\{\E[y_n\vec x_n],\vec 0\}$.
As a consequence,
$\bar {\vec q}=\bar q \E[y_n\vec x_n]$, for some $\bar q\in [0,1]$.
Therefore,
 $\w^\top\left[\vec m^k+\bar{\vec q}^k\right]=(-m+\bar q)\E[y_n\w^\top\vec x_n]=0$ and \eqref{eq-ex3-4} is clear.
 If $\E[y_n\w^\top\vec x_n]>0$, then $G_1(\w+\w^*)=\{\vec 0\}$
and thus,
$\w^\top\left[\vec m^k+\bar{\vec q}^k\right]=-m\E[y_n\w^\top\vec x_n]\leq 0$
 and \eqref{eq-ex3-4} holds.
On the other hand, if $\E[y_n\w^\top \vec x_n]<0$,  $G_1(\w+\w^*)=\{\E[y_n\vec x_n]\}$
and thus,
$\w^\top\left[\vec m^k+\bar{\vec q}^k\right]=(-m+1)\E[y_n\w^\top\vec x_n]\leq 0$
and \eqref{eq-ex3-3} is satisfied.

Case 2: $\E[y_n(\w^*)^\top\vec x_n]>1$. Then $\vec m=\vec 0$ and so, for all $\bar{\vec q}\in G_1(\w+\w^*)$, $\vec m+\bar{\vec q}=\bar q\E[y\vec x]$, for some $\bar q\in [0,1]$.
As a result,
 \eqref{eq-ex3-3} holds if $\E[y_n\w^\top\vec x_n]\leq 0$.
Otherwise, if $\E[y_n\w^\top\vec x_n]>0$, then $G_1(\w+\w^*)=\{\vec 0\}$, so $\vec m+\bar{\vec q}=\vec 0$ and \eqref{eq-ex3-3} still holds.

Case 3: $\E[y_n(\w^*)^\top\vec x_n]<1$. This case is similar to case 2.
\end{proof}

\begin{proof}[Proof of Proposition \ref{ex-prop-3}]
Since $U$ is convex, it is regular. The Clarke gradient $\partial U$ of $U$ is given by
$$
\partial U(\vx)=\begin{cases}
\{(s(w_1),s(w_2))\} \text{ if }|w_1|\neq 1 \text{ and }|w_2|\neq 1,\\
\co\;\{0,\text{sign}(w_1)\}\times\{s(w_2)\} \text{ if }|w_1|=1 \text{ and }|w_2|\neq 1,\\
\{s(w_1)\}\times\co\;\{0,\text{sign}(w_2)\} \text{ if }|w_1|\neq 1 \text{ and }|w_2|= 1,\\
\co\;\{0,\text{sign}(w_1)\}\times \co\;\{0,\text{sign}(w_2)\} \text{ if }|w_1|= 1 \text{ and }|w_2|= 1,
\end{cases}
$$
where
$$
s(w):=\begin{cases}
0\text{ if }-1<w<1,\\
\text{sign}(w)\text{ otherwise}.
\end{cases}
$$
It is noted that
$$
G(\w)=\begin{cases}
\{(-w_1+w_2,-w_1-w_2)\}\text{ if }w_1\neq 1\text{ and }w_2\neq 1,\\
\{(-w_1+w_2,-w_1-w_2)\}+[-1,1]\times \{0\}\text{ if }w_1= 1\text{ and }w_2\neq 1,\\
\{(-w_1+w_2,-w_1-w_2)\}+\{0\}\times [-1,1]\text{ if }w_1\neq 1\text{ and }w_2= 1,\\
\{(-w_1+w_2,-w_1-w_2)\}+[-1,1]\times [-1,1]\text{ if }w_1= 1\text{ and }w_2= 1.
\end{cases}
$$
Therefore, direct calculation yields that
$$
M^G_{\{U\}}(\w)=\begin{cases}
G(\w)\text{ if }|w_1|\neq 1\text{ and }|w_2|\neq 1,\\
\emptyset \text{ otherwise}.
\end{cases}
$$
Equivalently, one has
$$
M^G_{\{U\}}(\w)=\begin{cases}
\{(-w_1+w_2,-w_1-w_2)\}\text{ if }|w_1|\neq 1\text{ and }|w_2|\neq 1,\\
\emptyset \text{ otherwise}.
\end{cases}
$$
We have $\tilde G_{\{U\}}(\w)=M_{\{U\}}^G(\w)$; and $\partial V(\w)=2(w_1,w_2)^\top$.
Hence, the $\{U\}$-generalized derivative of $V$ in direction $G$ is given by
$$
\begin{aligned}
\dot{\bar V}_{\{U\}}^G(\w)=& \max_{\vec q\in \tilde G_{\{U\}}(\vx)} \partial V(\w)^\top\vec q\\
=&
\begin{cases}
-2\|\w\|^2\text{ if }|w_1|\neq 1\text{ and }|w_2|\neq 1,\\
-\infty\text{ otherwise}.
\end{cases}
\end{aligned}
$$
As a result, the proposition is proved.

\end{proof}

\subsection{Numerical Examples}\label{sec:num}
In this section, we provide some numerical examples to illustrate our findings.

\begin{example}\label{ex-num-1}
This example demonstrates
the results in Section \ref{sec:app2} as well as Theorem \ref{mth-6}.
We are concerned with the following optimization problem:
$
\text{Find }w^*\text{ to minimize }\E (h(w,\xi_n)+\beta_n)+\lambda\|w\|.
$
For simplicity,
we consider a real-valued function with
$h(w,\xi_n)=\frac 12(w+\xi_n-1)^2$, $\lambda=0.7$,  $\{\xi_n\}$ is a sequence of random variables with mean 0 and finite variance,
and $\{\beta_n\}$ is a sequence of random variables (assumed to be independent for simplicity) satisfying variance of $\beta_n \leq c_n$. We vary $c_n$ to see the effect of the bias on the convergence of the algorithm.
The problem becomes:
$
\text{find minimizer }w^*\text{ of }
\E (h(w,\xi_n)+\beta_n) + \lambda|w|=\frac 12(w-1)^2+0.7|w|.
$
Direct calculation shows that the true value is $w^*=0.3$.

Suppose that
only the noisy observations or measurements
$h(w_n,\xi_n)+\beta_n)$ are available, we can construct a recursive algorithm
\begin{equation}\label{num-eq-1}
\begin{aligned}
& w_{n+1}=w_n+a_n\left[(1+\xi_n-w_n)+\beta_n+g(w_n)\right].
\end{aligned}
\end{equation}
In  each iteration, we choose
$
g(w)\in\begin{cases}
\{-1\}\text{ if }w>0,\\
[-1,1]\text{ if }w=0,\\
\{1\}\text{ if }w<0.
\end{cases}
$
The numerical results are given in Table \ref{tab-1}.
\begin{center}
\captionof{table}{Numerical results of algorithm \ref{num-eq-1}}\label{tab-1}
\scalebox{1}{
\resizebox{0.9\textwidth}{!}{
\begin{tabular}{ |p{4cm}||p{1.5cm}|p{1.5cm}|p{1.5cm}|p{1cm}|p{1cm}|p{1cm}|p{1cm}|  }

\hline
 \multicolumn{2}{|c|}{Examples}&ex1&ex2&ex3&ex4&ex5&ex6\\
 \hline
num. of iterations & $\hat n$ &$10^3$  &$10^3$ &$10^3$ & $10^3$&  $10^3$&$10^3$\\
 \hline
num. of repeat& &$10^3$  &$10^3$ &$10^3$ & $10^3$&$10^3$&$10^3$\\
\hline
initial value&$w_0$& 5&50&5   &5 &5&5\\
 \hline
variance $c_n$ of the bias &$c_n$  & $1/n$& $1/n$ &$n^{-0.5}$ &1&10&10\\
 \hline
step sizes&$a_n$   &$1/\sqrt{n}$ &$1/\sqrt{n}$    &$1/\sqrt{n}$  &$1/\sqrt{n}$&$1/\sqrt{n}$&1/n\\

\hline
 error&$|\hat w-w^*|$&   $10^{-3}$&$10^{-3}$&$10^{-3}$  & 0.01   &0.37&0.12\\
 \hline
\end{tabular}
}
}
\end{center}

In Table \ref{tab-1}, columns ``ex1" and ``ex2" show the minimizer is globally attractive.
Columns ``ex1", ``ex3", and ``ex4" show the dependence of the convergence rate on how fast the bias
going to $0$.
If $c_n$ is large,
the algorithm may not converge fast enough to the true minimizer,
but just in its neighborhood, which is shown in columns ``ex5", and ``ex6".

The relation between $c_n$ and the mean of  the error (of approximated value) after repeating algorithm \ref{num-eq-1} (with $\hat n=1000$ iterations for each) is shown in Figure \ref{pic-ex1} (the left one).
This shows the numerical results for the theoretical one in Theorem \ref{mth-6}, i.e.,
the difference between the approximated value and the true value tends to $0$ when $\eta:=\limsup_{n} \|\beta_n\|\to 0$.
It is worth noting that the graph depends on $\beta_n$ through $\eta$ in two ways. First, $\eta$ and the upper bound of errors inherit the behavior of normal distributions $\beta_n$.
Second, they also depend on the magnitude of $\beta_n$.
As a results, the graph describes the relationship between  errors and bias $\eta$ varies like normal distributions (at each fixed $\eta$)
with
non-zero
means
(but tending to $0$ as $\eta\to 0$).
The graph on the right
 in Figure \ref{pic-ex1}
shows the convergence rate to $0$ of $c_n$ affects  the convergence of the algorithm.

\begin{figure}[hpt]
\centering
\includegraphics[width=6cm, height=5cm]{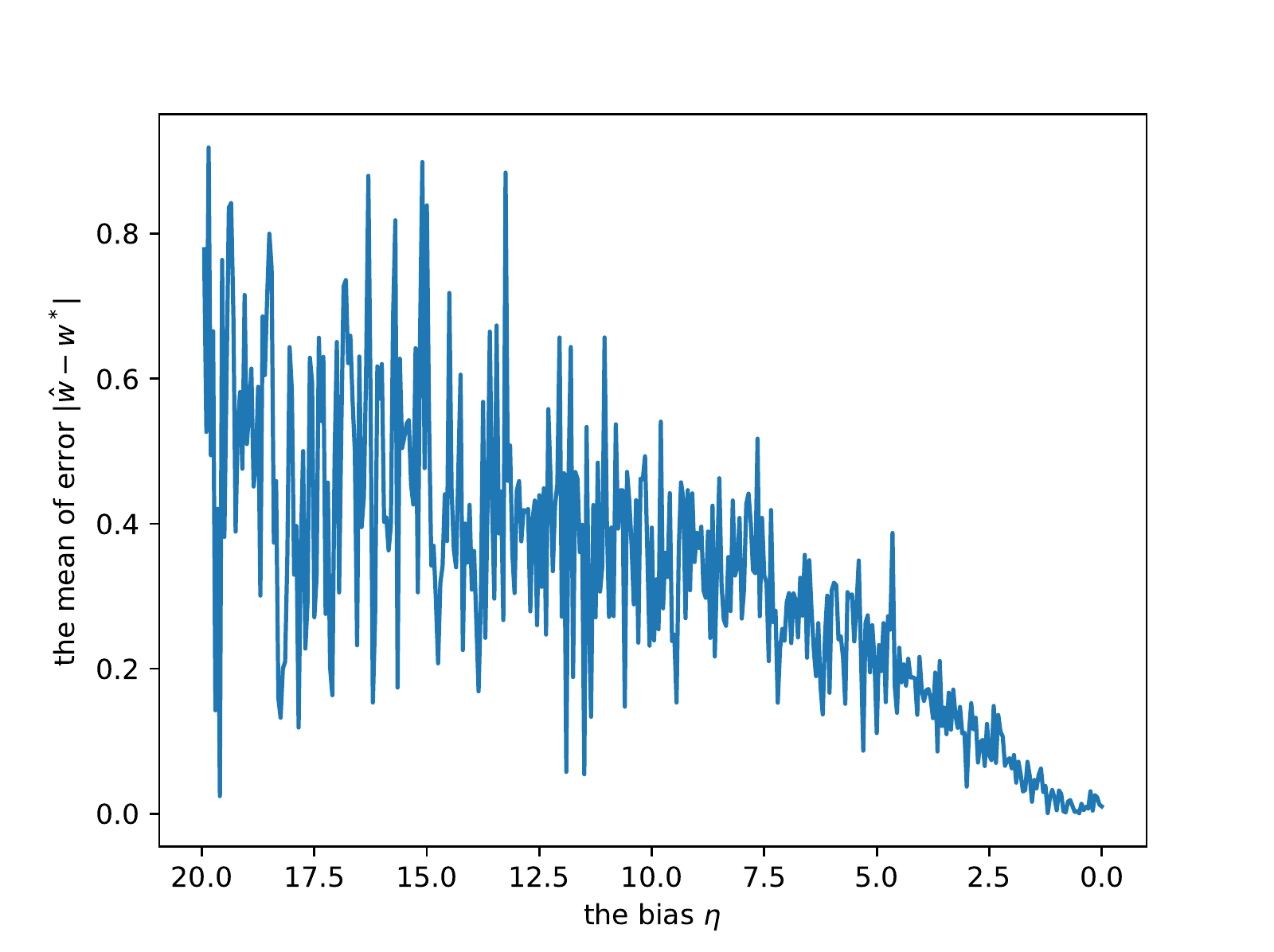}
\includegraphics[width=6cm, height=5cm]{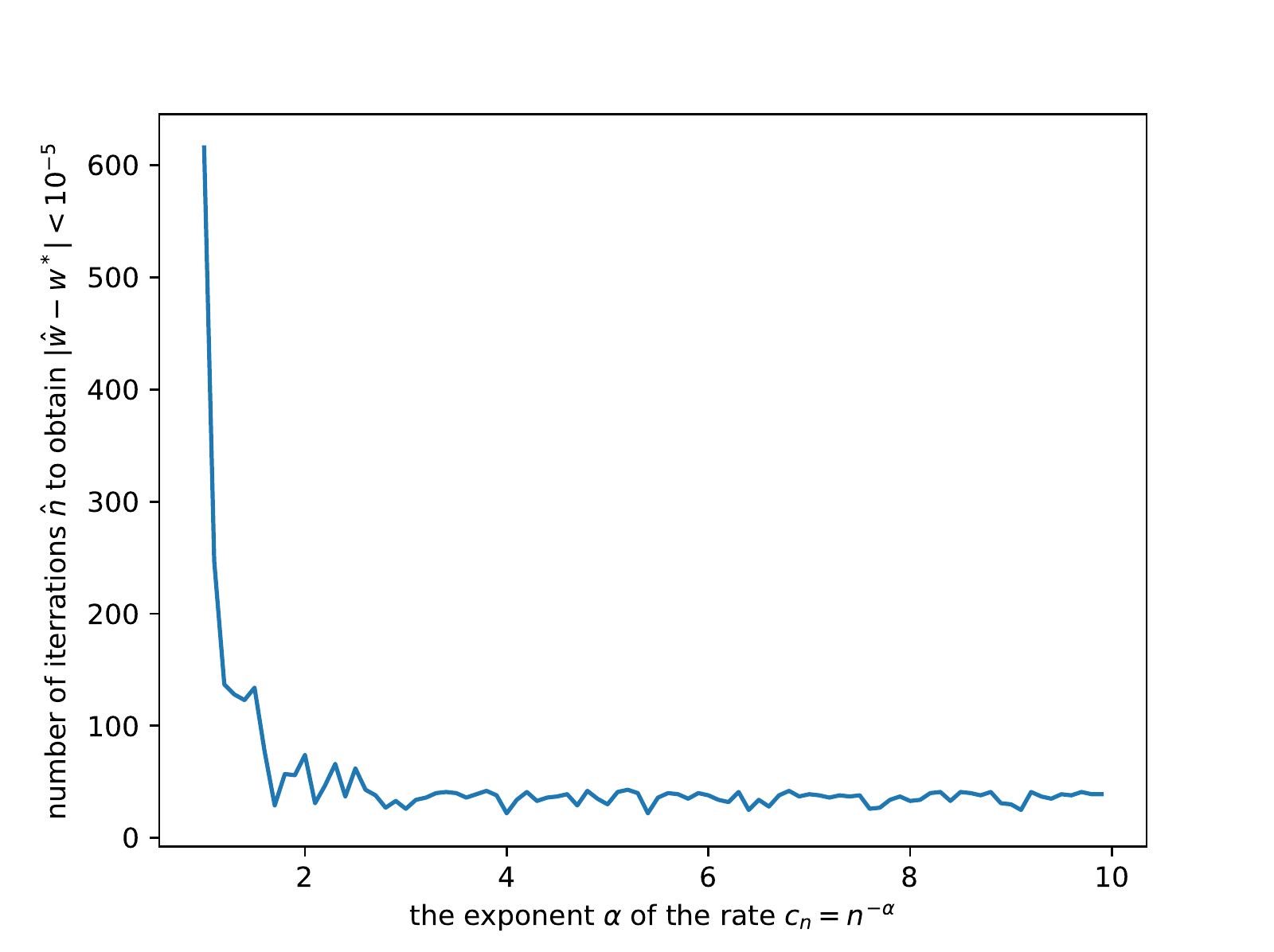}
\caption{Numerical results for Example \ref{ex-num-1}. Left: relation between bias $\eta$ and $|\hat w-w^*|$. Right: relation between number of iterations to obtain $|\hat w-w^*|\leq 10^{-5}$ and the exponent $\alpha$ of $c_n=n^{-\alpha}$ (describing the convergence rate to 0 of unbiased term $\vbe_n$).}\label{pic-ex1}
\end{figure}

\end{example}

\begin{example}\label{ex-num-2}
This example is concerned with using results in Section \ref{sec:app3}.
Consider the following problem (for better visualization, we consider $\w\in\R^2$):
$$
\text{Find minimizer }\w^* \text{ of }\lambda \|\w\|^2+\max\{0,1-\E\w^\top \vec{h}(\xi_n)\}.
$$
Assume that $\{\vec {h}(\xi_n)\}$
is a sequence of independent two-dimensional Gaussian vectors
with mean $(1,2)^\top$ and covariance matrix $I_{2\times 2}$ (two-dimensional identity matrix),
and $\lambda=1$.
A closed-form solution
is $\w^*=(0.2,0.4)^\top$.
We will design an algorithm to locate the optimum with
noise corrupted measurements or observations
$\vec {h}(\xi_n)$ and bias $\beta_n$. Denote $\vec x_n= \vec {h}(\xi_n)$.
Consider the algorithm with step sizes $a_n=1/\sqrt n$,
\begin{equation}\label{num-algo-2}
\begin{aligned}
& \w_{n+1}=\w_n+a_n\left[-2 \w_n+g(\w_n,\vec x_n)\right], \\
\end{aligned}
\end{equation}
where
$
g(\w,\vec x)\in\begin{cases}
\{ \vec 0\}\text{ if } \w^\top \vec x>1,\\
\co\;\{\vec 0,\vec x\}\text{ if }\w^\top \vec x=1,\\
\{\vec x\}\text{ if } \w^\top \vec x<1.
\end{cases}
$
Let $\w_0=(3,5)^\top$, with
1000
replications
(i.e., run algorithm \ref{num-algo-2} $1000$ times),
the numerical results are given in Figure \ref{pic-ex-2}.
Moreover,
the mean of $\hat \w$ is $(0.2+10^{-3}, 0.4+10^{-3})^\top$.

\begin{figure}[hpt]
\centering
\includegraphics[width=6cm, height=5cm]{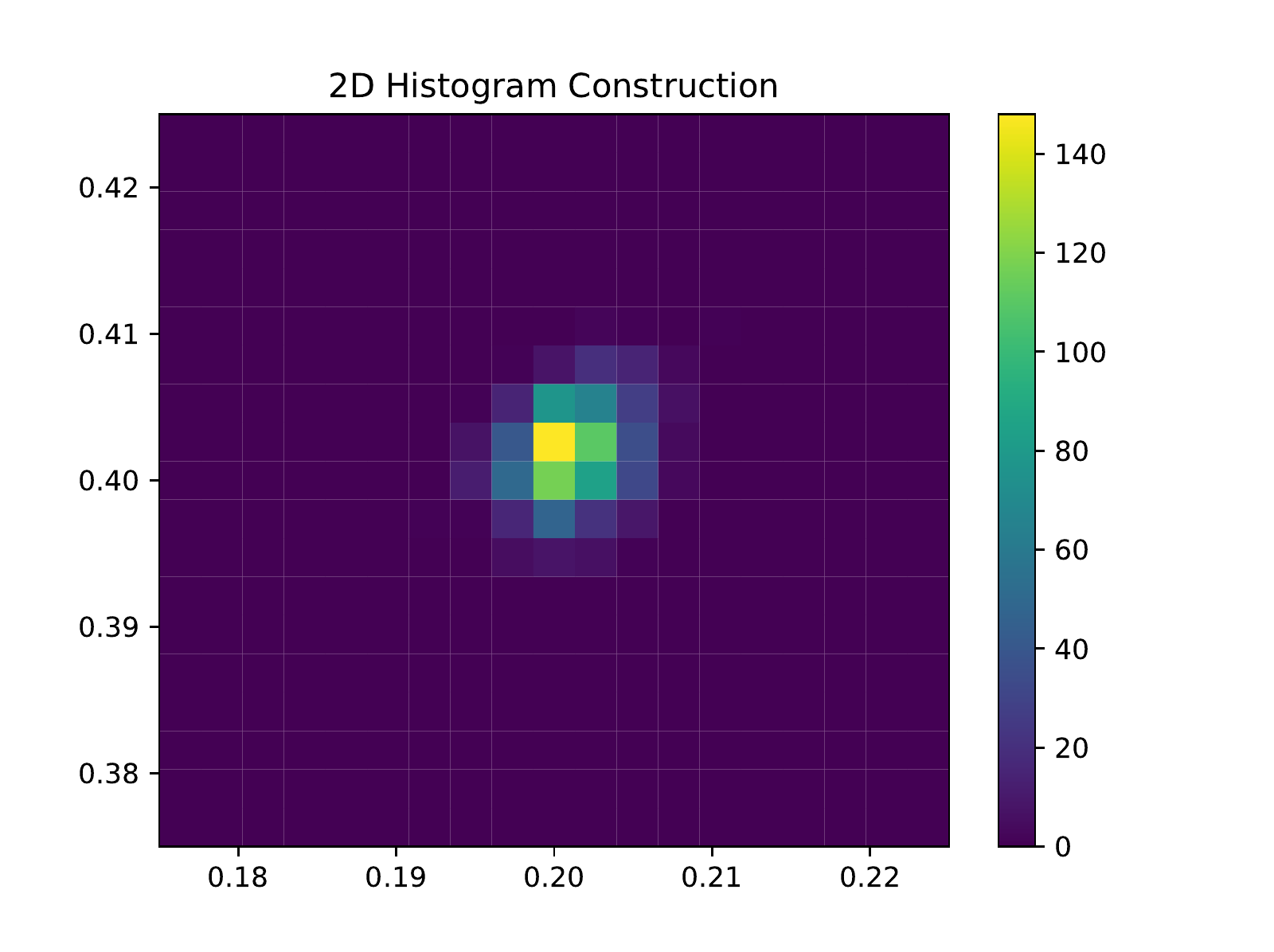}
\includegraphics[width=6cm, height=5cm]{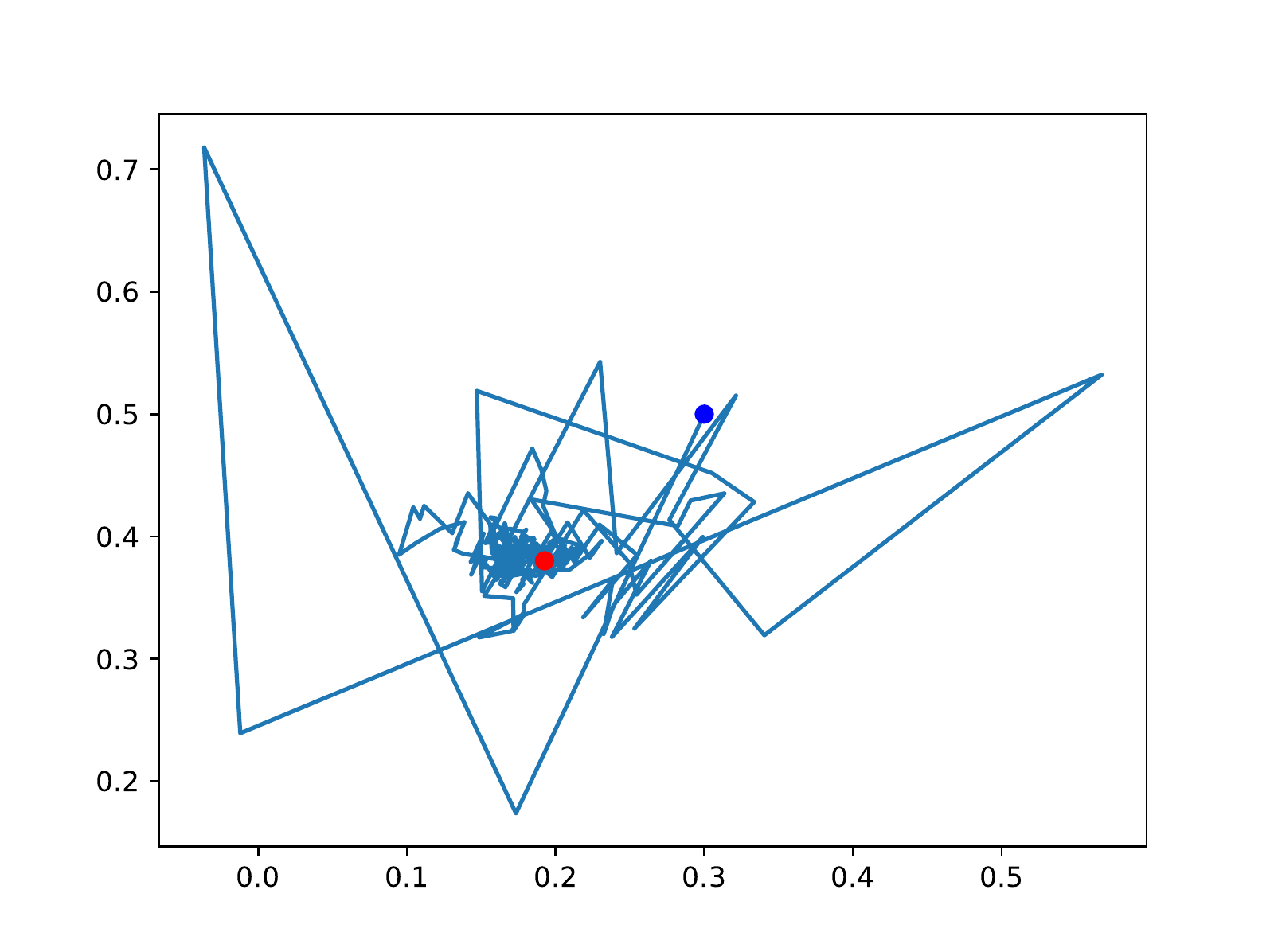}
\caption{Numerical results for Example \ref{ex-num-2}: Left: 2D histogram of $\hat \w$. Right: a trajectory of $\w_n$ (the solid blue and solid red points are the starting and ending points).
}\label{pic-ex-2}
\end{figure}

\end{example}

\begin{example}\label{ex-num-3}
This example is concerned with the results in Section \ref{sec:app4}.
We wish to find $\w^*$ such that $\vec 0\in G(\w^*)$, where
$
G(\w):
=
\big(-w_1+w_2+h(w_2),-w_1-w_2+h(w_1)\big)
$
with $\w=(w_1,w_2)^\top$
and
$
h(w)=\begin{cases}
0\text{ if }w\neq 1,\\
[-1,1]\text{ if }w=1.
\end{cases}
$
The true value is $\w^*=(0,0)^\top$. Consider the stochastic algorithm for the above problem when the observations are corrupted by random disturbances with step sizes $a_n=1/\sqrt n$ and
\begin{equation}\label{num-algo-3}
\begin{aligned}
& \w_{n+1}=\w_n+a_n\left[g(\w_n)+\vbe_n\right],\;g(\w_n)\in G(\w_n),
\end{aligned}
\end{equation}
and
$\vbe_n=(\beta_n^1,\beta_n^2)^\top$ so that $\{\vbe_n\}$
is a sequence of i.i.d. normal random variables with mean $(0,0)^\top$ and covariance being the identity matrix.
We  consider two initial points $\w_0=(1,1)^\top$, (near the minimizer)
and $\w_0=(10,-20)^\top$, (far away from the minimizer).
Running algorithm \ref{num-algo-3} $1000$ times,
we obtain the mean of $\hat\w$ to be $(10^{-3},10^{-2})^\top$
for both cases.
A  histogram and a trajectory of $\{\w_n\}$ (in the case of $\w_0=(1,1)^\top$) are shown in Figure \ref{pic-ex3}.

\begin{figure}[hpt]
\centering
\includegraphics[width=6cm, height=5cm]{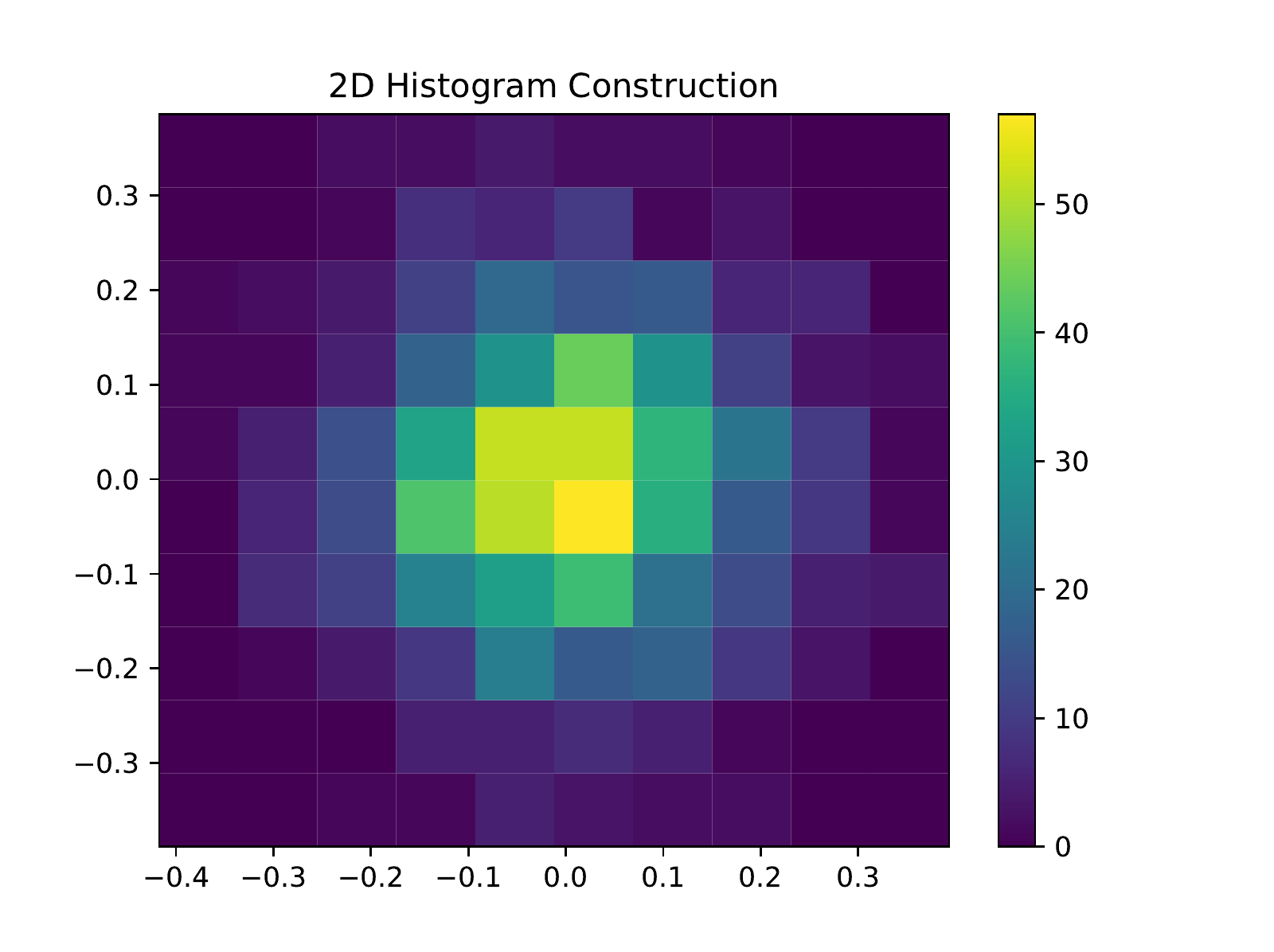}
\includegraphics[width=6.1cm, height=5cm]{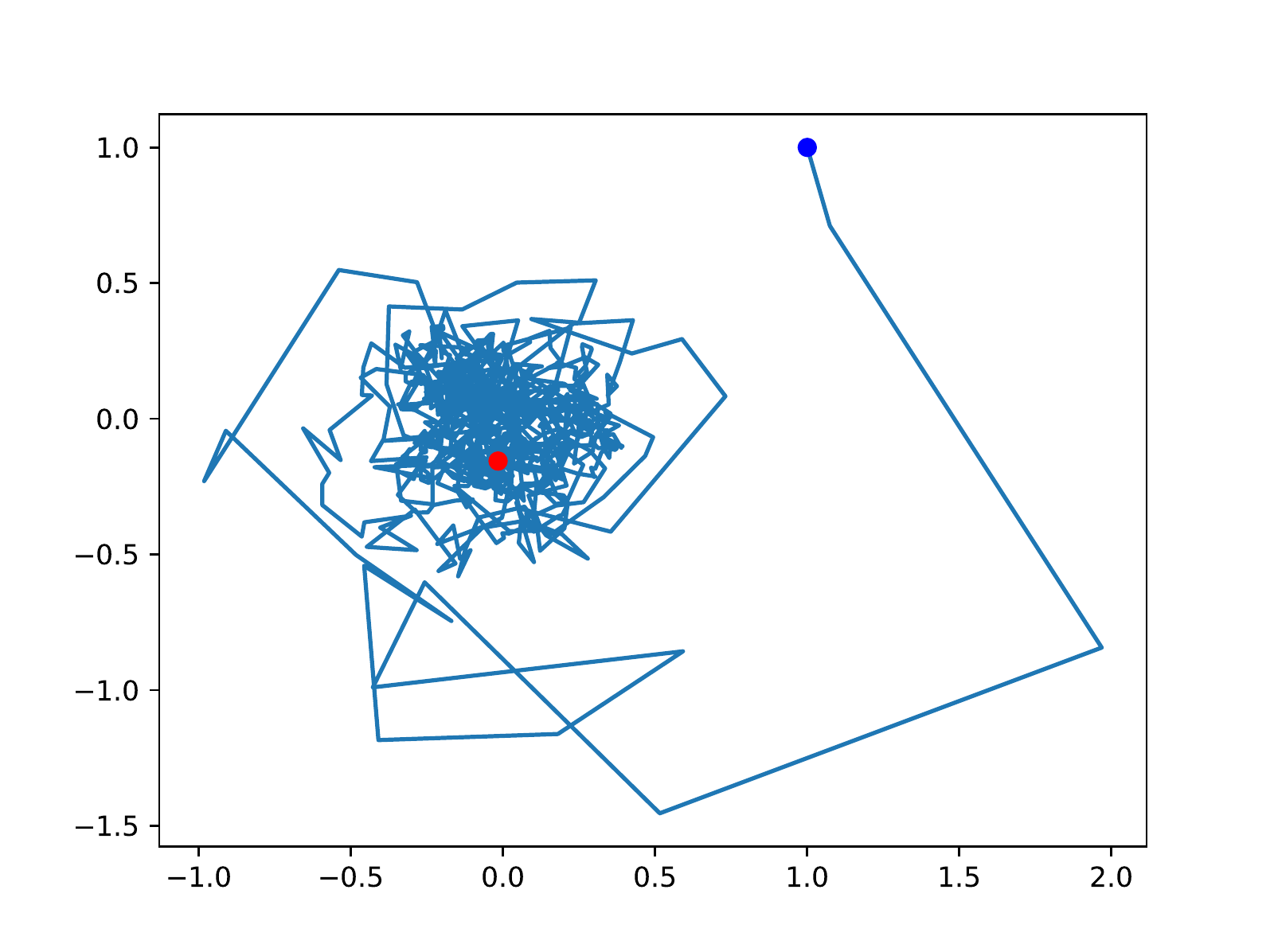}
\caption{Numerical results in Example \ref{ex-num-3}. Left:  2D histogram of $\hat\w$.
Right: a trajectory of $\w_n$ (the solid blue and solid red points are the starting and ending points).}\label{pic-ex3}
\end{figure}

\end{example}

\begin{example}\label{ex-num-4}
This example considers
the comments in Remark \ref{rem1}.
We will give an example to show that if conditions on stability of the zero points are violated, the sequence obtaining by stochastic approximation may not converge to the right points
even if the algorithm starts from one of the optima.
Assume that
$$
\vh(\w)=\begin{cases}
[0,1]\times [-2,1]\text{ if }\w=(2,2)^\top,\\
\{0\}\times [-2,-1]\text{ if }1\leq w_1\leq 2, -1<w_2\leq 2,\w\neq (2,2)^\top, \\
[-2,-1]\times \{0\}\text{ if } -1< w_1\leq 2, -2<w_2\leq -1,\\
\{0\}\times[1,2]\text{ if } -2< w_1\leq -1, -2\leq w_2<-1,\\
[1,2]\times\{0\}\text{ if }-2\leq w_1<1, 1\leq w_2\leq 2,\\
\{(-0.005w_1,-0.005w_2)^\top\}\text{ otherwise},
\end{cases}
$$
and consider the problem:
$
\text{find }\w^*\text{ such that }\vec 0\in \E (\vh(\w^*)+\vbe_n),
$
where $\{\vbe_n\}$ is a sequence of i.i.d. normal random variables with mean $(0,0)^\top$ and covariance being the identity matrix.
The optimum
is given by
$\w^*\in\{(0,0)^\top,(2,2)^\top\}$.
Consider a stochastic approximation algorithm for this problem as follow
\begin{equation}\label{num-algo-4}
\begin{aligned}
& \w_{n+1}=\w_n+\frac 1{\sqrt n}\left[g(\w_n)+\vbe_n\right],\;g(\w_n)\in \vh(\w_n).
\end{aligned}
\end{equation}

We run the algorithm with $\w_0=(2,2)^\top$ for 10 millions iterations and note the points at 1 million, 2 million, $\dots$, 10 million iterations. The algorithm does not converge even the number of iterations is large. In fact,
$\{\w_n\}$ tends to be close to some subset of chain-recurrent points, which are strictly larger than the set of the roots.
The numerical results are shown in Figure \ref{pic-ex4}.

\begin{figure}[hpt]
\centering
\includegraphics[width=6cm, height=5cm]{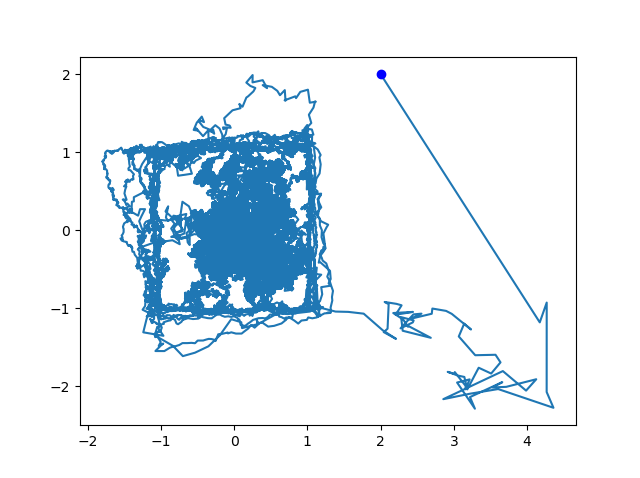}
\includegraphics[width=6cm, height=5cm]{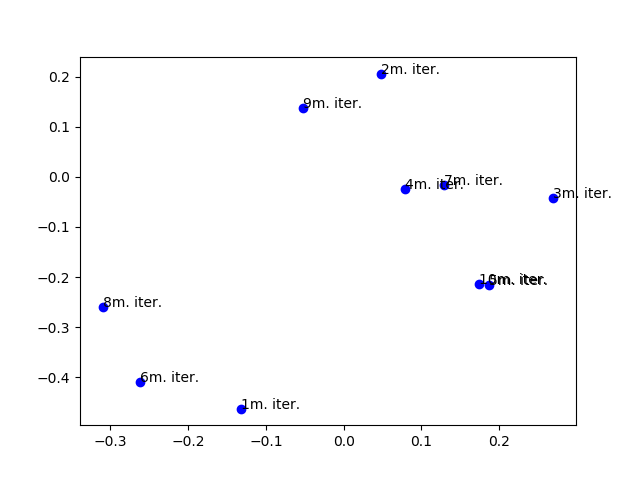}
\caption{Numerical results for Example \ref{ex-num-4}. Left: a trajectory of $\{\w_n\}$, starting from $(2,2)$ (the solid blue point). Right: The points at 1 million, 2 million, $\dots$, 10 million iterations.
}\label{pic-ex4}
\end{figure}

\end{example}

\section{Concluding Remarks}\label{sec:con}
Motivated by a wide variety of
applications,
we considered
stochastic approximation with discontinuous dynamics and set-valued mappings. Unconstrained,
constrained,  and biased algorithms are considered.  The traditional approach in the existing literature cannot be used due to the discontinuity. Another main challenge is that we have to deal with
set-valued mappings.

Under broad conditions, we use the theory of
ODEs with discontinuous right-hand side,  differential inclusions, and set-valued analysis, to overcome the difficulties of
lack of continuity.  Concepts in non-smooth analysis, set-valued dynamic systems, and novel results in stability of differential inclusions enable us to obtain the convergence to the desired optimal points.
The continuation of chain recurrent set of
the limit differential inclusions enables us to obtain desired bounds in biased stochastic approximation.
The rates of convergence are obtained by using
the newly developed concepts in set-valued analysis ($T$-differentiability) and stochastic differential inclusions (weak compactness of the set of solutions).

Then we make use of our results in applications including  Markov decision processes, stochastic sub-gradient descent algorithms, minimizing
$L^1$ regularized loss functions (online Lasso algorithms, among others), and Pegasos algorithms (in SVMs classification).
It is shown that convergence w.p.1 of these stochastic algorithms can be obtained using our results.
It is also demonstrated that
our results can be used to prove convergence in certain cases, which cannot be done otherwise in the existing literature. New insights for analyzing convergence, rates of convergence, and robustness of these algorithms are also obtained.

\appendix
\section{Appendix: Mathematics Preparation}
\label{sec:pre}
\subsection{ODEs with Discontinuous Right-hand Sides and Differential Inclusions}\label{sec:ode}
This section is devoted to
ODEs with discontinuous right-hand sides and differential inclusions.
Consider the
differential equation
\begin{equation}\label{eq-ode}
\dot \vX(t)=f(\vX(t)).
\end{equation}
Given a function $f:\R^d\to\R^d$, define the set-valued function $\mathcal K[f]:\R^d\to2^{\R^d}$,  known as the Krasovskii operator, as follows
$$
\mathcal K[f](\vy)=\cap_{\delta>0}\bar{\text{co}}\;f(B(\vy,\delta)).
$$

\begin{lem}\label{lem-K}
	If $f$ is continuous, then $\mathcal K[f](\vx)=\{f(\vx)\}$.
		If $f,g$ are locally bounded and either $f$ or $g$ is continuous then $\mathcal K[f+g](\vx)=\mathcal K[f](\vx)+\mathcal K[g](\vx)$.
\end{lem}

\begin{proof}
	The first assertion is obvious.
	By \cite[Theorem 1]{Pad87}, we have that
	$
	\mathcal K[g](\vx)=\text{co}\{\lim g(\vx_i)| \vx_i\to \vx\}.
	$
	Using this fact, the lemma can be  proved; some details are  omitted.
\end{proof}

\begin{deff}{\rm  (see \cite{H79})
A function $\boldsymbol{\varphi}:J\to\R^d$ ($J$ is an interval in $\R$) is said to be a Krasovskii
solution to \eqref{eq-ode} if it is absolutely continuous
on each compact subinterval of $J$
and is a solution of the differential inclusion
\begin{equation}\label{eq-Filip}
\dot \vX(t)\in \mathcal K[f](\vX(t)),
\end{equation}
i.e., $\boldsymbol{\varphi}$ satisfies \eqref{eq-Filip} almost every $t\in J$.
Moreover, $\boldsymbol{\varphi}$ is said to be a Carath\'eodory solution if it satisfies the \eqref{eq-ode} for almost every $t\in J$, or equivalently, it satisfies the corresponding integral equation.}
\end{deff}

\begin{deff}\label{def-usc}
{\rm A set-valued mapping $F$ is
upper semicontinuous at a given $\bar \vx$, if for every open set $U$, $F(\bar \vx)\subset U$, there is an open set $V$ such that $\bar \vx\in V$ and $F(\vx)\subset V$ for every $\vx\in V$.}
\end{deff}

Note that if $f$ is a locally bounded function, then $\mathcal K[f](\cdot)$ is upper semicontinuous,
nonempty, compact, and convex.
The following theorem of the existence of Krasovskii solution can be found in \cite{H79}.

\begin{lem}
If $f:\R^d\to\R^d$ is a locally bounded function,
 there exists at least a Krasovskii
 solution of \eqref{eq-Filip} starting from any initial condition.
\end{lem}

\begin{rem}
Some remarks are in order;
for more details, we refer to \cite{H79}.
\item {\rm (i)} For the uniqueness of Krasovskii solution,
we need further conditions for $f(\cdot)$, which can be found in \cite{H79}. The Carath\'eodory solutions are always Krasovskii solutions (if both of them exist), but the converse is not true.
If $f$ is continuous, they are the same.
\item {\rm (ii)}
Consider an example with $f(\cdot)=-\text{sign}(\cdot):\R\to\R$, i.e.,
$f(y)=
\begin{cases}
-1\text{ if }y>0,\\
0\text{ if }y=0,\\
1 \text{ if }y< 0.
\end{cases}$
In this case,
$
\mathcal K[f](y)=
\begin{cases}
\{-1\} \text{ if }y> 0,\\
[-1,1] \text{ if }y=0,\\
\{1\}\text{ if }y<0.
\end{cases}
$
%
%

\end{rem}


Next,
 ODEs with discontinuous right-hand sides are generalized
to
differential inclusions.

\begin{deff}
{\rm Let $F:\R^d\to 2^{\R^d}$ be a set-valued mapping. A solution to the differential inclusion
\begin{equation}\label{di-eq1}
\dot \vX(t)\in F(\vX(t))
\end{equation}
with initial point $\vx\in \R^d$ is an absolutely continuous function $\vX(\cdot):\R\to\R^d$ such that $\vX(0)=\vx$ and satisfies \eqref{di-eq1} for almost every $t\in\R$.}
\end{deff}

The following lemma shows that under Assumption  {\rm\ref{G1}} in our paper, the solutions of differential inclusion exists.

\begin{lem}\label{lem5.2} $($see \cite[Chapter 1 and Chapter 2]{AC84}$)$
Let $F:\R^d\to 2^{\R^d}$ be a set-valued map with values contained in a finite common ball and whose graph is closed.
Then $F$ is upper semicontinuous, and \eqref{di-eq1} admits at least one solution with any initial point.
\end{lem}

\subsection{Non-smooth Analysis: Set-valued Derivative and $\mathcal{U}$-generalized Derivative}\label{sec:non}
In this section, we provide some definitions of generalized derivatives in non-smooth analysis, which will be key  in studying stability of solutions of differential inclusions.

\begin{deff}\label{non-def-1}
{\rm We introduce the following definitions.
\begin{itemize}
\item [\rm{(i)}] (\cite{BC99} or \cite[p. 39]{C83})
A function $V(\cdot):\R^d\to\R$ is said to be regular at $\vx\in \R^d$ if for all $\vec v\in\R^d$, there exists the usual right directional derivative $V'_+(\vx,\vec v)$ and $V'_+(\vx,\vec v)=V^\circ(\vx,\vec v)$;
where
$$
V'_+(\vx,\vec v):=\lim_{t\downarrow 0}\frac{V(\vx+t\vec v)-V(\vx)}{t},
$$
and $V^\circ(\vx,\vec v)$ is the generalized directional derivative defined as
$$
V^\circ(\vx,\vec v):=\limsup_{\vy\to \vx,\; t\downarrow0}\frac{V(\vy+t\vec v)-V(\vy)}{t}.
$$
$V$ is said to be regular if it is regular at every $\vx\in\R^d$. Note that a convex function is not only Lipschitz continuous (in suitable domain), but also regular.

\item [\rm{(ii)}] (see \cite{C83}) The Clarke gradient $\partial V$ of $V$ is defined as
$
\partial V(\vx):=\co\{\lim \nabla V(\vx_i)| \vx_i\to \vx, \;\vx\notin\Omega_V\},
$
where $\Omega_V$ is the set of measure zero with $\nabla V$
being
not defined.
\item [\rm{(iii)}]
$($see \cite{BC99}$)$
The set-valued derivative of a regular function $V$ with respect to $F$ is defined as
$$
\dot{\bar V}^F(\vx)=\{a\in\R| \text{ there is }\vec q\in F(\vx) \text{ such that }\vec p^\top \vec q=a,\;\forall \vec p\in\partial V(\vx)\}.
$$
\item [\rm{(iv)}]
A function $V:\R^d\to\R$ is said to be positive definite if it is continuous, $V(\vec 0)=0$ and there are  continuous increasing functions $\alpha_1$ and $\alpha_2:\R_+\to \R$ with $\alpha_1(0)=\alpha_2(0)=0$ such that
$
\alpha_1(|\vx|)\leq V(\vx)\leq \alpha_2(|\vx|),\;\forall \vx\in\R^d.
$

\end{itemize}}
\end{deff}

The following lemma provides a view of the relationship between the above definitions and the dynamics of solutions of differential inclusions.

\begin{lem} $($see \cite[Lemma $1$]{BC99}$)$
Let $\vX(\cdot)$ be a solution of $\dot \vX(t)\in F(\vX(t))$, and $V:\R^d\to\R$ be a locally Lipschitz continuous and regular function.
Then, $\frac d{dt}V(\vX(t))$ exists almost everywhere and $\frac d{dt}V(\vX(t))\in \dot{\bar V}^F(\vX(t))$ almost everywhere.
\end{lem}

Finally, we recall the following definitions introduced in \cite{KDR17}, which are used in this paper.

\begin{deff} 
	{\rm \begin{itemize}
		\item [\rm{(i)}] Let $\mathcal U:=\{U_i\}_{i=1}^\infty$ be a collection of real-valued Lipschitz regular functions.
		We define
		$
		\tilde F_{\mathcal U}:=\cap_{i=1}^\infty M_{U_i}^F(\vx),
		$
		 where
		$
		M^F_{U_i}:=\{\vec q\in F(\vx) | \text{ there exists }a\in \R\text{ such that }\vec p^\top \vec q=a,\;\forall \vec p\in\partial U_i(\vx)\}.
		$
		If $\mathcal U$ is empty, we define $\tilde F_{\mathcal U}=F(\vx)$.
		$\tilde F_{\mathcal U}$ is called the $\mathcal U$-reduced differential inclusion.
\item [\rm{(ii)}] \label{def-56-11}
The $\mathcal U$-generalized derivative of locally Lipschitz function $V:\R^d\to\R$ with direction $F$, denoted by $\dot{\bar V}_{\mathcal U}$ is defined as
\[
\dot{\bar V}^F_{\mathcal U}(\vx):=
\begin{cases}
\min_{\vec p\in\partial V(\vx)}\max_{\vec q\in \tilde F_{\mathcal U}(\vx)} \vec p^\top \vec q\text{ if $V$ is regular},\\
\max_{\vec p\in\partial V(\vx)}\max_{\vec q\in \tilde F_{\mathcal U}(\vx)} \vec p^\top \vec q\text{ if $V$ is not regular}.
\end{cases}
\]
The $\mathcal U$-generalized derivative is understood to be $-\infty$ if $\tilde F_{\mathcal U}$ is empty.
Such a Lyapunov function $V$ with $\dot{\bar V}^F_{\mathcal U}(\vx)\leq 0,\;\forall \vx$ is called as $\mathcal U$-generalized Lyapunov function.
	\end{itemize}}
	
\end{deff}

\begin{example}\label{1-ex-1}{\rm
To illustrate,
let $F(x)=\mathcal K[f](x):\R\to2^{\R}$, $f(x)=-\text{sign}(x)$, i.e.,
$
F(x)=\begin{cases}
-1\text{ if }x>0,\\
[-1,1]\text{ if }x=0,\\
1\text{ if }x<0,
\end{cases}
$
$U:\R\to\R$, $U(x)=\max\{x,0\}$, $\mathcal U=\{U\}$ and $V(x)=x^2$.
Since $U$ is convex, it is regular. The Clarke gradient of $U$ is given by
$
\partial U(x)=\begin{cases}
1\text{ if }x>0,\\
[0,1]\text{ if }x=0,\\
0\text{ if }x<0.
\end{cases}
$
The reduced inclusion $M_{\mathcal U}^F$ is given by
$
M_{\mathcal U}^F(x)=\begin{cases}
-1\text{ if }x>0,\\
0\text{ if }x=0,\\
1\text{ if }x<0.
\end{cases}
$
Moreover, $V$ is continuously differentiable, $\partial V(x)=2x$.
Hence, the $\mathcal U$-generalized derivative of $V$ in direction $F$ is given by
$$
\begin{aligned}
\dot{\bar V}_{\mathcal U}^F(x)=& \max_{q\in \tilde F_{\{U\}}(x)} \partial V(x)q=\max_{q\in\tilde F_{\{U\}}(x)}2xq
=
\begin{cases}
-2x\text{ if }x>0,\\
0\text{ if }x=0,\\
2x\text{ if }x<0.
\end{cases}
\end{aligned}
$$ }
\end{example}

\subsection{Stability of Differential Inclusions}\label{sec:63-asy}
In this section, we consider the asymptotic stability of solutions of the ODEs with discontinuous right-hand sides and differential inclusions, which contains two parts. The first is stability of Krasovskii solutions and the second is for general differential inclusions.

Let $F:\R^d\to 2^{\R^d}$ be a set-valued mapping such that $F$ is upper semicontinuous  whose values are non-empty, compact, and convex. Consider the differential inclusion
\begin{equation}\label{eq-F}
\dot \vX(t)\in F(\vX(t)).
\end{equation}

\begin{deff}{\rm $($see \cite[Definition 2.1]{C88}$)$
	The differential inclusion \eqref{eq-F} is strongly asymptotically stable $($in Clarke's sense$)$ if there is no solution exhibiting finite time blow-up and the following properties hold.
	\begin{itemize}
		\item[]{\rm (a)} Uniform attraction: for any $r>0$, $R>0$, there is $T=T(R,r)$ such that for any solution $\vX(\cdot)$ of \eqref{eq-F} with $|\vX(0)|<R$ then
		$
		|\vX(t)|\leq r\text{ for all }t\geq T.
		$
		\item[]{\rm (b)} Uniform boundedness: there is a continuous non-increasing function $m:(0,\infty)\to(0,\infty)$ such that for any solution $\vX(\cdot)$ of \eqref{eq-F} with $|\vX(0)|\leq R$ then
		$
		|\vX(t)|\leq m(R)\text{ for all }t\geq 0.
		$
		\item[]{\rm (c)} Lyapunov stability:
		$
		\lim_{R\downarrow 0}m(R)=0.
		$
	\end{itemize}
}
\end{deff}

\begin{deff} {\rm $($Classical Lyapunov stability$)$ $($see \cite[Definition 7.1]{KDR17}$)$
	The differential inclusion $\dot \vX(t)\in F(\vX(t))$ is said to be $($strongly$)$ asymptotically stable at $\vx=\vec 0$ if every solution is stable at $\vx=\vec 0$,
	$[$that is, for any $\eps>0$, there is $\delta>0$ such that if $|\vX(0)|\leq \delta$ then $|\vX(t)|<\eps,\forall t\geq 0$$]$ and there is $c>0$ such that if $|\vX(0)|\leq c$ then $\lim_{t\to\infty} |\vX(t)|=0$.
	Moreover, it is said to be globally asymptotically stable if the constant $c$ can be $\infty$.}
\end{deff}

\begin{prop}$($see \cite{C88}$)$
	The strongly asymptotic stability $($in Clarke's sense$)$ implies the classical asymptotic stability in the Lyapunov sense.
\end{prop}

The following theorem (\cite[Theorem 1.3]{C88}) provides  necessary and sufficient conditions for strongly asymptotic stability of the Karasovskii
solutions of the ODEs with discontinuous right-hand sides (see Section \ref{sec:ode} for definition)
\begin{equation}\label{eq-f11}
\dot \vX(t)=f(\vX(t)).
\end{equation}

\begin{thm}\label{thm-sta}
	Let $f$ be a locally bounded function.
		Then, Krasovskii solutions
	of \eqref{eq-f11} are strongly asymptotically stable if and only if there exists a $\mathcal{C}^\infty$-smooth pair of functions $(V,\hat V_0)$ satisfying
	\begin{itemize}
		\item[] {\rm(1)} $V(\vx)>0$ and $\hat V_0(\vx)>0$ for all $\vx\neq\vec 0$, $V(\vec 0)=0;$
		\item []{\rm(2)} the sublevel sets $\{\vx\in\R^d: V(\vx)\leq l\}$ are bounded for every $l\geq 0;$
		\item[] {\rm (3)} $\limsup_{\vy\to \vx}\langle \nabla V(\vx), f(\vy)\rangle\leq -\hat V_0(\vx),\;\forall \vx\neq\vec 0.$
	\end{itemize}
\end{thm}

\begin{rem}
	Note that in differential inclusions, the uniqueness of solution is not always guaranteed.
		Hence, the term ``strongly" in definitions of stability means that these definitions hold for all solutions. In contrast, ``weak" stability means that there is a solution that is stable. The condition of ``weak asymptotic stability" of Krasovskii solutions can be found in
	\cite{C88}.
\end{rem}

In contrast to Theorem \ref{thm-sta},  sufficient conditions for asymptotic stability of general differential inclusions can be found in \cite{BC99} and references therein.
Recently, these sufficient conditions for differential inclusion $\dot \vX(t)\in F(\vX(t))$ are much improved in \cite{KDR17}.
We state this result in the following theorem.

\begin{thm}\label{thm-sta2} $($\cite[Theorem 7.2]{KDR17}$)$
	If there exists a $\mathcal U$-generalized Lyapunov function $V:\R^d\to\R$ such that
	$
	\dot {\bar V}_{\mathcal U}^{F}(\vx) \leq -\hat V_0(\vx),
	$
	for some positive definite function $\hat V_0$ $($see Definitions {\rm\ref{non-def-1}} and {\rm\ref{def-56-11}$)$},
	then \eqref{eq-F} is $($strongly$)$ asymptotically stable $($in the sense of Lyapunov$)$
	at $\vx=\vec 0$. Furthermore, if
	$\{\vx\in\R^d: V(\vx)\leq l\}$ are compact for all $l>0$ then \eqref{eq-F} is $($strongly$)$ globally asymptotically stability $($in the sense of Lyapunov$)$ at $\vx=\vec 0$.
\end{thm}

\begin{rem}
	Another
	result on stability of differential inclusions using Lyapunov functional method can be found in \cite{BC99}.
	The technique is based on the ``set-valued derivative of a regular function V" with respect to $F$.
	However, using the $\mathcal U$-generalized derivative is shown to be much stronger and more effective; see \cite{KDR17}. Moreover, if $U(\cdot)$ satisfies $\partial U(\cdot)=(1,\dots,1)^\top$, then the $\{U\}$-generalized derivative is the set-valued derivative.
\end{rem}

\subsection{Set-valued Dynamical Systems: Invariant Set, Limit Set, and Chain Recurrence}
\label{sec:dyn}
Consider the differential inclusion
\begin{equation}\label{dyn-1}
\dot \vX(t)\in F(\vX(t)).
\end{equation}
We recall some concepts, which are used in this paper; more details can be found in \cite{B96,BHS05,TD17} and references therein.

\begin{deff}{\rm (see \cite[Section 3]{BHS05})
 Let $\vX(\cdot)$ be
a solution of \eqref{dyn-1}.
The limit set of $\vX(\cdot)$, denoted by $L(\vX)$, is defined as
$
L(\vX)=\cap_{t\geq 0}\bar{\{\vX(s):s\geq t\}}.
$}
\end{deff}

\begin{deff} {\rm (see \cite[Definition V]{BHS05}) A set $A\subset \R^d$ is said to be invariant if for all $\vx\in A$, there exists a solution $\vX(\cdot)$ of \eqref{dyn-1} with $\vX(0)=\vx$
such that $\vX(\R)\subset A$.
}
\end{deff}

\begin{deff} {\rm (see \cite[Definition VI]{BHS05}) Let $A$ be a subset of $\R^d$.
\begin{itemize}
\item $\vx,\vy\in A$ is said to be chain connected in $A$ if for every $\eps>0$ and $T>0$, there exist an integer $n\in\N$, and solutions $\vX_1(\cdot),\dots,\vX_n(\cdot)$ to \eqref{dyn-1}, and real numbers $t_1,\dots,t_n>T$ such that
\begin{itemize}
\item[]{(a)} $\vX_i(s)\in A$ for all $0\leq s\leq t_i$, $i=1,\dots,n$;
\item[]{(b)} $|\vX_i(t_i)-\vX_{i+1}(0)|\leq \eps$ for all $i=1,\dots,n-1$;
\item[]{(c)} $|\vX_1(0)-\vx|\leq \eps$ and $|\vX_n(t_n)-\vy|\leq \eps$.
\end{itemize}
\item $A$ is said to be ``internally chain transitive" of \eqref{dyn-1} if $A$ is compact and $\vx,\vy$ are chain-connected in $A$ for all $\vx,\vy\in A$.
\end{itemize} }
\end{deff}

\begin{deff} {\rm (see \cite{BHS05,KY03,TD17})
$\boldsymbol{\theta}$ is said to be a ``chain-recurrent point" of \eqref{dyn-1}
if
for any $\eps>0$ and $T>0$, there exist an integer $n\in\N$, and solutions $\vX_1(\cdot),\dots,\vX_n(\cdot)$ to \eqref{dyn-1} and real numbers $t_1,\dots,t_n>T$ such that
$$
|\vX_1(0)-\boldsymbol{\theta}|\leq \eps,\quad |\vX_i(t_i)-\vX_{i+1}(0)|\leq \eps\;\forall i=1,\dots,n-1,\quad |\vX_{n}(t_n)-\boldsymbol{\theta}|\leq\eps.
$$
Moreover, we say that $\boldsymbol{\theta}$ is a
 ``chain-recurrent point" in $A$ of \eqref{dyn-1}, if we assume further that
$\vX_i(s)\in A$ for all $0\leq s\leq t_i$, $i=1,\dots,n$.
}\end{deff}

The following lemma (see \cite[Lemma 3.5]{BHS05}) shows the relationship between invariant set and internally chain transitive set.

\begin{lem}
An internally chain transitive set is invariant.
\end{lem}

\subsection{Set-valued Analysis: Continuity and $T$-differentiability}\label{sec:set}
This section reviews definitions and results of set-valued analysis in \cite{AC84,Kis13,Pan11,RW97} and references therein.
Recall that
 $B=\{\vx\in\R^d:|\vx|<1\}$ and $\bar {B}$ is its closure.

\begin{deff} {\rm (see \cite[Chapter 1, Section 1]{AC84} or \cite{RW97})
\begin{itemize}
\item A set-valued mapping $F$ is said to be lower semicontinuous at $\bar \vx$ if for every open set $U$ with $F(\bar \vx)\cap U\neq \emptyset$, there is an open set $V$ such that $\bar \vx\in V$ and $F(\vx)\cap U\neq \emptyset$ for every $\vx\in V$.
\item $F$ is said to be continuous if it is both lower semicontinuous and upper semicontinuous (see Definition \ref{def-usc}).
\end{itemize} }
\end{deff}

\begin{lem}\label{5.4-lem54} $($Criteria on continuity, see \cite[Chapter 2.2]{Kis13}$)$
If a set-valued mapping $F:\R^d\to2^{\R^d}$ has convex and compact values, then $F$ is continuous if and only if for each $\vec p\in\R^d$,
$\sigma(\vec p,F(\vx))$ is continuous $($in $\vx)$, where
$
\sigma(\vec p,A):=\sup\{\vec p^\top \vec a: \vec a\in A\}.
$
\end{lem}

\begin{deff} {\rm (see \cite[Definition 2.1]{Pan11})
		A set-valued mapping $T:\R^d\to2^{\R^d}$ is positively homogeneous if
		$
		T(\vec 0)$ is a cone, and $T(k\vx)=kT(\vx)$ for all $k>0$, $\vx\in\R^d.$
	}
\end{deff}

\begin{deff} {\rm (see \cite[Definition 4.1]{Pan11})\label{54-def-14}
Let $T:\R^d\to2^{\R^d}$ be a positively homogeneous set-valued mapping.
We say $F:\R^d\to2^{\R^d}$ is outer $T$-differentiable at $\vx^*$ if for any $\delta>0$, there exists a neighborhood $V$ of $\vx^*$ such that
\begin{equation}\label{eq-otd}
F(\vx)\subset F(\vx^*)+T(\vx-\vx^*)+\delta |\vx-\vx^*|B\text{ for all }\vx\in V.
\end{equation}
}
\end{deff}

The relationship between $T$-differentiability and others differentiability, and the analysis as well as computation examples of $T$-differentiability can be found in \cite{Pan11}.

\subsection{Stochastic Differential Inclusions}\label{sec:sdi}
Given a set-valued mapping $F:\R^d\to2^{\R^d}$ taking non-empty values, there exists an  $f:\R^d\to\R^d$ such that $f(\vx)\in F(\vx),\;\forall \vx\in \R^d$, such a function $f$ is called a
selector of $F$.
For an $L^2$-continuous (continuous in mean) $\F_t$-nonanticipative stochastic process $(\vX(t))_{0\leq t\leq T}$ and set-valued mapping $F_1:[0,T]\times \R^d\to 2^{\R^d}$, $F_2:[0,T]\times \R^d\to 2^{\R^{d\times m}}$ taking closed (subset) values, we denote $(F_1\circ \vX)(t)(\omega):=F_1(t,\vX(t)(\omega))$, $(F_2\circ \vX)(t)(\omega)=F_2(t,\vX(t)(\omega))$ and denote by
$S(F_1\circ \vX)$, $S(F_2\circ \vX)$ the family of all $\F_t$-nonanticipative selectors of $F_1\circ \vX$ and $F_2\circ \vX$, respectively.
Let $(\vW(t))_{0\leq t\leq T}$ be an $m$-dimensional $\F_t$-Brownian motion and define the following sets
$$
\int_s^t (F_1\circ \vX)(r)dr:=\left\{\int_0^T \1_{[s,t]}(r)f(r)dr: f\in S(F_1\circ \vX)\right\},
$$
$$
\int_s^t (F_2\circ \vX)(r)dr:=\left\{\int_0^T \1_{[s,t]}(r)g(r)d\vW(r): g\in S(F_2\circ \vX)\right\}.
$$
Consider the stochastic differential inclusion
\begin{equation}\label{sdi-eq-1}
d\vX(t)\in F_1(t,\vX(t))dt+F_2(t,\vX(t))d\vW(t).
\end{equation}

\begin{deff} {\rm (see \cite{Kis01,Kis13})
We define the (stochastic) weak solution to \eqref{sdi-eq-1} as a system consisting of a complete filtered probability space $\{\Omega,\F,\{\F_t\},\PP\}$, a continuous $\F_t$-adapted process $(\vX(t))_{0\leq t\leq T}$, and an $\F_t$-Brownian motion $\vW(t)$ satisfying
$$\vX(t)-\vX(s)\in \int_s^t F_1(r,\vX(r))dr+\int_s^t F_2(r,\vX(r))d\vW(r),\quad \forall\; 0\leq s<t\leq T,\quad\text{w.p.1}.$$
Denote by $\X_\mu(F_1,F_2)$ a set of all weak solutions to \eqref{sdi-eq-1} with an initial distribution $\mu$.
It  is called a (stochastically) strong solution (solution for short) if the complete filtered probability space and the Brownian motion have been given.
}
\end{deff}

\begin{lem}\label{5.5-lem56} $($see \cite{Kis01}$)$
Assume that $F_1,F_2$ are measurable and bounded and have convex values, where $F_2$ has convex value in the sense of that $\{g\cdot g^\top: g\in F_2(t,\vx)\}$ is convex for each $(t,\vx)\in [0,T]\times\R^d$ and $F_1(t,\cdot)$, $F_2(t,\cdot)$ are
continuous for fixed $t\in[0,T]$.
Then, for any initial distribution $\mu$, the set $\X_\mu(F_1,F_2)$ is non-empty.
\end{lem}

When convexity is absent, the above results
were studied in \cite{Kis05}.
For more details on stochastic differential inclusions, the reader is referred to \cite{Kis01,Kis05,Kis13} and  references therein.

\subsection{Proof of Proposition \ref{prop:tight}}\label{sec:provetight}
\begin{proof}
	Without loss of generality and for notational simplicity, we  assume that $\vx^*=\vec 0$ and verify the tightness for sequence of $\frac{\vX_n}{\sqrt{a_n}}$.
	To prove this tightness, it suffices to show that for each small $\kappa > 0$, there are finite constants $M_\kappa$ and $C_\kappa$ such that
	\begin{equation}\label{eq-cond-tight}
	\PP \Big(\frac{\vX_n}{\sqrt {a_n}}\geq C_\kappa\Big)\leq \kappa, \text{ for } n \geq M_\kappa.
	\end{equation}
	Let $\eps> 0$ be small. Because $\vX_n \to \vx^* = \vec 0$ w.p.1, for any given
	small $\nu > 0$, there exists an $N_{\nu,\eps}$ such that $|\vX_n| \leq \eps$ for $n \geq N_{\nu,\eps}$ with
	probability $\geq 1 - \nu$. By modifying the processes on a set of probability
	at most $\nu$, one can assume that $|\vX_n| \leq \eps$ for $n \geq N_{\nu,\eps}$ and that all the
	assumptions continue to hold. Denote the modified sequence by $\{\vX_n^{\nu}\}$ and if
	we can show $\{\frac{\vX_n^{\nu}}{\sqrt{a_n}}\}$ is tight for each $\eps > 0$, $\nu > 0$, then the original
	sequence is tight. Hence, for the purposes of the tightness proof and
	by shifting the time origin if needed, it can be supposed without loss of
	generality that $|\vX_n| \leq \eps$ for all $n$ for the original process, where $\eps > 0$ is
	arbitrarily small.
	
	Next, denote by $\E_n$ the conditional expectation on the
	past information up to time $n$ (i.e., the $\sigma$-algebra generated
	by $\{\xi_j : j < n\}$.
	We have that
	\begin{equation}\label{eq-tight1}
	\begin{aligned}
	\E_n\big(&V(\vX_{n+1})-V(\vX_n)\big)\\
	=&a_n\E_n\big(V_{\vx}(\vX_n)(\bar {\vec h}(\vX_n)+\vec b_n(\vX_n))\big)+a_n\E_n\big(V_{\vx}(\vX_n)(\vec h(\vX_n,\xi_n)-\bar{\vec h}(\vX_n))\big)\\
	&+O(a_n^2)\E_n|\vec b_n(\vX_n)+\vec h(\vX_n,\xi_n)|^2 \\
	\leq & a_n\E_n \max \dot{\bar V}^{G+\bar {\bold h}}(\vX_n)+a_nV_{\vx}(\vX_n)\E_n(\vec h(\vX_n,\xi_n)-\bar{\vec h}(\vX_n))+O(a_n^2)(1+V(\vX_n)) \\
	\leq & -\lambda a_n\E_n V(\vX_n)+a_nV_{\vx}(\vX_n)\E_n(\vec h(\vX_n,\xi_n)-\bar{\vec h}(\vX_n))+O(a_n^2)(1+V(\vX_n)).
	\end{aligned}
	\end{equation}
Let
$
V_1(\vx;n):=a_nV_{\vx}(\vx)\E_n(\vec h(\vx,\xi_n)-\bar{\vec h}(\vx)),
$
and define the perturbed Lyapunov function
$
\wdt V(\vx;n):=V(\vx)+V_1(\vx;n).
$
In fact, the idea of perturbed Lyapunov functional method is that the perturbations added are small in terms of order of magnitude, and they lead to desired cancellation of the un-wanted terms in \eqref{eq-tight1}.
Thus, by using the usual computation in the perturbed Lyapunov functional method (see e.g., \cite[Theorem 10.4.2, page 345-346]{KY03}), we can obtain from \eqref{eq-tight1} that
\begin{equation*}
\E_nV(\vX_{n+1})-V(\vX_n)\leq -\lambda_1 a_n V(\vX_n)+O(a_n^2),
\end{equation*}
where $0 < \lambda_1 < \lambda$. By taking $\eps$ small enough, it can be supposed that $\lambda_1$
is arbitrarily close to $\lambda$. Thus, there is a real number $K_1$ such that for all $n \geq 0$
\begin{equation}\label{eq-EVn}
\E V(\vX_{n+1})\leq \prod_{i=1}^n (1-\lambda_1 a_i)\E V(\vX_0)+ K_1\sum_{i=0}^n\prod_{j=i+1}^n (1-\lambda_1a_j)a_i^2.
\end{equation}
Therefore, it is readily seen that to obtain \eqref{eq-cond-tight}, it suffices to prove that the right side of \eqref{eq-EVn} is of the order of $a_n$.
However, this fact can be easily proved by approximating this quantity by an exponential approximation. The detail of this argument can be found in \cite[Section 10, page 342-343]{KY03} and is thus omitted here.
\end{proof}

\end{document}